\documentclass[a4paper]{amsart}

\usepackage[applemac]{inputenc}
\usepackage[T1]{fontenc}
\usepackage[english]{babel}
\usepackage{babelbib}
\usepackage{enumerate}
\usepackage{url}
\usepackage{amssymb}
\usepackage[hypertexnames=false,backref=page,pdftex,
 	pdfpagemode=UseNone,
 	breaklinks=true,
 	extension=pdf,
 	colorlinks=true,
 	linkcolor=blue,
 	citecolor=red,
 	urlcolor=blue,
 ]{hyperref}
\usepackage{color}
\usepackage{tkz-fct}
\usepackage{tikz}
\usepackage{tikz-cd}
\usetikzlibrary{matrix}
\usetikzlibrary{arrows}
\usepackage{graphicx}
\usepackage{mathabx}
\usepackage{cleveref}

\tikzcdset{arrow style=math font}
\usepackage[cal=boondoxo]{mathalfa}

\mathcode`A="7041 \mathcode`B="7042 \mathcode`C="7043 \mathcode`D="7044
\mathcode`E="7045 \mathcode`F="7046 \mathcode`G="7047 \mathcode`H="7048
\mathcode`I="7049 \mathcode`J="704A \mathcode`K="704B \mathcode`L="704C
\mathcode`M="704D \mathcode`N="704E \mathcode`O="704F \mathcode`P="7050
\mathcode`Q="7051 \mathcode`R="7052 \mathcode`S="7053 \mathcode`T="7054
\mathcode`U="7055 \mathcode`V="7056 \mathcode`W="7057 \mathcode`X="7058
\mathcode`Y="7059 \mathcode`Z="705A

\newcommand{\bbA}{\mathbb{A}}

\newcommand{\bbC}{\mathbb{C}}
\newcommand{\bbD}{\mathbb{D}}

\newcommand{\bbN}{\mathbb{N}}

\newcommand{\bbP}{\mathbb{P}}
\newcommand{\bbQ}{\mathbb{Q}}
\newcommand{\bbR}{\mathbb{R}}

\newcommand{\bbV}{\mathbb{V}}

\newcommand{\bbZ}{\mathbb{Z}}

\newcommand{\Gm}{\mathbb{G}_m}

\newcommand{\cA}{\mathcal{A}}

\newcommand{\cH}{\mathcal{H}}

\newcommand{\cM}{\mathcal{M}}

\newcommand{\cO}{\mathcal{O}}

\newcommand{\cV}{\mathcal{V}}

\newcommand{\rH}{\textup{H}}

\newcommand{\frp}{\mathfrak{p}}

\newcommand{\frm}{\mathfrak{m}}

\newcommand{\an}{\textup{an}}
\newcommand{\s}{\textup{s}}
\newcommand{\sft}{\textup{sft}}

\newcommand{\sat}{\textup{sat}}

\newcommand{\dR}{\textup{dR}}

\newcommand{\too}{\longrightarrow}
\newcommand{\into}{\hookrightarrow}
\renewcommand{\phi}{\varphi}
\renewcommand{\epsilon}{\varepsilon}
\renewcommand{\ker}{\Ker}

\newcommand{\iso}{\simeq}

\DeclareMathOperator{\pr}{pr}

\DeclareMathOperator{\Spec}{Spec}
\DeclareMathOperator{\Proj}{Proj}
\DeclareMathOperator{\Spf}{Spf}

\DeclareMathOperator{\Hom}{Hom}

\DeclareMathOperator{\im}{Im}
\DeclareMathOperator{\Ker}{Ker}
\DeclareMathOperator{\Frac}{Frac}

\DeclareMathOperator{\Fit}{Fit}

\DeclareMathOperator{\coker}{Coker}

\DeclareMathOperator{\red}{red}

\DeclareMathOperator{\Pic}{Pic}

\DeclareMathOperator{\id}{id}

\DeclareMathOperator*{\hotimes}{\hat{\otimes}}

\DeclareMathOperator{\codim}{codim}

\renewcommand{\le}{\leqslant}
\renewcommand{\ge}{\geqslant}

\countdef \commentson = 1
\newcommand{\MM}[1]{\ifnum\commentson = 1{\color{blue}{#1}}\fi}

\newcommand{\JP}[1]{\ifnum\commentson = 1{\color{magenta}{#1}}\fi}

\DeclareMathOperator*{\colim}{colim}
\renewcommand{\projlim}{\lim}
\renewcommand{\injlim}{\colim}

\theoremstyle{plain}
\newtheorem{theoremintro}{Theorem}
\newtheorem*{corollaryintro}{Corollary}
\newtheorem{theorem}{Theorem}[section]
\newtheorem{lemma}[theorem]{Lemma}
\newtheorem{proposition}[theorem]{Proposition}
\newtheorem{corollary}[theorem]{Corollary}
\newtheorem{claim}[theorem]{Claim}

\theoremstyle{definition}
\newtheorem{definition}[theorem]{Definition}

\newtheorem*{exampleintro}{Example}

\theoremstyle{remark}
\newtheorem{remark}[theorem]{Remark}

\numberwithin{equation}{section}

\begin{document}
\title{Affine vs. Stein in rigid geometry}

\author{Marco Maculan}
\email{marco.maculan@imj-prg.fr}
\address{Institut de Math\'ematiques de Jussieu, Sorbonne Universit\'e, 4 place Jussieu, F-75252 Paris}

\author{J\'er\^ome Poineau}
\email{jerome.poineau@unicaen.fr}
\address{Laboratoire de math\'ematiques Nicolas Oresme, Universit\'e de Caen Normandie, BP 5186, F-14032 Caen Cedex}


\maketitle

\begin{abstract} We investigate the relationship between affine and Stein varieties in the context of rigid geometry. We show that the two concepts are much more closely related than in complex geometry, \emph{e.g.} they are equivalent for surfaces. This rests on the density of algebraic functions in analytic functions. One key ingredient to prove such a density statement is an extension result for Cartier divisors.
\end{abstract}

\setcounter{section}{-1}

\section{Introduction}

\subsection{Motivation}

Let $K$ be a complete non-trivially valued non-Archimedean field. A variety $X$ over $K$, that is, a separated $K$-scheme of finite type, is Stein if the $K$-analytic space  $X^\an$ attached to $X$ admits a closed $K$-analytic embedding in $\bbA^{n, \an}_K$ for some $n \in \bbN$; see \ref{sec:SteinSpaces}. Needless to say, affine varieties are Stein. The aim of the present note is to investigate to which extent the converse statement holds true. The complex analytic version of the question is a classical problem, whose history (recalled below) consists mainly of negative results. For instance, given an integer $d \ge 2$, there is a $d$-dimensional smooth complex variety that is not affine but whose associated complex manifold is Stein. In stark contrast, and breaking the persisting analogy with complex geometry, we derive several general consequences of the Stein property for non-Archimedean varieties summarized as follows (see \cref{Thm:SteinImpliesQuasiAffine,,Thm:AffineIfFiniteType,,Thm:RestrictionToCurvesIsSurjective}):

\begin{theoremintro} \label{Thm:WhatSteinImpliesIntro} A Stein variety $X$ is quasi-affine and, for any closed subscheme~$Y$ of dimension~$\le 1$, the restriction map $\Gamma(X, \cO_X) \to \Gamma(Y, \cO_Y)$ is surjective. Furthermore, if the ring $\Gamma(X_{\red}, \cO_X)$ is Noetherian or if the $K$-algebra $\Gamma(X, \cO_X)$ is finitely generated, then $X$ is affine.
\end{theoremintro}

Recall that a variety $X$ is quasi-affine if it admits an open embedding into an affine variety. The $K$-algebra of regular functions on an algebraic group is finitely generated \cite[III.3.8]{DemazureGabriel}. Therefore as a consequence we obtain a new proof of the following result previously proved in \cite{UniversalExtension}:

\begin{corollaryintro} \label{Thm:AffineIFFSteinGroupsIntro} An algebraic group is Stein if and only if it is affine.
\end{corollaryintro}

Also, a result of Goodman-Landman (see \cite[proposition 3.13]{GoodmanLandman}, or \cref{Prop:AffineIFFRestrictionToSubvarietySurjective}) permits to derive the equivalence in the case of surfaces (see \cref{Thm:SteinIFFAffineSurfaces}):

\begin{corollaryintro} \label{Thm:AffineIFFSteinSurfacesIntro} An algebraic surface is Stein if and only if it is affine.
\end{corollaryintro}

\subsection{Complex background} As we are about to remind, the complex analogue of the previous statements is false---not to make too fine a pun, the behaviour in non-Archimedean geometry is far more rigid.  Let $M_\dR(E)$ be the moduli space of line bundles endowed with a connection on a complex elliptic curve $E$. The line bundles in question are necessarily of degree $0$, and forgetting the connection gives rise to a surjective map of complex varieties $\pi \colon M_\dR(E) \to \Pic^0(E) = E$. Tensor product of line bundles makes $M_{\dR}(E)$ an algebraic group and $\pi$ a morphism of algebraic groups. Its kernel is identified with the space of differential forms on~$E$:
\[ 0 \too \bbV(\Gamma(E, \Omega^1_E)) \too M_\dR(E) \too E \too 0.\]
The algebraic group $M_\dR(E)$ is \emph{not} affine---the quotient of an affine group by a subgroup is affine, whereas the elliptic curve $E$ is most certainly not. Actually, more is true: it is not even quasi-affine since every algebraic function on $M_{\dR}(E)$ is constant (see \cite[proposition 2.3]{BrionAntiAffine}). From the transcendental perspective, Riemann-Hilbert's correspondence sets up a biholomorphism of $M_\dR(E)^\an$ with the group of characters of the topological fundamental group of the elliptic curve 
\[ \Hom_{\bbZ}(\pi_1(E(\bbC), 0), \bbC^\times) \cong \bbC^\times \times \bbC^\times.\]
In particular, $M_\dR(E)^\an$ admits a closed holomorphic embedding in $\bbC^3$. Summing up, the complex algebraic group $M_\dR(E)$ is not affine, but the complex  Lie group $M_\dR(E)^\an$ is Stein. Since it is of dimension $2$, this shows that the analogues of both corollaries to  \cref{Thm:WhatSteinImpliesIntro} do not hold over the complex numbers.\footnote{Hartshorne attributes this example to Serre (see \cite[Chapter VI, Example 3.2]{AmpleSubvarieties}), but it seems that such a phenomenon was known earlier by Severi and Conforto.}

As remarked by Neeman \cite{Neeman}, a complex quasi-affine variety $X$ is affine if and only if $X^\an$ is Stein and the $\bbC$-algebra $\Gamma(X, \cO_X)$ is finitely generated. The equivalence has been strengthened by Brenner \cite{Brenner} (see \cref{Cor:QuasiAffineSteinNoetherianIsAffine}), replacing the finite generation of $\Gamma(X, \cO_X)$ by being Noetherian. Keeping the notation of the previous paragraph, let $H$ be an ample line bundle on $E$. Consider the $\Gm$-principal bundle $X$ on $M:= M_\dR(E)$ associated with the line bundle $\pi^\ast H^\vee$ dual to the pull-back of $H$ along $\pi \colon M \to E$. Ampleness of $H$ implies that $X$ is quasi-affine. Since $M^\an$ is Stein, the $\bbC^\times$-bundle $X^\an$ is also Stein \cite{MatsushimaMorimoto}. However, the complex variety $X$ cannot be affine, as it would imply that $M$ is. It follows that the ring $\Gamma(X, \cO_X)$ is not Noetherian.\footnote{This can be proved by purely algebraic means \cite[theorem 3.9]{BrionAntiAffine}.}

\subsection{Main results} Over the non-Archimedean field $K$ the technical crux leading to \cref{Thm:WhatSteinImpliesIntro} is the following (see \cref{Thm:DensityAlgebraicSections}):

\begin{theoremintro}\label{Thm:DensityAlgebraicSectionsIntro}Let $F$ be a semi-reflexive coherent sheaf on a variety $X$. Then $\Gamma(X, F)$ is dense in $\Gamma(X^\an, F^\an)$.
\end{theoremintro}

See \ref{sec:StatementsDensityPolarSection} for the definition of the topology on $\Gamma(X, F)$. A series of comments:

\begin{enumerate}
\item Within the framework of Berkovich spaces, the topological space underlying the $K$-analytic space $X^\an$ is locally compact: when $X$ is reduced, the topology on $\Gamma(X^\an, \cO_X^\an)$ is that of uniform convergence on compact subsets; when $X$ is not reduced, as in complex analysis, the topology is finer. 

\item An $\cO_X$-module $F$ is \emph{semi-reflexive} if the natural homomorphism $F \to F^{\vee \vee}$ is injective. Therefore \cref{Thm:DensityAlgebraicSectionsIntro} applies in particular to vector bundles and coherent sheaves of ideals. However, we ignore whether it stays true for any coherent $\cO_X$-module.

\item  The techniques involved in the proof are fairly general and permit to carry on the proof when $K$ is replaced by an affinoid $K$-algebra $A$ (in the sense of Berkovich) and $X$ by a separated $A$-scheme of finite type.

\item Let $X$ be a variety and $F$ a coherent $\cO_X$-module. If $X$ is affine, then the density of $\Gamma(X, F)$ in $\Gamma(X^\an, F^\an)$ is rather easy to prove. When $X$ is proper, the non-Archimedean GAGA theorem implies that the natural map $\Gamma(X, F) \to \Gamma(X^\an, F^\an)$ is an isomorphism. Therefore, \cref{Thm:DensityAlgebraicSectionsIntro} is most interesting when $X$ belongs to neither of these two classes, and especially when $X$ does not admit a proper morphism onto an affine variety.

\item \Cref{Thm:DensityAlgebraicSectionsIntro} has little to do with algebraic varieties: the only property we use is that they can be ``compactified''. The general statement concerns the complement of an effective Cartier divisor in a proper $K$-analytic space (see \cref{Thm:DensityMeromorphicSectionsAnalytic}) or,  even more generally, holomorphically convex (see \cref{Thm:ComplementDivisorHolomorphicallyConvex}).

\item The argument for deducing \cref{Thm:WhatSteinImpliesIntro} from \cref{Thm:DensityAlgebraicSectionsIntro} is not exclusive to non-Archimedean analysis and can be adapted over the complex numbers to yield the following result: if $X$ is a complex variety such that $\Gamma(X, \cO_X)$ is dense in $\Gamma(X^\an, \cO_X^\an)$ and $X^\an$ is Stein, then it satisfies the conclusions of \cref{Thm:WhatSteinImpliesIntro}.

\item Consider the moduli space $M_{\dR}(E)$ of line bundles endowed with a connection on a complex elliptic curve $E$. Then, all of the algebraic functions on $M_\dR(E)$ are constant (see for instance \cite[proposition 2.3]{BrionAntiAffine}), while $M_\dR(E)$ is biholomorphic to $\bbC^\times \times \bbC^\times$. Therefore, the complex analogue of \cref{Thm:DensityAlgebraicSectionsIntro} is false.
\end{enumerate}

\Cref{Thm:DensityAlgebraicSectionsIntro} marks quite a difference between complex and non-Archimedean analysis: roughly speaking, there are fewer rigid analytic functions with essential singularities than holomorphic ones.\footnote{Resembling phenomena were observed by Cherry throughout his work in non-Archimedean Nevanlinna theory \cite{CherryThesis}, \cite{CherryAbelian}, \cite{CherryZhuan}.} Mirroring this discrepancy, its proof has to rely on some technique unavailable in the complex framework: with no originality, it is the involvement of models over the ring of integers $R$ of $K$. To state the result let $\varpi \in R \smallsetminus \{ 0 \}$ be a topological nilpotent element and for an $R$-scheme $S$ let $S_n$ be the closed subscheme $S \times_R (R/ \varpi^n R)$ of $S$. The density result lying at the core of the method is the following (see \cref{Thm:DensityMeromorphicSectionsModel}):

\begin{theoremintro}\label{Thm:DensityAlgebraicSectionsModelIntro} Let $D$ be an effective divisor in a proper and flat $R$-scheme $X$. If $D$ is flat over $R$, then for any semi-reflexive coherent $\cO_X$-module $F$  the natural map 
\[ \Gamma(X \smallsetminus D, F) \too \projlim_{n \in \bbN} \Gamma(X_n \smallsetminus D_n, F), \]
is injective, and has a dense image.
\end{theoremintro}

Let us relate this to \cref{Thm:DensityAlgebraicSectionsIntro}. Let $\hat{U}$ be the formal $R$-scheme obtained as formal completion of $U = X \smallsetminus D$ along its special fiber and $i \colon \hat{U} \into U$ the natural morphism of locally ringed spaces. Then, 
\[  \Gamma(\hat{U}, i^\ast F) = \projlim_{n \in \bbN} \Gamma(U_n, F). \]
Let $U_K$ be the generic fiber of $U$ and $F_K$ the restriction of $F$ to $U_K$. With this notation, the $K$-vector space $\Gamma(\hat{U}, i^\ast F) \otimes_R K$ is the space of global sections $\Gamma(\hat{U}_\eta, F_K^\an)$ of the coherent sheaf $F_K^\an$ on a certain compact subset $\hat{U}_\eta$ of $U_K^\an$, called the \emph{Raynaud generic fiber} of the admissible formal $R$-scheme $\hat{U}$. \Cref{Thm:DensityAlgebraicSectionsModelIntro} states that any $K$-analytic section $f \in \Gamma(\hat{U}_\eta, F_K^\an)$ can be approximated globally on the compact subset $\hat{U}_\eta$ by algebraic sections $g \in \Gamma(U_K, F_K)$.

\begin{exampleintro} \label{Ex:FailureOfDensity} In general, \cref{Thm:DensityAlgebraicSectionsModelIntro} fails to be true when flatness of $D$ is discarded. Suppose for simplicity $R$ discretely valued and let $k$ be the residue field of $R$. Consider the blow-up $\pi \colon X \to \bbP^1_R$  along a $k$-rational point of the special fiber. Let $D$ be the exceptional divisor. The completion of the open subset $U:= X \smallsetminus D$ along its special fiber is isomorphic to the completion of the affine line $\bbA^1_R$ along its special fiber. Thus 
\[ \projlim_{n \in \bbN} \Gamma(U_n, \cO_X) \cong R \{ z \} := \projlim_{n \in \bbN} R/\varpi^n R[z],\]
where $z$ is the coordinate function on $\bbA^1_R$. However the generic fiber of $U$ is $\bbP^1_K$ so by flat base change  $\Gamma(U, \cO_X) \into \Gamma(\bbP^1_K, \cO_{\bbP^1_K}) = K$. In particular, the image of
\[ R=  \Gamma(U, \cO_X) \too \projlim_{n \in \bbN} \Gamma(U_n, \cO_X) \cong R \{ z \}, \]
is not dense.
\end{exampleintro}

Most of the comments for \cref{Thm:DensityAlgebraicSectionsIntro} have their counterpart for \cref{Thm:DensityAlgebraicSectionsModelIntro}:

\begin{enumerate}
\item The topology on $\projlim_{n \in \bbN} \Gamma(U_n, F)$ is the prodiscrete one.

\item Technically speaking, it is the flatness of $F$ and of $F_{\rvert D}$ that is used during the proof rather than the semi-reflexivity of $F$ and the flatness of $D$. However, in the deduction of \cref{Thm:DensityAlgebraicSectionsIntro}, we fail to see how to extend an arbitrary coherent sheaf on the generic fiber to one on the model whose restriction to $D$ is flat over $R$, while we manage to do this under the semi-reflexivity assumption.

\item The argument to prove \cref{Thm:DensityAlgebraicSectionsModelIntro} goes through when $R$ is replaced by a flat $R$-algebra of finite type.

\item When $U$ is affine, \cref{Thm:DensityAlgebraicSectionsModelIntro} follows at once from the surjectivity of the restriction map $\Gamma(U, F) \to \Gamma(U_n, F)$, in turn implied by the vanishing of higher cohomology. A slight generalization of the previous argument involving the formal GAGA theorem yields \cref{Thm:DensityAlgebraicSectionsModelIntro} when $U$ admits a proper morphism onto an affine $R$-scheme of finite type. Therefore, \cref{Thm:DensityAlgebraicSectionsModelIntro} is more interesting when $U$ does not fall in the latter category.

\item The natural setting for \cref{Thm:DensityAlgebraicSectionsModelIntro} is that of a proper formal $R$-scheme $X$ together with an effective Cartier divisor $D \subset X$. The result then states the density of sections over $X \smallsetminus D$ having polar singularities along $D$ among all the sections over $X \smallsetminus D$ (see \cref{Thm:DensityMeromorphicSectionsModel}).
\end{enumerate}

For \cref{Thm:DensityAlgebraicSectionsModelIntro} to be of any use, given a variety $U_K$ over $K$, one has to be able to produce a model $U$ of $U_K$ which is the complement of an effective Cartier divisor $D$ flat over $R$ in a proper flat $R$-scheme $X$ (see \cref{Thm:ModelOfCartierDivisors}):

\begin{theoremintro} \label{Thm:ModelOfCartierDivisorsIntro}  Let $D_K$ be an effective Cartier divisor in a be a proper variety $X_K$. Then, there are an effective Cartier divisor $D$  in a proper flat $R$-scheme $X$ and an isomorphism of $K$-schemes $X_K \cong X \times_R K$ inducing an isomorphism $D_K  \cong D \times_R K $.
\end{theoremintro}

In order to construct such an $X$ and $D$, the first reflex is to take any proper flat model  $X$ of $X_K$ and the Zariski closure $D$ of $D_K$  in $X$. The closed subscheme $D$ is without doubt flat over $R$ but may fail to be a Cartier divisor. To repair such an issue, one considers the blow-up $X'$ of $X$ along $D$. The inverse image $D'$ of $D$ in $X'$ is thus Cartier, but will not be flat over $R$ anymore in general. Looking at the Zariski closure in $X'$ of the generic fiber of $D'$ brings us back to square one.  Not to be caught in a vicious circle, one has to choose carefully the blow-up of $X$ to perform. Namely, if $I$ is the sheaf of ideals of $\cO_X$ defining the closed subscheme $D$, then we consider the blow-up $\pi \colon X' \to X$ of $X$ along the ideal $I + \varpi^n \cO_X$, for some $n \in \bbN$ big enough. The strict transform $D'$ of $D$ in $X'$ is shown to be the effective Cartier divisor $\pi^\ast D - E$, where $E$ is the exceptional divisor and the minus sign stands for the difference of effective Cartier divisors (\emph{cf.} \cref{Lemma:TechnicalBlowingUpLemma}). Gabber communicated us an alternative proof of \cref{Thm:ModelOfCartierDivisorsIntro}. The same procedure allows to prove the following geometric variant, which might be of independent interest (see \cref{Thm:ExtendingCartierSchemeVersion}):

\begin{theoremintro} \label{Thm:ExtendingCartierSchemeVersionIntro} Let $D$ be a closed subscheme of a Noetherian scheme $X$. Let $U \subset X$  be an open subset such that 
\begin{enumerate}
\item $D \cap U$ is scheme-theoretically dense in $D$;
\item $D \cap U$ is an effective Cartier divisor in $U$.
\end{enumerate}
Then there is a $U$-admissible blow-up $\pi \colon X' \to X$ such that the strict transform of~$D$ in $X'$ is an effective Cartier divisor.
\end{theoremintro}

 With \cref{Thm:ModelOfCartierDivisorsIntro} under the belt, the proof of \cref{Thm:DensityAlgebraicSectionsIntro} is achieved by arbitrarily enlarging the compact subset obtained as Raynaud's generic fiber (of the formal completion) of the model. This is a routine operation (\emph{cf.} \cref{Lemma:ExhaustionByBlowUp}), except for the extra care one has to employ not to lose the flatness over $R$ of the Cartier divisor in question (\emph{cf.} \cref{Prop:StrictTransformBlowUpIntersectionTwoCartierFormal}).

\subsection*{Acknowledgements} We thank O. Gabber for communicating an alternative proof of \cref{Thm:ModelOfCartierDivisorsIntro} and  J.-B. Bost and F. Charles for sharing an earlier version of their manuscript \cite{BostCharles}. We thank the referee for his comments permitting to improve the paper in many points. M.M. was supported by ANR-18-CE40-0017. J.P. acknowledges support by the ERC project TOSSIBERG, grant agreement 637027, and the Deutsche Forschungsgemeinschaft, through TRR 326 \emph{Geometry and Arithmetic of Uniformized Structures}, project number 444845124.

\section{Extending Cartier divisors: the algebraic case} \label{sec:ExtensionCartierDivisorsAlgebraic}

\subsection{Statements} \label{sec:CartierDivisor} We start by recalling some basic notions concerning Cartier divisors in a scheme $X$. For a closed subscheme $Z \subset X$ let $I_Z$ denote the kernel of the restriction map $\cO_X \to \cO_Z$.  An \emph{effective Cartier divisor} $D \subset X$ is a closed subscheme whose defining sheaf of ideals $I_D$ is invertible. Let $\cO_X(D)$ denote the dual of $I_D$. The homomorphism $s_D \colon \cO_X \to \cO_X(D)$ dual to the inclusion $I_D \to \cO_X$ is called the \emph{canonical section} associated with $D$.  For effective Cartier divisors $D, D' \subset X$, the Cartier divisor $D + D'$ is the closed subscheme with sheaf of ideals $I_D \otimes_{\cO_X} I_{D'}$ whence the identity $\cO_X(D + D') = \cO_X(D) \otimes_{\cO_X} \cO_X(D')$. If the closed immersion $D \to X$ factors through $D'$, then $I_{D'} \subset I_{D}$ and the sheaf of ideals $I_{D'} \otimes_{\cO_X} \cO_X(D)$ defines the Cartier divisor $D ' - D$. A closed subscheme $Z$ of a scheme $X$ is \emph{finitely presented} if the closed immersion $i \colon Z \to X$ is of finite presentation. Let $U \subset X$ be an open subscheme. A morphism $X' \to X$ is called a \emph{$U$-admissible blow-up} if there exists a finitely presented closed subscheme $Z \subset X \smallsetminus U$ such that $X'$ is isomorphic to the blow-up of $X$ along $Z$.

\begin{theorem} \label{Thm:ExtendingCartierSchemeVersion} Let $D$ be a closed subscheme of a Noetherian scheme $X$. Let $U \subset X$  be an open subset such that 
\begin{enumerate}
\item $D \cap U$ is scheme-theoretically dense in $D$;
\item $D \cap U$ is an effective Cartier divisor in $U$.
\end{enumerate}
Then there is a $U$-admissible blow-up $\pi \colon X' \to X$ such that the strict transform of~$D$ in $X'$ is an effective Cartier divisor.
\end{theorem}

Over the ring of integers $R$ of a field $K$ with a non-trivial non-Archimedean absolute value we can get round of the Noetherian hypothesis and obtain the following generalization of \cref{Thm:ModelOfCartierDivisorsIntro}:

\begin{theorem} \label{Thm:ModelOfCartierDivisors}  Let $Y$ be a finitely presented affine $R$-scheme, $f_K \colon X_K \to Y \times_R K$ a proper morphism of $K$-schemes and  $D_K \subset X_K$ an effective Cartier divisor. 

Then, there are a flat $R$-scheme $X$, a proper morphism of $R$-schemes $f \colon X \to Y$, an effective Cartier divisor $D \subset X$ flat over $R$, and an isomorphism $X_K \cong X \times_R K$ of $K$-schemes  through which $D_K$ is identified with $D \times_R K$ and $f_K$ with $f \times \id_K$.
\end{theorem}

One recovers \cref{Thm:ModelOfCartierDivisorsIntro} in the introduction by taking $Y = \Spec R$. Notice that the hypothesis of $Y$ being finitely presented over $R$ is easily met. Indeed, given a finitely generated $K$-algebra $A_K$ and a surjection $\phi \colon K [ t_1, \dots, t_n ] \to A_K$ of $K$-algebras, the ideal the ideal $I := \ker \phi \cap R [ t_1, \dots, t_n ]$ is finitely generated by \cref{Lemma:FinitePresentationAndFlatness}. Thus the $R$-algebra $A := R [ t_1, \dots, t_n ] / I$ is flat, finitely presented, and such that $A \otimes_R K$ is isomorphic to $A_K$.

\subsection{Proofs} The remainder of this section is devoted to the proof of \cref{Thm:ExtendingCartierSchemeVersion,,Thm:ModelOfCartierDivisors}. We begin with the following technical result:

\begin{lemma} \label{Lemma:ControlDenominators} Let $X$ be a Noetherian scheme and $\phi \colon F \to G$ a homomorphism of coherent $\cO_X$-modules such that $\phi_{\rvert U} \colon F_{\rvert U} \to G_{\rvert U}$ is surjective on the complement  $U$ of an effective Cartier divisor  $Z \subset X$. 

Then there is an integer $N \ge 1$ such that, for all $n \ge N$, the image of $G \to G(nZ)$ is contained in the image of $\phi \otimes \id \colon F(nZ) \to G(nZ)$.
\end{lemma}

\begin{proof} By quasi-compactness of $X$, one may suppose that $X = \Spec A$ is affine, and $Z$ is a principal Cartier divisor with equation $a = 0$ for a nonzerodivisor $a \in A$. The $A$-modules $\Gamma(X, F)$, $\Gamma(X, G)$ are finitely generated and $U$ is the principal open subset $D(a) = \Spec A_a$. Let $g_1, \dots, g_r$ be generators of $\Gamma(X, G)$. The hypothesis of $\phi_{\rvert U} \colon F_{\rvert U} \to G_{\rvert U}$ being surjective implies that, for $i = 1, \dots, r$ there are non-negative integers $n_i, m_i \in \bbN$ and $f_i \in \Gamma(X, F)$ such that $a^{m_i}( a^{n_i} g - \phi(f_i)) = 0$. The integer $N = \max \{ n_1 + m_1, \dots, n_r + m_r \}$ does the job.
\end{proof}

The key result at the core of the method is the following:

\begin{lemma} \label{Lemma:TechnicalBlowingUpLemma} Let $D$ be a finitely presented closed subscheme of a scheme $X$, $Z \subset X$ an effective Cartier divisor and $U = X \smallsetminus Z$ such that
\begin{enumerate}
\item $D \cap U$ is scheme-theoretically dense in $D$;
\item  $D$ is contained in an effective Cartier divisor $D_0$ of $X$. 
\end{enumerate}
Assume there is an integer $n \ge 1$ such that the image of $I_{D} \to I_{D}(nZ)$ is contained in the image of $I_{D_0}(nZ) \to I_{D}(nZ)$. Let $\pi \colon X' \to X$ be the blow-up of $X$ along the scheme-theoretic intersection of $D_0$ and $nZ$, and $E \subset X'$ be the exceptional divisor. 

Then, the strict transform of $D$ in $X'$ is the effective Cartier divisor $\pi^{-1}(D_0) - E$.
\end{lemma}

In particular, there is a $U$-admissible blow-up  $\pi \colon X' \to X$ such that the strict transform of $D$ is an effective Cartier divisor.

\begin{proof} The question is local on $X$, thus one may assume that $X = \Spec A$ is affine and that $D_0$, $Z$ are principal Cartier divisors. Let $f_0 \in A$ be a generator of the ideal of $D_0$, and let $g \in A$ be a generator of the ideal of $Z$. Since $D$ is finitely presented, there are $f_1, \dots, f_r$ such that $f_0, \dots, f_r$ generate the ideal of $D$. With this notation, note that the hypothesis implies that the ideals $f_0 A_g$ and $f_0 A_g + \cdots + f_r A_g$ coincide. Moreover, the blow-up $X'$ is covered by the following affine open subsets:
\begin{align*}
X'_0 &= \Spec A[ f_0 / g^n], & A[ f_0 / g^n] &= (A[t] / (f_0 - t g^n)) / (g\textup{-torsion}),\\
X'_1 &= \Spec A[ g^n / f_0], & A[ g^n / f_0] &= (A[u] / (g^n - u f_0)) / (f_0\textup{-torsion}).
\end{align*}

\begin{claim}
The strict transform $D'$ of $D$ in $X'$ does not meet $X'_1$. 
\end{claim} 

\begin{proof}[Proof of the claim]
The strict transform of $D$ is the scheme-theoretic closure in $X'$ of $\pi^{-1}(D \smallsetminus Z)$. Now, the function $g^n$ is invertible on $X'_1 \smallsetminus \pi^{-1}(Z)$. The equality $g^n = uf_0$ on $X'_1$ implies that $f_0$ is also invertible on $X'_1 \smallsetminus \pi^{-1}(Z)$. Therefore,
\[ \pi^{-1}(D \smallsetminus Z) \cap X'_1 = \pi^{-1}(D) \cap (X'_1 \smallsetminus \pi^{-1}(Z)) = \emptyset. \]
It follows that $D'$ does not meet $X'_1$.
\end{proof}

\begin{claim} The strict transform $D'$, seen as closed subscheme of $X'_0$, is the Cartier divisor with equation $t =0$.
\end{claim}

\begin{proof}[Proof of the claim] Let $A [f_0/g^n]_{g}$ denote the localization of $A [f_0/g^n]$ with respect to $g$: via the ring homomorphism $A \to A[f_0/g^n]$, it is isomorphic to the localization of $A$ with respect to $g$. The strict transform $D'$ is the scheme-theoretic closure of $\pi^{-1}(D \smallsetminus Z)$ in $X'_0$. Therefore, the ideal $I_{D'} \subset A[f_0/g^n]$ defining $D'$ is
\[ \ker (A[f_0/g^n] \to A[f_0/g^n]_{g} / (f_0)).\]
Proving the claim amounts to showing that $I_{D'}$ is generated by $t$. Of course, $t$ belongs to $I_{D'}$ as $t = f_0 / g^n$ in $A[f_0/g^n]_{g}$. On the other hand, any element of $A[f_0/g^n]$ can be written as $\sum_{i = 0}^d a_i t^i$ for some non-negative integer $d$ and $a_0, \dots, a_d \in A$.  Such an element belongs to the ideal generated by $t$ if and only if $a_0$ belongs to the ideal general by $t$. We are therefore led back to prove the following: each $a\in A$ whose image in $A[f_0/g^n]_{g} /(f_0)$ is $0$ belongs  to the ideal generated by $t$. Since the homomorphism $A_{g} \to A[f_0/g^n]_{g}$ is an isomorphism, $a$ is already $0$ in $A_{g} / (f_0)$. The hypothesis of $D \cap U$ being scheme-theoretically dense in $D$ is rephrased as the equality
\[ \Ker(A \to A_{g} / (f_0)) = f_0 A + \cdots + f_r A.\]
(We used here that the ideals $f_0 A_g$ and $f_0 A_g + \cdots + f_r A_g$ coincide.) It follows that $a$ belongs to the ideal $f_0 A + \cdots + f_r A$. Thus, in order to conclude the proof, it suffices to express $f_i$ in terms of $t$ ($i = 1, \dots, r$). By the choice of $n$, there are $b_1, \dots, b_r \in A$ such that $f_i = b_i f_0 / g^n$. Written otherwise, for $i = 1, \dots, r$, the equality $f_i = b_i t$ holds in $A[f_0 / g^n]$, which concludes the proof.
\end{proof}

The Cartier divisor $\pi^{-1}(D_0)$ is given by the ideal $f_0 \cO_{X'}$. The exceptional divisor $E$ is given by the ideal $g^n A[f_0/g^n]$ on the affine chart $X'_0$, while it is given by the ideal $f_0 A[g^n/f_0]$ on the affine chart $X'_1$. It follows that the effective Cartier divisor $\pi^{-1}(D_0) - E$ is given by the ideal generated by $ f_0 / g^n = t$ on $X'_0$ and by the ideal generated by $f_0 / f_0 = 1$ on $X'_1$. One concludes the proof by comparing this with the description of the strict transform $D'$ given by the above claims.
\end{proof}

\begin{proposition} \label{Prop:StrictTransformBlowUpIntersectionTwoCartier} Let  $D, Z$  be effective Cartier divisors in a scheme $X$ whose intersection $D \cap (X \smallsetminus Z)$   is scheme-theoretically dense in $D$. Let $\pi \colon X' \to X$ be the blow-up of $X$ along $D \cap Z$ and $E \subset X'$ the exceptional divisor. 

Then, the strict transform of $D$ in $X'$ is the effective Cartier divisor $\pi^{-1}(D) - E$.
\end{proposition}

\begin{proof} Apply \cref{Lemma:TechnicalBlowingUpLemma} with $D = D_0$ and  $n = 1$.
\end{proof}

\begin{proof}[{Proof of \cref{Thm:ExtendingCartierSchemeVersion}}] Endow the closed subset $Z := X \smallsetminus U$ with its reduced structure. The blow-up $\pi \colon X' \to X$ of $X$ along $Z$ is $U$-admissible because $Z$ is finitely presented. Up to replacing $Z$ by the Cartier divisor $\pi^{-1}(Z)$, and $D$ by $\pi^{-1}(D)$, one may assume that $Z$ is a Cartier divisor. (Note that the composition of $U$-admissible blow-ups is a $U$-admissible blow-up \cite[\href{https://stacks.math.columbia.edu/tag/080L}{Lemma 080L}]{stacks-project}.) Suppose $Z$ is a Cartier divisor. The blow-up $\pi \colon X' \to X$ of $X$ along $D$ is $U$-admissible because $D$ is finitely presented and, since $D \cap U$ is supposed to be a Cartier divisor in $U$, the induced map $\pi \colon \pi^{-1}(U) \to U$ is an isomorphism.  The exceptional divisor  $D'_0 := \pi^{-1}(D)$ is an effective Cartier divisor, as well as the pre-image $Z' = \pi^{-1}(Z)$ of $Z$. Let $D'$ be the scheme-theoretic closure of $D'_0 \smallsetminus Z'$ in $X'$. We may now apply \cref{Lemma:TechnicalBlowingUpLemma} to $X'$, $D'$, $D_0'$ and $Z'$. Indeed, let $U ' = X ' \smallsetminus Z'$. The closed immersion $D' \cap U' \to D_0' \cap U'$ is an isomorphism. On the level of sheaves of ideals, this means that the natural inclusion $I_{D'_0 \rvert U} \to I_{D' \rvert U}$ is an isomorphism. According to \cref{Lemma:ControlDenominators}, there is a non-negative integer $N$ with the following property: given an integer $n \ge N$, the injective homomorphism $I_{D_0'}(n Z') \to I_{D'}(n Z')$ is an isomorphism. In particular, \cref{Lemma:TechnicalBlowingUpLemma} can be applied with every such integer $n$.
\end{proof}

Let $R$ the ring of integers of a field  $K$ endowed with a non-trivial non-Archimedean absolute value. Recall that an $R$-module $M$ is flat if and only if it is torsion-free \cite[\href{https://stacks.math.columbia.edu/tag/0539}{Lemma 0539}]{stacks-project}.

\begin{lemma} \label{Lemma:FinitePresentationAndFlatness}  Let $A$ be a flat $R$-algebra of finite type and $M$ a finitely generated $A$-module flat over $R$.

\begin{enumerate}
\item The $A$-module $M$ is of finite presentation.
\item Let $N \subset M$ be a submodule such that $M / N$ is flat over $R$. Then, $N$ is of finite presentation.
\end{enumerate}
\end{lemma}

\begin{proof} (1) This is a particular case of \cite[theorem 3.4.6]{RaynaudGruson}. 

(2) The $A$-module $M/N$ is finitely generated and flat over $R$ thus finitely presented by (1). Since $M$ is finitely generated and $M/N$ finitely presented, $N$ is finitely generated \cite[\href{https://stacks.math.columbia.edu/tag/0519}{Lemma 0519 (5)}]{stacks-project}. The $A$-module $N$ is finitely generated and flat over $R$, thus, again thanks to (1), it is finitely presented. 
\end{proof}

\begin{proof}[{Proof of \cref{Thm:ModelOfCartierDivisors}}] 
By Nagata's Compactification theorem (see \cite{LutkebohmertNagata},  \cite{ConradNagata} and \cite{ConradNagataErratum}, or \cite{TemkinRiemannZariski}), the morphism of $K$-analytic schemes $\sigma_K \colon X_K \to Y \times_R K$ can be extended to a proper morphism of $R$-schemes $\sigma \colon X \to Y$. Furthermore, up to replacing $X$ by the scheme-theoretic closure of $X_K$ in $X$, one may suppose that $X$ is flat over $R$. Let $D$ be the scheme-theoretic closure of $D_K$. Note that $X$ and $D$ are flat over $R$.  If $K$ is discretely valued, \emph{i.e.} $R$ is Noetherian, then one can blithely apply \cref{Thm:ExtendingCartierSchemeVersion} with $Z$ being the special fibre of $X$. Note that the so-obtained blow-up $\pi \colon X' \to X$ is such that $X'$ is flat over $R$. Moreover, the strict transform of $D$ is flat over $R$ too, being  the scheme-theoretic closure of $\pi^{-1}(D_K)$ in $X'$. In the general case, some argument is needed in order to ensure the hypothesis of finite presentation of the strict transform considered in the proof of \cref{Thm:ExtendingCartierSchemeVersion}.

Let $\varpi \in R \smallsetminus \{ 0\}$ be  topologically nilpotent and let $Z$ be the closed subscheme of $X$ given by the vanishing of $\varpi$. The open subset $U = X \smallsetminus Z$ is the generic fiber $X_K$ of $X$. Let $\pi \colon X' \to X$ the blow-up of $X$ along the first Fitting ideal of $I_D$. Since $D$ is flat and finitely presented over $R$ by \cref{Lemma:FinitePresentationAndFlatness}, the ideal $I_D$ is finitely presented by \cref{Lemma:FinitePresentationAndFlatness}, hence $\Fit_1 I_D$ is finitely presented by \cite[\href{https://stacks.math.columbia.edu/tag/07ZA}{Lemma 07ZA (4)}]{stacks-project} and \cref{Lemma:FinitePresentationAndFlatness}. \Cref{Lemma:BlowUpFittingIdealAlg} below states that $\pi^{-1}(D)$ is an effective Cartier divisor.  Moreover, since $D_K = D \cap U$ is an effective Cartier divisor, the induced morphism of $R$-schemes  $\pi \colon \pi^{-1}(U) \to U$ is an isomorphism. By combining these facts, one deduces that $\pi \colon X' \to X$ is a $U$-admissible blow-up. Let $D_0'$ be the inverse image of $D$ in $X'$: it is an effective Cartier divisor. Let $D'$ be the Zariski of closure of $\pi^{-1}(D_K)$ in $X'$. Again, thanks to \cref{Lemma:FinitePresentationAndFlatness}, the closed subscheme $D'$ is finitely presented and one can apply \cref{Lemma:TechnicalBlowingUpLemma} to $X'$, $D_0'$, $D'$ and the closed subscheme $Z'$ of $X'$ defined by the vanishing of $\varpi$. As in the discretely valued case, the blow-up $\pi' \colon X''\to X'$ given by \cref{Lemma:TechnicalBlowingUpLemma} is such that $X''$ is flat over $R$ as well as the strict transform of $D'$ in $X''$.
\end{proof}

In the above proof we used the following fact:
\begin{lemma} \label{Lemma:BlowUpFittingIdealAlg} Let $X$ be a scheme, and $D \subset X$ a closed subscheme whose ideal sheaf~$I_D$ is finitely presented and whose complement $X \smallsetminus D$ is scheme-theoretically dense. Let $\pi \colon X' \to X$ be the blow-up of $X$ along the first Fitting ideal of $I_D$. Then $\pi^{-1}(D)$ is an effective Cartier divisor in $X'$.
\end{lemma}

\begin{proof} Set $I:= I_D$. By \cite[\href{https://stacks.math.columbia.edu/tag/07ZA}{lemma 07ZA (3)}]{stacks-project} the $k$-th Fitting ideal of $\pi^\ast I$ for $k \ge 0$ is the pull-back of that of $I$:
\[ \Fit_k(\pi^\ast I) = \Fit_k(I)\cO_{X'}.\]
Since $X \smallsetminus D$ is scheme-theoretically dense in $X$, the ideal $\Fit_0(I)$ vanishes altogether, thus so does $\Fit_0(\pi^\ast I)$. On the other hand, by the defining property of the blow-up, the ideal sheaf $\Fit_1(I) \cO_{X'}$ is invertible. It follows that the ideal $J:= \Fit_1(\pi^\ast I)$ is invertible. Set $M:= \pi^\ast I$ and let $M' \subset M$ be the submodule annihilated by~$J$. We claim that the natural morphism $\phi \colon \pi^\ast I \to \cO_{X'}$ induces an isomorphism \[M/M' \stackrel{\sim}{\too} I \cO_{X'}.\] This will conclude the proof, as the $\cO_{X'}$-module $M / M'$ is invertible \cite[\href{https://stacks.math.columbia.edu/tag/0F7M}{Lemma 0F7M}]{stacks-project}. To show the claim, first note that $I \cO_{X'}$ is the image of $\phi$, so it suffices to prove $M' = \ker \phi$.
The inclusion $M' \subset \ker \phi$ is obvious because $\cO_{X'}$ has no torsion with respect to $J$, the latter ideal being invertible. For the converse inclusion, $\pi$ induces an isomorphism outside the exceptional divisor $E \subset X'$ hence $\phi$ is injective outside $E$. The ideal defining $E$ is $J$, hence any local section of $\ker \phi$ is annihilated by some power of~$J$. It follows from \cite[\href{https://stacks.math.columbia.edu/tag/0F7M}{Lemma 0F7M (2)}]{stacks-project} that $\ker \phi$ is actually annihilated by $J$, which concludes the proof.
\end{proof}

\section{Extending Cartier divisors: the formal case} \label{sec:ExtensionCartierDivisorsFormal} Let $K$ be a complete non-trivially valued  non-Archimedean field   and $R \subset K$ its ring of integers.

\subsection{Preliminaries on formal schemes} Recall that a formal scheme over $R$ is \emph{admissible} if it is locally the formal spectrum of a flat $R$-algebra which is topologically of finite presentation. To such a formal scheme $X$ is attached its \emph{(Raynaud) generic fiber}~$X_\eta$ seen as a $K$-analytic space in the sense of Berkovich. The generic fiber of  an affine formal scheme $X= \Spf A$ where $A$ is flat topologically of finite type is the Banach spectrum $X_\eta = \cM(A \otimes_R K)$ of the affinoid $K$-algebra $A \otimes_R K$. Here we recall some facts that are repeatedly used afterwards:

\begin{lemma} \label{Lemma:NoetherianTopologicalSpace} Let $X$ be a quasi-compact formal $R$-scheme topologically of finite type. Then the  topological space underlying $X$ is Noetherian.
\end{lemma}

\begin{proof} Let $k$ be the residue field of $R$. The topological space $|X|$ underlying $X$ coincides with the topological space underlying the finite type $k$-scheme $X_{\red}$: in particular $|X|$ is a Noetherian topological space. 
\end{proof}

\begin{lemma} \label{Lemma:FinitePresentationAndFlatnessFormalVersion} Let $A$ be a flat $R$-algebra topologically of finite type and $M$ a finitely generated $A$-module flat over $R$.

\begin{enumerate}
\item The $A$-module $M$ is of finite presentation.
\item If $N \subset M$ is a submodule such that $M / N$ is flat over $R$, then $N$ is of finite presentation.
\item The ring $A$ is coherent.
\end{enumerate}
\end{lemma}

\begin{proof} (1) and (2) are respectively theorem 7.3.4 and lemma 7.3.6 in  \cite{BoschLectures}, and (3) follows from (2).
\end{proof}

\begin{lemma} \label{Lemma:InvertingUniformizerFormalScheme} Let $X$ be a quasi-compact formal $R$-scheme topologically of finite type and $F$ a sheaf of $R$-modules on $X$. The presheaf $U \mapsto \Gamma(U, F) \otimes_R K$ on $X$ is a sheaf.
\end{lemma}

\begin{proof} Let $U$ be an open subset of $X$. Since the topological space underlying $X$ is Noetherian (\cref{Lemma:NoetherianTopologicalSpace}), the open subset $U$ is quasi-compact. Let $U = \bigcup_{i = 1}^n U_i$ a finite open cover of $U$. The exact sequence
\[ 0 \too F(U) \stackrel{\delta_0}{\too} \prod_{i = 1}^n F(U_i) \stackrel{\delta_1}{\too} \prod_{i, j = 1}^n F(U_i \cap U_j),\]
where $\delta_0$ and $\delta_1$ are the usual coboundary operator of the \v{C}ech cochain complex, stays exact after taking its tensor product with $K$ (because $K$ is a flat $R$-algebra):
\[ 0 \too F(U) \otimes_R K \stackrel{\delta_0}{\too} \left( \prod_{i = 1}^n F(U_i) \right) \otimes_R K \stackrel{\delta_1}{\too} \left( \prod_{i, j = 1}^n F(U_i \cap U_j) \right)\otimes_R K.\]
One concludes by noticing that tensor product commutes with finite products.
\end{proof}

\begin{definition} Let $X$ be a quasi-compact formal $R$-scheme topologically of finite type and $F$ an $\cO_X$-module. We denote the sheaf $U \mapsto \Gamma(U, F) \otimes_R K$ by  $F_K$.
\end{definition}

\subsection{Admissible blow-ups} The central tool in this section is blowing-up in the context of admissible formal schemes. Let us recall a few definitions. Let $X$ be an admissible formal $R$-scheme, $\varpi \in R \smallsetminus \{ 0\}$ topologically nilpotent, $I \subset \cO_X$ a coherent sheaf of ideals such that, locally on $X$, it contains a sheaf of ideals of the form $\mu \cO_X$ for some $\mu \in R \smallsetminus \{ 0\}$. Following \cite[def. 8.2.3]{BoschLectures}, the formal $R$-scheme
\[ X' = \injlim_{n \in \bbN} \Proj \left( \bigoplus_{d \in \bbN} I^d \otimes_{\cO_X} (\cO_X / \varpi^n \cO_X)\right),\]
with structural morphism $\pi \colon X' \to X$, is called the \emph{blow-up of $X$ along $I$}.
The \emph{exceptional divisor} of $\pi$ is the closed formal subscheme of $X'$ defined by the coherent sheaf of ideals $I \cO_{X'}$. The formal scheme $X'$ is admissible by \cite[corollary 8.2.8]{BoschLectures}.

\begin{lemma} \label{Lemma:CoherenceIdealFormalStrictTransform} With the notation above, let $Z \subset X$ be a closed admissible formal subscheme defined by a coherent sheaf of ideals $J \subset \cO_X$. Then,
\[ J' := \Ker( \cO_{X'} \to \cO_{X', K} / J\cO_{X', K} )\]
is a coherent sheaf of ideals of $\cO_{X'}$.
\end{lemma}

\begin{proof} The statement is local on $X$, therefore one may assume that $X = \Spf A$ is affine, where $A$ is a topologically finitely generated flat $R$-algebra. In order to simplify notation, identify $I$ (resp. $J$) with the finitely presented ideal $\Gamma(X, I)$ (resp. $\Gamma(X, J)$). Let $f_0, \dots, f_r$ be generators of $I$. The blow-up $X'$ is covered by affine open subsets $X'_i$, $i = 0, \dots, r$,
\begin{align*} 
X'_i = \Spf A \{ f_j / f_i\}_{j \neq i}, 
&& A\{ f_j / f_i\}_{j \neq i} = [A \{ t_{ij}\}_{j \neq i} / (f_j - t_{ij} f_i)_{j \neq i}] / (f_i\textup{-torsion}).
\end{align*}
By definition, 
\[ \Gamma(X'_i, J') = \Ker( A \{ f_j / f_i\}_{j \neq i} \to [A \{ f_j / f_i\}_{j \neq i} / J A \{ f_j / f_i\}_{j \neq i}] \otimes_R K).\]
According to \cref{Lemma:FinitePresentationAndFlatnessFormalVersion},  $\Gamma(X'_i, J')$ is a finitely presented ideal of $A \{ f_j / f_i\}_{j \neq i}$, which concludes the proof.
\end{proof}

With the notation of above lemma the formal closed subscheme $Z' \subset X'$ associated with $J'$ is called the \emph{strict transform} of $Z$.

\subsection{Statements} We are now in position to state the main results of this section. To do this, notice that the concepts in \cref{sec:CartierDivisor} are translated \emph{verbatim} when $X$ is a formal $R$-scheme or a $K$-analytic space. An effective Cartier divisor $D$ on a quasi-compact admissible formal $R$-scheme $X$ is \emph{admissible} if it is admissible as a formal $R$-scheme.

\begin{theorem} \label{Thm:ExtendingCartierFormalVersion}  Let $X$ be a quasi-compact admissible formal $R$-scheme and $D \subset X$ an admissible formal closed subscheme such that $I_{D, K}$ is an invertible $\cO_{X, K}$-module. 

Then, there exists an admissible blow-up $\pi \colon X' \to X$ such that the strict transform of $D$ in $X'$ is an admissible effective Cartier divisor.
\end{theorem}

\begin{theorem} \label{Thm:FormalModelsOfCartierDivisors} Let $Y$ be an admissible affine formal $R$-scheme, $X_K$ a strict compact quasi-separated $K$-analytic space, $f_K \colon X_K \to Y_\eta$ a proper morphism of $K$-analytic spaces and 
$D_K \subset X_K$ an effective Cartier divisor. 

Then, there are a quasi-compact quasi-separated admissible formal $R$-scheme $X$, a proper morphism $f \colon X \to Y$, an admissible effective Cartier divisor $D$ in $X$, and an isomorphism $X_K \cong X_\eta$ through which $D_K$ is identified with $D_\eta$ and $f_K$ with $f_\eta$.
\end{theorem}

Notice that the hypothesis of $Y$ being an admissible formal $R$-scheme is easily achieved. Indeed let $A_K$ be a strictly affinoid $K$-algebra and $\phi \colon K \{ t_1, \dots, t_n \} \to A_K$ a surjective homomorphism of Banach $K$-algebras. According to \cref{Lemma:FinitePresentationAndFlatnessFormalVersion}, the ideal $I := \ker \phi \cap R \{ t_1, \dots, t_n \}$ is finitely generated. Thus $A := R \{ t_1, \dots, t_n \} / I$ is a flat, topologically finitely presented $R$-algebra and $A \otimes_R K \cong A_K$.

\subsection{Fitting ideals in formal geometry} In the proof of \cref{Thm:ModelOfCartierDivisors} we started by blowing up $D$ so that we could suppose that it was a Cartier divisor. Here it is not the right operation to do, somewhat because a formal scheme does not have a true generic fiber. As we will see the correct transformation to perform is to blow-up $X$ along the first Fitting ideal of the coherent sheaf of ideals defining $D$. For lack of references we recall here the definition of Fitting ideals. Let $X$ be an admissible formal $R$-scheme and $F$ a coherent $\cO_X$-module.

\begin{lemma} \label{Lemma:FormalFittingIdeal} For an integer $r \ge 0$  let $\Fit_rF$ be the sheafification of the presheaf associating to an affine open subset $U$ the $r$-th Fitting ideal of $\Gamma(U, F)$. Then $\Fit_rF$ is a coherent sheaf of ideals and, for an affine open subset $U$,
\[ \Gamma(U, \Fit_r F) = \Fit_r \Gamma(U, F). \]
\end{lemma}

\begin{proof} Let $U = \Spf(A)$ be an affine open subset, where $A$ is a finitely presented flat $R$-algebra, and $M := \Gamma(U, F)$.  We claim that $\Fit_r (F)$ coincides with the $\cO_X$-module associated with the finitely generated ideal $\Fit_r M$. Indeed, for $f \in A$, the following formula holds \cite[\href{https://stacks.math.columbia.edu/tag/07ZA}{lemma 07ZA (3)}]{stacks-project}: 
\[ \Fit_r (M \otimes_A A_{\{f\}}) = (\Fit_r M) A_{\{f\}}. \]
On the other hand $A_{\{f\}}$ is flat\footnote{The ring $A_{\{ f\}}$ is the completion of $A_f$, therefore it is flat over $A_f$ \cite[8.2, lemma 2]{BoschLectures}. The ring $A_f$ is flat over $A$, hence $A_{\{ f\}}$ is flat over $A$.} over $A$, thus the right-hand of the previous equality coincides with $ (\Fit_r M) \otimes_A A_{\{f\}}$. This concludes the proof.
\end{proof}

The sheaf $\Fit_r F$ is called the \emph{$r$-th Fitting ideal of $F$}.

\begin{lemma} \label{Lemma:BlowUpFittingIdeal}  Let $D \subset X$ be an admissible closed formal subscheme such that $I_{D ,K}$ is an invertible $\cO_{X, K}$-module.
\begin{enumerate}
\item The ideal $\Fit_1 I_D$ locally contains a sheaf of the form $\varpi \cO_{X}$ with $\varpi \in R \smallsetminus \{ 0\}$. 
\item If $\pi \colon X' \to X$ is the blow-up of $X$ along $\Fit_1 I_D$, then $I_D \cO_{X'}$ is an invertible $\cO_{X'}$-module  and $\pi^{-1}(D)$ is an effective Cartier divisor.
\item If $D'$ be the strict transform of $D$ in $X'$, then $I_{D} \cO_{X', K} = I_{D', K}$.
\end{enumerate}
\end{lemma}

\begin{proof} All statements are local on $X$, therefore one may assume $X = \Spf A$, where~$A$ is a topologically finitely presented flat $R$-algebra, and $I_K := \Gamma(X, I) \otimes_R K$ is a principal ideal of $A_K := A \otimes_R K$.

(1)  We have $ \Fit_1 I_K = (\Fit_1 I) A_K = (\Fit_1 I) \otimes_R K$ by  \cite[\href{https://stacks.math.columbia.edu/tag/07ZA}{lemma 07ZA (3)}]{stacks-project}  and flatness of the $A$-algebra $A_K$. The ideal $I_K$ is principal, thus $\Fit_1 I_K = A_K$ by \cite[\href{https://stacks.math.columbia.edu/tag/07ZD}{lemma 07ZD}]{stacks-project}. This implies that $\Fit_1 I$ contains some $\varpi \in R \smallsetminus \{ 0\}$.

(2) The proof of \cref{Lemma:BlowUpFittingIdealAlg} can be translated in the formal setting with no changes, and we do not repeat it here.

(3) By definition $I_{D'} = \Ker (\cO_{X'} \to \cO_{X', K} / I_{D}\cO_{X', K})$ whence the statement.
\end{proof}

\subsection{Proofs} After these preliminaries we can finally prove \cref{Thm:ExtendingCartierFormalVersion,,Thm:FormalModelsOfCartierDivisors}. The arguments in the proof for \cref{Thm:ExtendingCartierFormalVersion} are a routine adaptation of those in the proofs of \cref{Thm:ExtendingCartierSchemeVersion,,Thm:ModelOfCartierDivisors}. For an admissible formal $R$-scheme $X$, a coherent $\cO_X$-module $F$ flat over $R$ and $\varpi \in R \smallsetminus \{ 0\}$, let $\frac{1}{\varpi} F$ denote the image of $F$ in $F_K$ through the map $f \mapsto f / \varpi$. The proofs of the following three statements are obtained \emph{mutatis mutandis} respectively from those of \cref{Lemma:ControlDenominators,,Lemma:TechnicalBlowingUpLemma,,Prop:StrictTransformBlowUpIntersectionTwoCartier}:

\begin{lemma} \label{Lemma:ControlDenominatorsFormal} Let $\phi \colon F \to G$ a homomorphism between coherent sheaves on a quasi-compact admissible formal $R$-scheme  $X$. If $\phi_{K} \colon F_{K} \to G_{K}$ is surjective, then there is $\varpi \in R \smallsetminus \{ 0\}$ such that  $\im(G \to G_K)$ is contained in $\im(\frac{1}{\varpi}F \to G_K)$.
\end{lemma}

\begin{lemma} \label{Lemma:TechnicalBlowingUpLemmaFormal} Let $D_0$ be an effective Cartier divisor in an admissible formal $R$-scheme and $D \subset D_0$ an admissible formal closed subscheme. Assume that the image of $I_{D}$ in $I_{D, K}$ is contained in $\frac{1}{\varpi}I_{D_0, K}$ for some $\varpi \in R \smallsetminus \{ 0\}$. Let $\pi \colon X' \to X$ be the blow-up along the ideal $\varpi \cO_X + I_{D_0}$ and $E \subset X'$ the exceptional divisor. 

Then, the strict transform of $D$ in $X'$ is the effective Cartier divisor $\pi^{-1}(D_0) - E$.
\end{lemma}

\begin{proposition} \label{Prop:StrictTransformBlowUpIntersectionTwoCartierFormal} Let $D$ be an effective Cartier divisor in a quasi-compact admissible formal $R$-scheme $X$, $\varpi \in R \smallsetminus \{ 0\}$, $\pi \colon X' \to X$ the blow-up of $X$ along $I_D + \varpi \cO_X$ and $E \subset X'$ the exceptional divisor. Then, $\pi \colon X' \to X$ is an admisible blow-up and the strict transform of $D$ in $X'$ is the effective Cartier divisor $\pi^{-1}(D) - E$.
\end{proposition}

\begin{proof}[{Proof of \cref{Thm:ExtendingCartierFormalVersion}}] Let $\pi \colon X' \to X$ the blow-up along the first Fitting ideal $\Fit_1(I_D)$ of the coherent $\cO_X$-module $I_D$. According to \cref{Lemma:BlowUpFittingIdeal}, the blow-up $\pi \colon X' \to X$ is admissible and $D_0 = \pi^{-1}(D)$ is an effective Cartier divisor in $X'$. Moreover, if $D'$ denotes the strict transform of $D$ in $X'$, then $I_{D', K'} = I_{D_0, K'}$. Thanks to \cref{Lemma:ControlDenominatorsFormal} one can apply \cref{Lemma:TechnicalBlowingUpLemmaFormal} to $X'$, $D'$ and $D_0'$.
\end{proof}

Let $X$ be a quasi-compact admissible formal $R$-scheme and $Z_K \subset X_\eta$ a closed analytic subspace. Then, there is a unique closed formal subscheme $Z$ flat over $R$ such that $Z_\eta = Z_K$. We call $Z$ the \emph{Zariski closure} of $Z_K$ in analogy with the situation for schemes. When $X = \Spf(A)$ is affine, $Z$ is defined by the ideal $I \cap A$ where $I$ is the kernel of the evaluation morphism $A \otimes_R K = \Gamma(X_\eta, \cO_{X_\eta}) \to \Gamma(Z_K, \cO_{X_\eta})$. Note that $I$ is finitely presented by \cref{Lemma:FinitePresentationAndFlatnessFormalVersion}.

\begin{lemma} \label{Lemma:ExtensionOfCartierToFormalSchemes} Let $X$ be a quasi-compact admissible formal $R$-scheme, $D_K \subset X_\eta$ an effective Cartier divisor and $D \subset X$ the Zariski closure of $D_K$. Then, the $\cO_{X, K}$-module $\ker(\cO_X \to \cO_D)_K$ is invertible.
\end{lemma}

\begin{proof} One may suppose that $X = \Spf A$ is affine. The closed subscheme $D$ is then defined by an ideal $I \subset A$. We need to show that $I \otimes_R K$ is invertible knowing that $D_\eta \subset X_\eta = \cM(A \otimes_R K)$ is an effective Cartier divisor. This is a consequence of the faithful flatness of the morphism $\cM(A \otimes_R K) \to \Spec (A \otimes_R K)$, see \cite[theorem 2.1.4]{BerkovichIHES}. 
\end{proof}

\begin{proof}[{Proof of \cref{Thm:FormalModelsOfCartierDivisors}}] Raynaud's point of view on rigid analytic spaces relies upon the insight that every (strict, paracompact, and quasi-separated) $K$-analytic space, and morphisms between, can be obtained as the generic fibre of an admissible formal $R$-scheme \cite[8.4, lemma 4]{BoschLectures}.  Let $X$ be an admissible formal $R$-scheme with generic fibre $X_K$. Up to taking a suitable formal admissible blow-up of $X$ the morphism $f_K$ extends to a morphism of formal $R$-schemes $f \colon X \to Y$. The morphism $f$ is proper by \cite[theorem 3.1]{LutkebohmertProperness} in the discretely valued case and \cite[Corollary 4.4]{TemkinLocalPropertiesI} in the general case.  According to \cref{Lemma:ExtensionOfCartierToFormalSchemes}, there is an admissible blow-up $\pi \colon X' \to X$ and an admissible closed formal subscheme $D'$ in $X'$ with generic fibre $\pi_\eta^{-1}(D_K)$ and such that $I_{D', K}$ is an invertible $\cO_{X', K}$-module. We conclude by applying \cref{Thm:ExtendingCartierFormalVersion} to $X'$ and $D'$.
\end{proof}

\section{Approximating essential singularities by poles: the formal case} \label{sec:GlobalApproximation}

Let $R$ be the ring of integers of a field $K$ complete with respect to a on-trivial non-Archimedean absolute value.

\subsection{Functions with polar singularities} The main result of this section concerns the approximation of formal functions having essential singularities at infinity by those having polar singularities. More precisely, let $D$ be an effective Cartier divisor on a scheme $X$. Given an $\cO_X$-module $F$, for $d \in \bbN$ set $F(dD) := F \otimes_{\cO_X} \cO_X(dD)$ and let
\[ F(\ast D) := \injlim_{d \in \bbN} F(dD)\]
be the sheaf of sections of $F$ with polar singularities along $D$, where the transition maps are induced by the canonical section $s_D \colon \cO_X \to \cO_X(D)$. The $\cO_X$-module $F$ is \emph{$D$-torsion free} if $\id_F \otimes s_D \colon F \to F(D)$ is injective. Let $U = X \smallsetminus D$ and let $j \colon U \to X$ be the open immersion. The canonical section induces an isomorphism $\cO_U \cong \cO_X(D)_{\rvert U}$ hence an isomorphism $j^\ast F \cong j^\ast F(dD)$ for each $d \in \bbN$. By taking the injective limit of the homomorphisms $F(dD) \to j_\ast j^\ast F(dD) \to j_\ast j^\ast F$, the latter map being the inverse of previous isomorphism, one obtains a homomorphism $F(\ast D) \to j_\ast j^\ast F$ that will be referred to as the \emph{restriction map}.

\begin{lemma} Let $D$ be an effective Cartier divisor in a quasi-compact scheme~$X$ and~$F$ a quasi-coherent $\cO_X$-module. Then the restriction map induces an isomorphism
\[ \injlim_{d \in \bbN} \Gamma(X, F(dD)) \stackrel{\sim}{\too} \Gamma(X \smallsetminus D, F).\]
Moreover if $X$ is quasi-separated, then the map $F(\ast D) \to j_\ast j^\ast F$ given by restriction is an isomorphism where $j \colon X\smallsetminus D \to X$ is the open immersion.
\end{lemma}

\begin{proof} For an $\cO_X$-module $E$ let $\Gamma(E, D) := \bigoplus_{d \in \bbN} \Gamma(X,  E(dD))$. With this notation $\injlim_{d \in \bbN} \Gamma(X, F(dD))$ is the homogeneous localization of the $\Gamma(\cO_X, D)$-graded module $\Gamma(F, D)$ with respect to the canonical section $s_D$, which is a homogeneous element of degree $1$. The statement then follows from \cite[proposition 1.4.5]{EGAIII1}. For the latter statement notice that on a quasi-compact scheme, every open cover may be refined by a finite open affine cover. Moreover, if the scheme is quasi-separated, then the intersection of affine open subsets is a finite union of open affine subsets, hence quasi-compact. Thus the natural map
$ \injlim_{d \in \bbN} \Gamma(V, F(dD)) \to  \Gamma(V, F(\ast D))$ is an isomorphism for any quasi-compact open subset $V \subset X$  by \cite[\href{https://stacks.math.columbia.edu/tag/009F}{lemma 009F (4)}]{stacks-project}. Applying statement (1) to $V$,  the composite map
\[  \injlim_{d \in \bbN} \Gamma(V, F(dD)) \stackrel{\sim}{\too}  \Gamma(V, F(\ast D))  \too  \Gamma( V \cap (X \smallsetminus D), F) = \Gamma( V, j_\ast j^\ast F) \]
is an isomorphism. The result follows because quasi-compact open subsets form a basis for the topology of $X$.
\end{proof}

\begin{lemma} \label{Lemma:MeromorphicSectionsFormal} Let $D$ be an effective Cartier divisor in a separated quasi-compact formal $R$-scheme $X$ topologically of finite type. Then, for any sheaf of $\cO_X$-modules~$F$, the maps $F(dD) \to F(\ast D)$ induce an isomorphism
\[ \injlim_{d \in \bbN} \Gamma(X, F(dD)) \stackrel{\sim}{\too} \Gamma(X, F(\ast D)).\]
\end{lemma}

\begin{proof} Since the topological space underlying $X$ is Noetherian (\cref{Lemma:NoetherianTopologicalSpace}) the statement follows from \cite[\href{https://stacks.math.columbia.edu/tag/009F}{lemma 009F (4)}]{stacks-project}.
\end{proof}

\subsection{Statements} In order to state the main result of this section, say that an admissible formal $R$-scheme $X$ is \emph{proper over affine} if there are an affine admissible formal $R$-scheme $Y$ and a proper morphism of formal $R$-schemes $X \to Y$.

\begin{theorem} \label{Thm:DensityMeromorphicSectionsModel} Let $D$ be an effective Cartier divisor in an admissible formal $R$-scheme $X$ which is proper over affine. Let $F$ be a coherent sheaf on $X$, which is $D$-torsion free, and such that $F$ and $F_{\rvert D}$ are flat over $R$. Then, the natural map 
\[ \Gamma(X, F(\ast D)) \too \Gamma(X \smallsetminus D, F), \]
is injective, and has a dense image. \end{theorem}

Recall that the topology on the global section of a coherent sheaf $F$ on a quasi-compact admissible formal $R$-scheme $X$ is defined as follows. For any topologically nilpotent $\varpi \in R \smallsetminus \{ 0 \}$ we have
\[ \Gamma(X, F) = \projlim_{n \in \bbN} \Gamma(X, F/\varpi^n F) \]
and the prodisrete topology on the left-hand side is given by the discrete topology on each $\Gamma(X, F/\varpi^n F)$. See also \cite[\href{https://stacks.math.columbia.edu/tag/0AHY}{Section 0AHY}]{stacks-project}. When $X$ is affine and $D$ is principal, \cref{Thm:DensityMeromorphicSectionsModel} is rather elementary and boils down to approximating power series by polynomials. The strength of \cref{Thm:DensityMeromorphicSectionsModel} comes from the non-affine case, where the proof becomes more involved. We do not know whether \cref{Thm:DensityMeromorphicSectionsModel} holds without the assumption proper over affine. However, notice that the conclusion of \cref{Thm:DensityMeromorphicSectionsModel}  holds more generally when $X$ can be compactified into a proper admissible formal scheme over $R$, by applying \cref{Thm:DensityMeromorphicSectionsModel} to the compactification. The remainder of this section is devoted to the proof of \cref{Thm:DensityMeromorphicSectionsModel}.

\subsection{Norms and modules} \label{sec:NormsAndModules} We begin by recalling some facts about the topology on an $R$-module $M$. Its \emph{completion} is the $R$-module $\hat{M} := \projlim_{n \in \bbN} M / \varpi^n M$ where $\varpi \in R \smallsetminus \{ 0 \}$ is topologically nilpotent. Notice that $\hat{M}$ does not depend on the chosen $\varpi$.  The $R$-module $M$ is said to be \emph{separated} (resp. \emph{complete}) if the natural map $\epsilon \colon M \to \hat{M}$ is injective (resp. \emph{bijective}). Suppose $M$ torsion-free and identify $M$ with $\im(M \hookrightarrow M \otimes_R K)$. For $m \in M \otimes_R K$, let 
\[ \| m \|_M := \inf \{ |\varpi| : \varpi \in K^\times, m/\varpi \in M\}.\]
This defines a semi-norm $\| \cdot \|_M \colon M \otimes_R K \to \bbR_+$; it is a norm if and only if $M$ is separated. The $R$-module $\hat{M}$ is torsion-free and $\| \epsilon(m) \|_{\hat{M}} = \| m\|_M$ for each $m \in M$. Moreover $\epsilon$ extends to a homeomorphism between the completion of $M$ with respect to the semi-norm $\| \cdot \|_M$ and $\hat{M}$. The construction of the semi-norm $\| \cdot\|_M$ is functorial. Let $f \colon M \to M'$ be a homomorphism between separated torsion-free $R$-modules. Then $\| f(m) \|_{M'} \le \| m \|_M$ for each $m \in M$. If $f$ injective, then $\coker f$ is torsion-free if and only if $\| f(m)\|_{M'} = \| m \|_M$ for each $m \in M$.

\begin{lemma} \label{Lemma:CompletenessGlobalSections}
Let $X$ be an admissible formal $R$-scheme and $F$ a coherent $\cO_X$-module flat over $R$. Then,
\begin{enumerate}
\item the natural topology on $\Gamma(X, F)$ coincides with the $\varpi$-adic one;
 \item $\Gamma(X, F)$ is complete.
\end{enumerate} 
\end{lemma} 

\begin{proof} We prove both statements at the same time. Let $\varpi \in R \smallsetminus \{ 0 \}$ be topologically nilpotent and for $n \in \bbN$ let $X_n$ be the closed subscheme of $X$ defined by the equation $\varpi^n = 0$. By definition,
\[ \Gamma(X, F) = \projlim_{n \in \bbN} \Gamma(X_n, F).\]
Flatness of $F$ over $R$ implies that the sequence
\[ 0 \too \varpi^n F \too F \too F_{\rvert X_n} \too 0\]
is short exact. It follows that the natural map $\Gamma(X, F) / \varpi^n \Gamma(X, F) \to \Gamma(X_n, F)$ is injective, because taking global sections is a left exact functor. Since the projective limit is also left-exact, the natural map
\[ \projlim_{n \in \bbN} \Gamma(X, F) / \varpi^n \Gamma(X, F) \too \projlim_{n \in \bbN} \Gamma(X_n, F).\]
is injective. Therefore the map $\Gamma(X, F) \to \projlim_{n \in \bbN} \Gamma(X, F) / \varpi^n \Gamma(X, F) $ is bijective. This proves both statements.
\end{proof}

\subsection{Proof} Let $S$ be an $R$-scheme, $M$ a quasi-coherent $\cO_X$-module, and $\varpi$ a nonzero topologically nilpotent element of $R$. Then, the sequence of $R / \varpi R$-modules
\[ 0 \too \Gamma(S, M) / \Gamma(S, \varpi M) \too \Gamma(S, M / \varpi M) \too \Gamma(S, M / \varpi M) / \Gamma(S, M) \too 0 \]
is short exact. For a nonzero topologically nilpotent element $\varpi'$ such that $|\varpi'| \le |\varpi|$ and a homomorphism of $\cO_S$-modules $\phi \colon M \to M'$, the above discussion, applied with $S = S_{\varpi'}$, yields the following commutative and exact diagram of $R$-modules:
\begin{center}
\begin{tikzcd}
0 \ar[r] & \displaystyle \frac{\Gamma(S_{\varpi'}, M) }{ \Gamma(S_{\varpi'}, \varpi M)} \ar[r] \ar[d] & \Gamma(S_\varpi, M) \ar[r] \ar[d] & \displaystyle  \frac{ \Gamma(S_\varpi, M) }{ \Gamma(S_{\varpi'}, M)} \ar[r] \ar[d] & 0\\
0 \ar[r] & \displaystyle \frac{\Gamma(S_{\varpi'}, M') }{ \Gamma(S_{\varpi'}, \varpi M')} \ar[r] & \Gamma(S_\varpi, M') \ar[r] &\displaystyle   \frac{\Gamma(S_\varpi, M') }{ \Gamma(S_{\varpi'}, M') }\ar[r] & 0.
\end{tikzcd}
\end{center}

\begin{lemma} \label{Lemma:KeyInjectivityFormal}Let $D$ be an effective Cartier divisor in an admissible formal $R$-scheme $X$ with canonical section $s \colon \cO_X \to \cO_X(D)$. Let $F$ be a coherent $\cO_X$-module which is $D$-torsion free and such that $F$, $F_{\rvert D}$ are flat over $R$. Then, for any $i, d \in\bbN$ and any $\varpi \in R \smallsetminus \{ 0 \}$, 
\begin{enumerate}
\item the map $s^i \colon \Gamma(X_\varpi, F(dD)) \to \Gamma(X_\varpi, F((d+i)D))$ is injective;
\item given $\varpi' \in R \smallsetminus \{0\}$ such that $|\varpi'| \le |\varpi|$, the map $s^i$ induces an injection:
\[ 
s^i \colon \frac{\Gamma(X_{\varpi}, F(dD))}{\Gamma(X_{\varpi'}, F(dD))} \hookrightarrow
\frac{\Gamma(X_{\varpi}, F((d+i)D))}{\Gamma(X_{\varpi'}, F((d+i)D))}.
\]
\end{enumerate}
\end{lemma}

\begin{proof} Up to replacing $F$ by $F(dD)$ one may assume $d = 0$. Moreover, by induction on $i$ one reduces to the case $i = 1$.

(1) Since $s \colon F \to F(D)$ is injective by assumption, the following sequence of coherent $\cO_X$-modules is exact:
\begin{equation} \label{SES:RestrictionToDivisor} 0 \too F \stackrel{s}{\too} F(D) \too F(D)_{\rvert D} \too 0.\end{equation}
The coherent sheaf $F(D)_{\rvert D}$ is flat over $R$. Therefore, the previous short exact sequence stays exact after taking its tensor product with any $R$-module $M$. In particular, taking $M = R / \varpi R$ yields the following short exact sequence of coherent $\cO_{X_\varpi}$-modules 
\[ 0 \too F_{\rvert X_\varpi} \stackrel{s}{\too} F(D)_{\rvert X_\varpi} \too F(D)_{\rvert D_\varpi} \too 0. \]
The statement follows by taking global sections.

(2) The short exact sequence of $\cO_X$-modules \eqref{SES:RestrictionToDivisor} yields the following commutative diagram:
\[
\begin{tikzcd}
  \displaystyle \frac{\Gamma(X_{\varpi'}, F)}{ \Gamma(X_{\varpi'}, \varpi F)} \ar[d, hook] \ar[r, "s"] &
 \displaystyle  \frac{\Gamma(X_{\varpi'}, F(D))}{ \Gamma(X_{\varpi'}, \varpi F(D))} \ar[d, hook] \ar[r] &
 \displaystyle  \frac{\Gamma(D_{\varpi'}, F(D))}{ \Gamma(D_{\varpi'}, \varpi F(D))} \ar[d, hook] \\
 \Gamma(X_{\varpi}, F) \ar[r, "s", hook]&  \Gamma(X_{\varpi}, F(D)) \ar[r] & \Gamma(D_{\varpi}, F(D)),
\end{tikzcd} 
\]
where $s \colon  \Gamma(X_{\varpi}, F) \to  \Gamma(X_{\varpi}, F(D))$ is injective because of (1), and the vertical arrows are injective because of the discussion before the lemma. The statement then follows by diagram chasing.
\end{proof}

\begin{lemma} \label{Lemma:MeromorphicInjectsInHolomorphicFormal} Let $D$ be an effective Cartier divisor in a quasi-compact separated admissible formal $R$-scheme $X$. Let $F$ be a $D$-torsion free coherent $\cO_X$-module such that $F$ and $F_{\rvert D}$ are flat over $R$.  Let $j \colon X \smallsetminus D \to X$ be the open immersion. Then,

\begin{enumerate}
\item the natural homomorphisms $\alpha \colon F \to F(\ast D)$, $\beta \colon F \to j_\ast j^\ast F$ are injective with cokernel flat over $R$;
\item the restriction homomorphism $\rho \colon F(\ast D) \to j_\ast j^\ast F$ is injective;
\item the cokernel of $\rho$ is a sheaf of $K$-vector spaces;
\item the restriction map $\rho \colon \Gamma(X, F) \to \Gamma(X \smallsetminus D, F)$  is an isometric embedding with respect to the norms $\| \cdot \|_{\Gamma(X, F)}$ and $\| \cdot \|_{\Gamma(X \smallsetminus D, F)}$ defined in \cref{sec:NormsAndModules}. In particular, $\rho$ induces a homeomorphism of $\Gamma(X, F)$ with a closed subset of $\Gamma(X \smallsetminus D, F)$.
\end{enumerate}
\end{lemma}

\begin{proof} Statements (1), (2) and (3) are local on $X$, therefore we may suppose that $X = \Spf(A)$ is affine, for a topologically finitely presented flat $R$-algebra $A$, and $D$ is a principal Cartier divisor.  Let $f \in A$ be a generator of the ideal of $D$. Let $M:= \Gamma(X, F)$. With this notation, 
\begin{align} \label{Eq:MeromorphicSectionsAndSectionsOnOpenSubets}
\Gamma(X, F(\ast D)) = M [t] / (1 - tf), && \Gamma(U, F) = M \{ t \} / (1 - tf).
\end{align}
We first prove (2), (3) and then (1) under this additional assumption. Finally, we prove (4) for general $X$.

(2) Let $m(t) \in M [t]$. Write $m(t) = \sum_{i = 0}^d m_i t^i$ with $m_i \in M$. Suppose its image in $\Gamma(U, F)$ vanishes. This means that there exists $n(t) \in M \{ t \}$ such that $m(t) = (1 - ft) n(t)$. Write $n(t) = \sum_{i = 0}^\infty n_i t^i$ with $n_i \to 0$ as $i \to \infty$. By comparing Taylor expansions, for $i,j \in \bbN$ with $i\ge d$, $n_{i + j} = f^j n_i$. Let $\varpi \in R \smallsetminus \{ 0 \}$ be topologically nilpotent. By \cref{Lemma:CompletenessGlobalSections} the topology on $M$ coincides with the $\varpi$-adic one. Therefore, because of convergence, there is $I_\varpi \in \bbN$ such that, for $i \ge I_\varpi$,
\[ n_i \equiv 0 \pmod{\varpi M} . \]
On the other hand, because of \cref{Lemma:KeyInjectivityFormal} (1), multiplication by $f$ is injective on $M / \varpi M$. One concludes $n_i = 0$ for all $i \ge d$. Therefore $n(t) \in M[t]$  and the image of $m(t)$ in $\Gamma(X, F(\ast D))$ is $0$.

(3) If $T$ denotes the cokernel of $\rho$, then  the following exact sequence of $A$-modules is exact:
\[ 0 \too \Gamma(X, F(\ast D)) \too \Gamma(U, F) \too \Gamma(X, T) \too 0.\]
Indeed it suffices to show that the homomorphism $\Gamma(U, F) \to \Gamma(X, T)$ is surjective. By quasi-compactness and (quasi-)separation of $X$,
\[ \rH^1(X, F(\ast D)) = \injlim_{d \in \bbN} \rH^1(X, F(dD)).\]
For $d \in \bbN$, since $X$ is affine and $F(dD)$ coherent, the cohomology group $\rH^1(X, F(dD))$ vanishes for each $d \in \bbN$, whence the claimed surjectivity. Equation \eqref{Eq:MeromorphicSectionsAndSectionsOnOpenSubets} can be rewritten as the following commutative and exact diagram of $A$-modules:
\begin{center}
\begin{tikzcd}
0 \ar[r] & M[t] \ar[r, "1 - ft"] \ar[d] & M [t]  \ar[r] \ar[d] & \Gamma(X, F(\ast D)) \ar[d] \ar[r] & 0 \\
0  \ar[r] & M\{ t \} \ar[r, "1 - ft"]  & M \{ t \} \ar[r] & \Gamma(U, F)   \ar[r]& 0
\end{tikzcd}
\end{center}
Since taking global sections is left exact, statement (2) implies that the right-most vertical arrow is injective. According to the previous claim, the Snake Lemma gives the following short exact sequence:
\[ 0 \too M \{t \} / M[t] \stackrel{1 - ft}{\too}  M \{t \} / M[t] \too \Gamma(X, T) \too 0. \]
Let $\varpi \in R \smallsetminus \{ 0 \}$ be topologically nilpotent.  An $R$-module is a $K$-vector space if and only if multiplication by $\varpi$ is bijective. Because of the previous short exact sequence, it suffices to show that $M \{ t \} / M[t]$ is a $K$-vector space.  In order to do so, for $m(t) \in M \{ t \}$, write $m(t) = \sum_{i = 0}^\infty m_i t^i$ with $m_i \in M$ such that $m_i \to 0$ as $i \to \infty$.

\medskip

(Injective.) Suppose $\varpi m(t) \in M[t]$. Let $d \in \bbN$ be its degree. For $i > d$, one has $ \varpi m_i = 0$, thus  $m_i = 0$ because $F$ is flat over $R$. In other words, $m(t)\in  M[t]$.

\medskip

(Surjective.) Because of the convergence of $m(t)$, there is $d$ such that, for $i \ge d$, $m_i$ is divisible by $\varpi$. For $i \ge d$ let $\tilde{m}_i := m_i / \varpi$ and $\tilde{m}(t) = \sum_{i = d}^\infty \tilde{m}_i t^i$. Then,
\[ m(t) \equiv \varpi \tilde{m}(t) \pmod{M[t]},\]
which concludes the proof.

\medskip

(1) Because of statement (2), if $\alpha$ is injective, then $\beta$ is. In order to prove that $\alpha$ in injective, let $m \in M$ and $n(t) \in M[t]$ such that $m = (1 - ft) n(t)$. Write $n (t) = \sum_{i = 0}^d n_i t^i$ with $n_i \in M$. By comparing the Taylor expansions, one obtains $m = n_0$, $n_i = f n_{i-1}$ for $i = 1, \dots, d$ and $f n_d = 0$. Since multiplication by $f$ on $M$ is injective, one obtains $n_0 = 0$, thus $m = 0$. This proves that $\alpha$ is injective.

\medskip

Recall that an $R$-module is flat if and only if it is torsion-free. In view of this, since the natural map $\coker \alpha \to \coker \beta$ is injective by (2), it suffices to prove the statement for $\beta$. In order to do so, suppose there are $m(t) \in M\{t \}$ and $\varpi \in R \smallsetminus \{ 0 \}$ such that $\varpi m(t)$ belongs to $(1 - ft) M\{t \} + M$. Let $n(t) \in M\{ t \}$ and $p \in M$ be such that $\varpi m(t) = p + (1 - ft) n(t)$. Write $m(t) = \sum_{i = 0}^\infty m_i t^i$ and $n(t) = \sum_{i = 0}^\infty n_i t^i$ with $m_i, n_i \in M$. By comparing Taylor expansions, $\varpi m_0 = p + n_0 $ and, for $i \ge 1$, $\varpi m_i = n_i - f n_{i-1}$. In other words, modulo $\varpi M$, the following congruences hold:
\begin{align*}
n_0 \equiv - p , &&
f n_{i-1}  \equiv n_i, \quad (i \ge 1).
\end{align*}
Moreover, by convergence, there is $i_0 \in \bbN$ such that $n_i$ is divisible by $\varpi$. On the other hand, because of \cref{Lemma:KeyInjectivityFormal} (1), multiplication by $f$ is injective on $M / \varpi M$. It follows that $n_i$ is divisible by $\varpi$ for all $i \in \bbN$, thus $p$ is also divisible by $\varpi$. Set $\tilde{p} = p / \varpi$, $\tilde{n}_i = n_i / \varpi$ and $\tilde{n}(t) = \sum_{i = 0}^{\infty} \tilde{n}_i t^i$. (Note that this makes sense because $M$ is torsion free). Since the $R$-module $M\{ t \}$ is torsion-free,
\[ m(t) = \tilde{p} + (1 - ft) \tilde{n}(t),\]
that is, the image of $m(t)$ in the cokernel of $\beta$ is $0$.

(4) Statement (3) implies $\| \rho(f) \|_{\Gamma(U, F)} = \| f \|_{\Gamma(X, F)}$ for each $f \in \Gamma(X, F)$. According to \cref{Lemma:CompletenessGlobalSections}, $\rho$ is an isometric embedding between complete metric spaces. It follows that the image is closed.
\end{proof}

\begin{proof}[{Proof of \cref{Thm:DensityMeromorphicSectionsModel}}] Let $j \colon U \to X$ be the open immersion of $U$ in $X$. Note that the situation here is no longer local, so $X$ cannot be taken to be affine, contrarily to what done in the proof of the previous lemma. In particular, one needs an argument finer than simply approximating power series by polynomials. To do so,
\cref{Lemma:MeromorphicInjectsInHolomorphicFormal} implies that the natural map $F(\ast D) \to j_\ast j^\ast F$ is injective. By taking global sections, one sees that the map $\rho \colon \Gamma(X, F(\ast D)) \to \Gamma(U, F)$ is injective. Let $\varpi \in R \smallsetminus\{ 0\}$ be topologically nilpotent. For an $R$-scheme $S$, let $S_n$ denote its base-change to $R / \varpi^n R$. Let $s \colon \cO_X \to \cO_X(D)$ be the canonical section. In order to show that the image of $\rho$ is dense, one has to prove that given $f \in \Gamma(U, F)$ and an integer $n \in \bbN$, there are a non-negative integer $d \in \bbN$ and $g \in \Gamma(X, F(dD))$ such that $f_{\rvert U_n} = g_{\rvert U_n}$. Let $f \in \Gamma(U, F)$ and $n \in \bbN$. Because of \cref{Lemma:KeyInjectivityFormal} (1), for each $d \in \bbN$, the map $\alpha_{d,n} \colon \Gamma(X_n, F(d D)) \to \Gamma(U_n, F)$ is injective. Let $d_n$ be the smallest integer $d$ such that there is a section $g \in \Gamma(X_n, F(dD))$ mapping to $f_{\rvert U_n}$, and let $f_n := \alpha_{d_n,n}^{-1}(f_{\rvert U_n})$. For non-negative integers $m \ge n$, the commutativity of the diagram
\begin{center}
\begin{tikzcd}
\Gamma(X_m, F(d_m D)) \ar[r, "\alpha_{d, m}"] \ar[d, "\cdot_{\rvert X_n}"] & \Gamma(U_m, F) \ar[d, "\cdot_{\rvert U_n}"]\\
\Gamma(X_n, F(d_m D)) \ar[r, "\alpha_{d, n}"] & \Gamma(U_n, F)
\end{tikzcd}
\end{center}
implies $d_m \ge d_n$. The map $s^{d_m - d_n} \colon \Gamma(X_n, F(d_n D)) \to \Gamma(X_n, F(d_m D))$ is seen to be injective by applying \cref{Lemma:KeyInjectivityFormal} (1) with $d = d_n$ and $i = d_m - d_n$. It follows that
\[s^{d_m - d_n} f_n = f_{m \rvert X_n}.\]

\begin{claim} For  $f \in \Gamma(U, F)$ and $m,n \in \bbN$ with  $m \ge n$, there is $g_m \in \Gamma(X_m, F(d_n D))$ such that $g_{m \rvert X_n} = f_n$.
\end{claim}

\begin{proof}[Proof of the claim] Let $i := d_m - d_n$. Because of the equality $f_{m \rvert X_n}= s^i f_n$, the image of the section $s^i f_n$ in $\Gamma(X_{n}, F(d_mD))/ \Gamma(X_{m}, F(d_mD))$ is zero. However, by \cref{Lemma:KeyInjectivityFormal}, the map 
\[ 
s^i \colon \frac{\Gamma(X_{n}, F(d_nD))}{\Gamma(X_{m}, F(d_nD))} \too
\frac{\Gamma(X_{n}, F(d_mD))}{\Gamma(X_{m}, F(d_mD))},
\]
is injective. Therefore, the image of $f_n$ in $\Gamma(X_{n}, F(d_nD))/ \Gamma(X_{m}, F(d_nD))$ is also zero. In other words, there is $g_m \in \Gamma(X_{m}, F(d_nD))$ such that $g_{m \rvert X_n} = f_n$.
\end{proof}

By hypothesis there are an affine admissible formal $R$-scheme $Y$ and proper morphism of formal $R$-schemes $\sigma \colon X \to Y$. Since $\sigma$ is proper and $Y$ is affine, the projective system $(\Gamma(X_m, F(d_nD)))_{m \ge n}$ satisfies the Mittlag-Leffler condition. In the discretely valued case, this is one the key ingredients of the proof of the formal GAGA theorem in \cite[Corollaire 4.1.7]{EGAIII1} (see also \cite[8.2.7, p. 191]{FGAExplained}); the general case can be found in \cite[Corollaire 2.11.7]{AbbesEGR} and \cite[proposition I.11.3.3 and Lemma 0.8.8.5 (2)]{FujiwaraKato}. For $m \ge n$ let \[\rho_m \colon \Gamma(X_m, F(d_nD)) \too \Gamma(X_n, F(d_nD))\] be the restriction map. According to the claim, the set $E_m = \rho_m^{-1}(f_n)$ is non-empty, so that the projective limit of the projective system $(E_m)_{m \ge n}$ is non-empty \cite[\href{https://stacks.math.columbia.edu/tag/0597}{lemma 0597}]{stacks-project}. In other words, there is $g_\infty \in \projlim_{m \in \bbN} \Gamma(X_m, F(d_nD))$ such that $g_{\infty \rvert X_n} = f_n$. To conclude, let $E$ be a flat coherent $\cO_X$-module. The comparison theorem in formal geometry (see \cite[Corollaire 4.1.7]{EGAIII1} or \cite[theorem 8.2.6.1]{FGAExplained} in the Noetherian setting, or \cite[theorem 0.9.2.1]{FujiwaraKato} in the general case) states that the natural map
\[ \phi_E \colon \Gamma(X, E) \too \projlim_{m \ge n} \Gamma(X_m, E),\]
is an isomorphism. One concludes the proof by taking $g$ to be the preimage of $g_\infty$ under the above isomorphism for $E := F(d_nD)$.
\end{proof}

\section{A Remmert factorization theorem with boundary}

Let $K$ be a complete non-trivially valued  non-Archimedean field. The analogue of \cref{Thm:DensityMeromorphicSectionsModel} in the context of analytic spaces is naturally stated for holomorphically convex spaces. In absence of boundary these spaces have been characterized in \cite[theorem 6.1]{notions} as being proper over boundaryless Stein spaces---this is what is called the Remmert factorization theorem in complex analysis. However affinoid spaces do have boundary and in order to include them we present here a version of Remmert factorization for spaces with boundary.

\subsection{Reminder on Stein spaces} \label{sec:SteinSpaces} We begin with a reminder on Stein spaces in the non-Archimedean framework, see also \cite{notions}.~Recall that a $K$-analytic space~$X$ is said to \emph{holomorphically separable} if, for $x, x' \in X$ distinct, there is $f \in \Gamma(X, \cO_X)$ such that $|f(x)| \neq |f(x')|$. It is said \emph{holomorphically convex} if, for each  $C \subset X$ compact, the holomorphically convex hull of $C$,
\[ \hat{C}_X := \{ x \in X : |f(x)| \le \| f \|_C, f \in \Gamma(X, \cO_X)\}\]
 is compact, where $\| f \|_C =  \sup_{x \in C} |f(x)|$.  An \emph{affinoid Stein exhaustion} of $X$ is a G-cover $\{ X_i \}_{i \in \bbN}$ of $X$ made of affinoid domains such that, for $i \in \bbN$, $X_i$ is a Weierstrass domain in $X_{i+1}$.
Here we say that $X$ is  \emph{Stein} if it admits an affinoid Stein exhaustion. This is what is called being being W-exhausted by affinoid domains in \cite{notions}. A $K$-analytic space Stein $X$ is holomorphically separable and holomorphically convex by \cite[theorem 1.1]{notions}.  Let $X$ be a $K$-analytic space admitting an affinoid Stein exhaustion $\{ X_i \}_{i \in \bbN}$. Then, for each coherent $\cO_X$-module $F$, 
\begin{enumerate}
\item for all $i \in \bbN$, the restriction map $\Gamma(X, F) \to \Gamma(X_i, F)$ has dense image;
\item for all $q \ge 1$, the cohomology group $\rH^q(X, F)$ vanishes.
\end{enumerate}
Given a $K$-affinoid space $S$, closed analytic subspaces of $\bbA^{n, \an}_K \times_K S$ are Stein. Conversely, let $X$ be a Stein $K$-analytic space without boundary such that
\[
\sup_{x \in X} \dim_{\cH(x)} \Omega_{X, x} \otimes_{\cO_{X, x}} \cH(x) < +\infty.
\] Then, there exist $n \in \bbN$ and a closed embedding $X \into \bbA^{n, \an}_K$. In particular for a $K$-scheme of finite type $X$ the $K$-analytic space $X^\an$ is Stein if and only if there is a closed embedding closed embedding $X^\an \into \bbA^{n, \an}_K$ for some $n \in \bbN$.

\subsection{Statement} Let $\pi \colon X \to Y$ be a morphism of $K$-analytic spaces. A \emph{Remmert factorization} of $\pi$ is a triple $(S, \sigma, \tilde{\pi})$ made of a Stein $K$-analytic space  $S$ and morphisms $\sigma \colon X \to S$, $\tilde{\pi} \colon S \to Y$ such that $\pi = \tilde{\pi} \circ \sigma$ with the following properties:
\begin{itemize}
\item[RF$_1$.]  the morphism $\sigma$ is proper, surjective, with geometrically connected fibers, and the homomorphism $\sigma^\sharp \colon \cO_S \to \sigma_\ast \cO_X$ induced by $\sigma$ an isomorphism;
\item[RF$_2$.] given a $K$-analytic space $T$ and a morphism of $K$-analytic spaces $\tau \colon X \to T$ set-theoretically constant on the fibers of~$\sigma$, there is a unique morphism of $K$-analytic spaces $\tilde{\tau} \colon S \to T$ such that $\tau = \tilde{\tau} \circ \sigma$.
\end{itemize} 

\begin{lemma} \label{Lemma:RemmertFactorizationFactorsHolSep} Let $\pi \colon X \to Y$ be a morphism of $K$-analytic spaces. Assume there is a Remmert factorization $(S, \sigma, \tilde{\pi})$ of $\pi$. Let $T$ be a holomorphically separable $K$-analytic space, and $\tau \colon X \to T$ a morphism. Then, there exists a unique morphism $\tilde{\tau} \colon S \to T$ such that $\tau = \tilde{\tau} \circ \sigma$.
\end{lemma}

\begin{proof} Let $x, x' \in X$ be such that $\sigma(x) = \sigma(x')$. In view of property RF$_2$, it suffices to prove that $\tau(x)$ and $\tau(x')$ coincide. Suppose, by contradiction, that this is not the case. Since $T$ is holomorphically separable, there is a $K$-analytic function $f$ on $T$ such that $|f(\tau(x))| \neq |f(\tau(x'))|$. Because of the isomorphism $\cO_S \to \sigma_\ast \cO_X$ we have  $\tau^\ast f = \sigma^\ast g$  for some $K$-analytic function $g$ on $S$. Then $|g(\sigma(x))| \neq |g(\sigma(x'))|$, whereas $\sigma(x) = \sigma(x')$. Contradiction.
\end{proof}

As usual, if it exists, a Remmert factorization is unique up to a unique isomorphism.

\begin{theorem} \label{thm:Remmert} Let $\pi \colon X \to Y$ be a morphism of $K$-analytic spaces where $X$ is separated, holomorphically convex and countable at infinity, $Y$ is Stein, and $\pi$ is without boundary. 

Then, the Remmert factorization $(S, \sigma, \tilde{\pi})$ of $\pi$ exists and the morphism $\tilde{\pi}$ is without boundary.
\end{theorem}

As an immediate consequence of \cref{thm:Remmert} one obtains:

\begin{corollary} \label{Cor:EquivalentNotionsStein} Let $\pi \colon X\to S$ a morphism without boundary of $K$-analytic spaces with $X$ is separated and countable at infinity, and $S$ Stein. Then  $X$ is Stein if and only if it is holomorphically separable and holomorphically convex.
\end{corollary}

In the rest of this section is devoted to the proof of \cref{thm:Remmert}.

\subsection{Proofs} The proof of \cref{thm:Remmert} follows closely that of \cite[theorem 6.1]{notions}, in turn inspired by the proof of the complex Remmert factorization theorem. In order to control boundary, we will need the following:

\begin{lemma} \label{Lemma:WithoutBoundary} Let $\pi \colon X \to S$ and $f \colon X \to Y$ be morphisms between $K$-analytic spaces, and let $V \subset X$ and $W \subset Y \times_K S$ be open subsets. Consider the morphism of $K$-analytic spaces $F = (\pi , f) \colon X \to Y \times_K S$.  If $X$ is separated and $\pi$ is without boundary, then the induced morphism $F \colon V \cap F^{-1}(W) \to W$ is without boundary.
\end{lemma}

\begin{proof} This follows by using repeatedly \cite[proposition 3.1.3 (iii)]{Berkovich90}. The morphism $\pi \times \id \colon X \times_K Y \to S \times_K Y$ is without boundary. Since $X$ is separated the graph $(\id, f) \colon X \to X \times_K Y$ is a closed immersion, hence without boundary. Thus the composite morphism $F := (\pi, f)  = (\pi \times \id) \circ (\id, f) \colon X \to S \times_K Y$ is without boundary hence so is $F \colon F^{-1}(W) \to W$. In particular restriction of $\pi$ to $V \cap F^{-1}(W)$ is without boundary too  because an open immersion is without boundary.
\end{proof}

We recall the statement of the Stein factorization theorem \cite[proposition 3.3.7]{Berkovich90} with a slight improvement. Let $f \colon X \to Y$ be a proper morphism of $K$-analytic spaces. Then, the sheaf of $\cO_Y$-algebras $\cA := f_\ast \cO_X$ is coherent. Moreover, there is a (unique up to a unique isomorphism) $K$-analytic space $\tilde{Y}$ together with a finite morphism $\pi \colon \tilde{Y} \to Y$, called the \emph{Stein factorization of $f$}, representing the functor associating to a morphism $g \colon S \to Y$ of $K$-analytic spaces the set of morphisms of $\cO_S$-algebras $g^\ast \cA \to \cO_S$. Concretely, when $Y= \cM(B)$ is affinoid, the $K$-analytic space $\tilde{Y}$ is the Banach spectrum of the finite $B$-algebra $A:= \Gamma(X, \cO_X)$. Let $\tilde{f} \colon X \to \tilde{Y}$ be the morphism corresponding to the morphism $f^\ast f_\ast \cO_X \to \cO_X$. 

\begin{lemma} \label{lemma:SteinFactorization} With the above notation, we have the following properties:
\begin{enumerate}
\item $\tilde{f}$ is proper and surjective, $\pi$ is finite and $f = \pi \circ \tilde{f}$;
\item the morphism $\cO_{\tilde{Y}} \to \tilde{f}_\ast \cO_X$ is an isomorphism;
\item the fibers of $\tilde{f}$ are geometrically connected;
\item given a morphism $g \colon X \to Z$ which is constant on the connected components of the fibers of $f$, there is a unique morphism $\tilde{g} \colon \tilde{Y} \to Z$ such that $g = \tilde{g} \circ \tilde{f}$;
\item the construction of the Stein factorization is compatible with extension of scalars: given a complete valued extension $K'$ of $K$, the Stein factorization of $X \times_K K' \to Y \times_K K'$ is $\tilde{Y} \times_K K' \to Y \times_K K'$. 
\end{enumerate}
\end{lemma}

\begin{proof} Properties (1), (2) and (4) are proved in \cite[proposition 3.3.7]{Berkovich90}. In \emph{loc.cit.} fibers of $\tilde{f}$ are only stated to be connected. We will deduce (3) from  this and the compatibility in (5). 

To prove (5), let $K'$ be a complete valued extension of $K$. Let $X'$, $Y'$ and $\tilde{Y}'$ be the $K'$-analytic spaces obtained respectively from $X$, $Y$ and $\tilde{Y}$ extending scalars to $K'$, and let $f' \colon X' \to Y'$, $\pi \colon \tilde{Y}' \to Y'$ and $\tilde{f} \colon X' \to \tilde{Y}'$ the  morphisms induced by $f$, $\pi$ and $\tilde{f}$ respectively. It suffices to show that the natural morphism $q^\ast f_\ast \cO_X \to f'_\ast \cO_{X'}$
is an isomorphism, where  $q \colon Y' \to Y$ is the base-change morphism. This boils down to prove, when $Y = \cM(B)$ is affinoid, that the natural homomorphism of Banach $B'$-algebras $\phi \colon \Gamma(X, \cO_X) \hotimes_B B' \to \Gamma(X', \cO_{X'})$ is an isomorphism, where $B' = B \hotimes_K K'$. This follows from the natural isomorphism $\Gamma(X', \cO_{X'}) \iso \Gamma(X, \cO_X) \hotimes_K K'$, see for instance \cite[theorem A.5]{notions}. This concludes the proof of (5).

To prove (3), pick a point $y \in \tilde{Y}$ and let $K'$ be the completion of an algebraic closure of the completed residue field $\cH(y)$ at $y$. The point $y$ defines a natural $K'$-point $y'$ of $\tilde{Y}'$. The base-change to $K'$ of $\tilde{f}^{-1}(y)$ is the fiber $\tilde{f}'$ at $y'$, thus we need to show that this is connected. By (5) the Stein factorization of $f'$ is $\pi'$, thus the fiber of $\tilde{f}'$ at $y'$ is connected by \cite[proposition 3.3.7]{Berkovich90}.
\end{proof}

We first prove \cref{thm:Remmert} when the target is affinoid:

\begin{lemma}\label{Prop:RemmertAffinoidBase} Let $\pi \colon X \to Y$ be a morphism of $K$-analytic spaces. Suppose that the $K$-analytic space $X$ is separated, holomorphically convex and countable at infinity, the $K$-analytic space $Y$ is affinoid, and the morphism $\pi$ is without boundary.  Then, the Remmert factorization $(S, \sigma, \tilde{\pi})$ of $\pi$ exists and the morphism $\tilde{\pi}$ is without boundary.
\end{lemma}

\begin{proof} Let $\{ C_i \}_{i \ge 0}$ be an exhaustion of $X$ by compact subsets such that, for each~$i$, $C_{i}$ is the closure of its interior and $C_i$ is contained in the interior of $C_{i + 1}$. Set 
\[n_{-1} = 0 ,\qquad X_{-1} =  U_{-1} = W_{-1} = \emptyset.\]
For $i \ge 0$, we construct inductively an integer $n_i \ge 0$, a relatively compact open subset $U_i$ of $X$, a compact analytic domain~$X_i$ of $U_i$, an open subset $W_i$ of $\bbA^{n_i, \an}_Y$ and $K$-analytic functions $f_{n_{i-1} + 1},  \dots, f_{n_i}$ on~$X$ such that:
\begin{enumerate}
\item the morphism $(f_1, \dots, f_{n_i}, \pi) \colon X \to \bbA^{n_i, \an}_Y$ induces a proper morphism $F_i \colon U_i \to W_i$; 
\item $\bbD_i := \bbD(\| f_1\|_{C_i \cup \overline{U}_{i-1} }, \dots, \| f_{n_i}\|_{C_i \cup \overline{U}_{i-1}})  \times_K Y$ is contained in $W_i$;
\item $X_i := F_i^{-1}(\bbD_i) \cap U_i$ contains the holomorphically convex hull of the compact subset $C_i \cup \overline{U}_{i-1}$.
\end{enumerate}

Set $C'_i := C_i \cup \overline{U}_{i - 1}$ and consider a relatively compact open subset $V_i \subset X$ containing the holomorphically convex hull of $C'_i$ in $X$ (which is compact since~$X$ is holomorphically convex). Let $\partial V_i$ be the topological boundary of $V_i$ in $X$. Since $\partial V_i$ is compact, there are $K$-analytic functions $f_{n_{i -1} + 1}, \dots, f_{n_{i}}$ on $X$ such that, for all $x \in \partial V_i$,
\[ \max_{j = 1, \dots, n_i} \frac{|f_j(x)|}{\| f_j \|_{C'_i}} > 1.\]
Set $W_i := \bbA^{n_i, \an}_Y \smallsetminus F_i(\partial V_i)$ and $U_i = V_i \cap F_i^{-1}(W_i)$. Property (1) holds because the morphism $F_i \colon U_i \to W_i$ induced by $(f_1, \dots, f_{n_i}, \pi) \colon X \to \bbA^{n_i, \an}_Y$ is without boundary (\cref{Lemma:WithoutBoundary}) and $F_i$ is topologically proper by construction, thus proper \cite[p. 50]{Berkovich90}. Properties (2) and (3) hold by construction. Let $X_i$ denote $\bbD_i \times_{W_i} U_i$.  The morphism $X_i \to \bbD_i$ induced by $F_i$ is proper. Let $\tilde{F}_i \colon S_i \to \bbD_i$ be the Stein factorization of the proper morphism $F_i$ and $\sigma_i \colon X_i \to S_i$ the induced morphism. The morphism $\tilde{S}_i$ is finite, hence $S_i$ is affinoid, thus Stein. It follows from \cref{lemma:SteinFactorization} that $(S_i, \sigma_i, \tilde{F}_i)$ is the Remmert factorization of $F_i$.  In particular, for each $i \in \bbN$, the Banach $K$-algebra $\cO(X_i)$ is affinoid. Moreover, for $i \in \bbN$, the restriction map $\cO(X_{i+1}) \to \cO(X_{i})$ is a bounded homomorphism of $K$-affinoid algebras and it induces a morphism of $K$-affinoid spaces $\epsilon_i \colon S_i \to S_{i+1}$. 

\begin{claim} The map $\epsilon_i$ identifies $S_i$ with a Weierstrass domain of $S_{i+1}$. \end{claim}

The proof of the claim is by any means analogous to the that of \cite[claim 6.3]{notions} and we do not repeat it here. Consider the $K$-analytic space $S = \bigcup_{i \in \bbN} S_i$ and the morphism $\sigma \colon X \to S$ obtained by glueing the morphisms $\sigma_i$. The $K$-analytic space $S$ is Stein by definition. Properties (RF$_1$) and (RF$_2$) are deduced from the properties of $\sigma_i$, for $i \in \bbN$. (Note that, for each $i \in \bbN$, the preimage $\sigma^{-1}(S_i)$ of $S_i$ is $X_i$.) The morphism $\tilde{\pi}$ is seen to be without boundary because $\sigma$ is surjective and the morphism $\pi$ is without boundary \cite[proposition 3.1.3 (ii)]{Berkovich90}.
\end{proof}

Let $\Hom_K(Y, X)$ denote the set of morphisms between some $K$-analytic spaces $X$ and $Y$. If $X$ and $Y$ are countable at infinity, let $\Hom_{K\textup{-alg}}(\cO(X), \cO(Y))$ be the set of $K$-algebra homomorphisms $\cO(X) \to \cO(Y)$ that are continuous with respect the natural Fr\'echet topologies on $\cO(X)$ and $\cO(Y)$.

\begin{proposition} \label{Prop:YonedaStein} Let $X$ and $Y$ be $K$-analytic spaces countable at infinity with $X$ Stein. Then the following map is bijective:
\[ \Hom_K(Y, X)  \too  \Hom_{K\textup{-alg}}(\cO(X), \cO(Y)),\quad f \colon Y \to X \longmapsto  f^\sharp \colon \cO(X) \to \cO(Y). \]
\end{proposition}

\begin{proof} By the very definition of morphisms between $K$-analytic spaces, the statements holds when $X$ is a $K$-affinoid space. 

 (Injective.) Let $f_1, f_2 \colon Y \to X$ be morphisms of $K$-analytic spaces. Let $\{ Y_i\}_{i \in \bbN}$ be a G-cover of $Y$ by compact analytic domains such $Y_i \subset Y_{i + 1}$ for all $i \in \bbN$. Let $\{ X_i \}_{i \in \bbN}$ be an affinoid Stein exhaustion of $X$. Since $Y_i$ is compact for any $i \in \bbN$, we may replace $\{ X_i \}_{i \in \bbN}$ by a subsequence such that $f_\alpha(Y_i) \subset X_i$ for any $i \in \bbN$ and $\alpha = 1, 2$. The morphism $f_\alpha$ then induces a morphism $f_{\alpha, i} \colon Y_i \to X_i$. Suppose $f_1^\sharp = f_2^\sharp$. For $i \in \bbN$, the $K$-algebra $\cO(X)$ is dense in the affinoid $K$-algebra $\cO(X_i)$. Therefore, the continuous homomorphisms \[ f_{1, i}^\sharp, f_{2, i}^\sharp \colon \cO(X_i) \too \cO(Y_i),\] induced by $f_{1, i}$ and $f_{2, i}$ coincide. By the affinoid case, the morphisms $f_{1 , i}$ and $f_{2, i}$ are the same. Since $i \in \bbN$ is are arbitrary, one deduces $f_1 = f_2$.
 
 (Surjective.) Let $\phi \colon \cO(X) \to \cO(Y)$ be a continuous $K$-algebra homomorphism.  Let $\{ Y_i\}_{i \in \bbN}$ be a G-cover of $Y$ by compact analytic domains such $Y_i \subset Y_{i + 1}$ for all $i \in \bbN$. For $i \in \bbN$, let $\| \cdot \|_{Y, i}$ be a norm on the Banach $K$-algebra $\cO(Y_i)$ defining its topology. By the definition of the topology on $\cO(X)$ and $\cO(Y)$ and the continuity of $\phi$, for each $i \in \bbN$, there is a real number $R_i > 0$, a compact analytic domain $X_i$ of $X$ and a norm $\| \cdot \|_{X, i}$ defining the topology of $\cO(X_i)$ such that, for all $f \in \cO(X)$,
\[ \| \phi(f)\|_{Y, i} \le R_i \| f\|_{X, i}. \]
Arguing by induction and by possibly enlarging $X_i$, we may assume that the collection $\{ X_i\}_{i \in \bbN}$ is an affinoid Stein exhaustion. By density of $\cO(X)$ in $\cO(X_i)$, the homomorphism $\phi$ induces a continuous homomorphism of Banach $K$-algebras $\phi_i \colon \cO(X_i) \to \cO(Y_i)$. By the affinoid case, for each $i \in \bbN$, there is a unique morphism of $K$-analytic spaces $f_i \colon Y_i \to X_i$ such that the induced homomorphism $f_i^\sharp \colon \cO(X_i) \to \cO(Y_i)$ is $\phi$. Moreover, by density of $\cO(X)$ in $\cO(X_i)$ and $\cO(X_{i+1})$, the following diagram is commutative:
\begin{center}
\begin{tikzcd}
Y_i \ar[r, "f_i"] \ar[d] & X_i \ar[d]\\
Y_{i + 1} \ar[r, "f_{i+1}"] & X_{i + 1},
\end{tikzcd}
\end{center}
where the vertical arrows are the inclusions. Therefore the morphisms $f_i$ glue to a morphism $f \colon Y \to X$ such that $f^\sharp = \phi$.
\end{proof}

\begin{lemma} \label{Lemma:PreimageWeierstrssDomainWExhausted}  Let $\pi \colon X \to Y$ be a morphism of $K$-analytic spaces with $X$ Stein and $Y$ affinoid. Let $Y' \subset Y$ be a  Weierstrass domain and $X':= \pi^{-1}(Y')$. Then,
\begin{enumerate}
\item the $K$-analytic space $X'$ is Stein;
\item for any affinoid Stein exhaustion $\{ X'_i\}_{i \in \bbN}$ of $X'$, there is an affinoid Stein exhaustion $\{ X_i \}_{i \in \bbN}$ of $X$ such that, for each $i \in \bbN$, $X'_i$ is a Weierstrass domain of $X_i$.
\end{enumerate}
\end{lemma}

\begin{proof} By definition of a Weierstrass domain there are   $f_1, \dots, f_n \in \Gamma(Y,\cO_Y)$ and real numbers $r_1, \dots, r_n >0$ such that $Y' = \{ y \in Y : |f_j(y)| \le r_j , j = 1, \dots, n\}$. Given an affinoid Stein exhaustion $\{ W_i \}_{i \in \bbN}$ of $X$ define, for $i\in \bbN$,
 \[ W'_i = \{ x \in W_i : |f_j(\pi(x))| \le r_j , j = 1, \dots, n\} . \] 

(1) The collection $\{ W'_{i} \}_{i \in \bbN}$ is an affinoid Stein exhaustion of $X'$.

(2) Note that, for each $i \in \bbN$, the $K$-affinoid space $W'_i$ is a Weierstrass domain of $W_i$. Let $\{ X_i' \}_{i \in \bbN}$ be affinoid Stein exhaustion of $X'$. 
Then, there is an increasing function $\tau \colon \bbN \to \bbN$ such that, for each $i \in \bbN$, $X'_i$ is a Weierstrass domain in $W'_{\tau(i)}$. For the latter note that it suffices that $X_i'$ is contained in $W'_{\tau(i)}$ because the restriction map $\cO(X')\to \cO(X'_i)$ has a dense image, thus \emph{a fortiori} $\cO(W'_i)\to \cO(X'_i)$ has a dense image.
In particular, the affinoid Stein exhaustion $\{ W_{\tau(i)}\}_{i \in \bbN}$ does the job.
\end{proof}

\begin{lemma} \label{Prop:RemmertFactorizationWeierstrassDomain} Under the hypotheses of \cref{Prop:RemmertAffinoidBase} let $(S, \sigma, \tilde{\pi})$ be  the Remmert factorization  of $\pi$ and $Y' \subset Y$ a Weierstrass domain.   Then $(S \times_Y Y', \sigma_{\rvert X \times_Y Y'}, \tilde{\pi}_{\rvert S \times_Y Y'})$ is the Remmert factorization of $\pi_{\rvert X \times_X Y'}$.
\end{lemma}

\begin{proof} By definition of Weierstrass domain, there are $f_1, \dots, f_n \in \Gamma(Y, \cO_Y)$ and $r_1, \dots, r_n > 0$ such that $Y'= \{ y \in Y : |f_i(y)| \le r_i \textup{ for all } i = 1, \dots, r\}$. In particular,
\[ X' := X \times_Y Y' = \{ x \in X : |f_i(\pi(x))| \le r_i \textup{ for all } i = 1, \dots, r \} .\]
The $K$-analytic space $X'$ is holomorphically convex. Indeed, given a compact subset of $C \subset X'$, it follows from the above description that $\hat{C}_X$ is contained in $X'$. On on the other hand, $\hat{C}_X$ is compact because $X$ is holomorphically convex and $\hat{C}_{X'}$ is closed in $\hat{C}_{X} \cap X' = \hat{C}_X$, thus $\hat{C}_{X'}$ is compact.
The morphism $\pi' \colon X' \to Y'$ is without boundary and $Y'$ is affinoid. Therefore, by \cref{Prop:RemmertAffinoidBase}, the Remmert factorization $(S', \sigma', \tilde{\pi}')$ of $\pi'$ exists. By \cref{Lemma:RemmertFactorizationFactorsHolSep}, there is a unique morphism $\phi \colon S' \to S$ such that $\sigma_{\rvert X'}= \phi \circ \sigma'$. By the Open Mapping theorem of Fr\'echet spaces \cite[corollary 8.7]{SchneiderFunctionalAnalysis} the continuous isomorphisms $\sigma_{\rvert X'}^\sharp \colon \cO(S \times_Y Y') \to \cO(X')$ and $\sigma'^\sharp \colon \cO(S') \to \cO(X')$
induced respectively by the proper morphisms $\sigma_{\rvert X'}$ and $\sigma'$ are homeomomorphisms. It follows that the homomorphism $\phi^\sharp \colon \cO(S \times_Y Y') \to \cO(S')$ induced by $\phi$ is an isomorphism of Fr\'echet algebras. Since $S \times_Y Y'$ and $S'$ are Stein (\cref{Lemma:PreimageWeierstrssDomainWExhausted}), \cref{Prop:YonedaStein} implies that $\phi$ is an isomorphism.
\end{proof}

\begin{proof}[{Proof of \cref{thm:Remmert}}] Let $\{ Y_i \}_{i \in \bbN}$ be an affinoid Stein exhaustion of $Y$.  For each $i \in \bbN$, set $X_i := \pi^{-1}(Y_i)$. By \cref{Prop:RemmertAffinoidBase}, the Remmert factorization $(S_i, \sigma_i, \tilde{\pi}_i)$ of $\pi \colon X_i \to Y_i$ exists. Moreover, by \cref{Prop:RemmertFactorizationWeierstrassDomain}, the Remmert factorization $(S_i, \sigma_i, \tilde{\pi}_i)$ is canonically isomorphic to $(S'_i, \sigma_{i + 1 \rvert S'_i}, \tilde{\pi}_{i + 1 \rvert S'_i})$, where $S'_i := \tilde{\pi}^{-1}(Y_{i+1})$. In what follows, we identify $S_i$ with $S_i'$ through the preceding canonical isomorphism. Thanks to \cref{Lemma:PreimageWeierstrssDomainWExhausted}, construct by induction for each $i \in \bbN$, a W-exhaustion by Weierstrass domain $\{ S_{ij} \}_{j \in \bbN}$ of $S_i$ such that, for $i' \le i$ and $j \in \bbN$, the affinoid space $S_{i'j}$ is a Weierstrass domain in $S_{ij}$. It follows that $\{ S_{ii}\}_{i \in \bbN}$ is an affinoid Stein exhaustion of the $K$-analytic space $S = \bigcup_{i \in \bbN} S_i$. Properties (RF$_1$) and (RF$_2$) are deduced from those of $\sigma_i$, for $i \in \bbN$. Analogously, the morphism $\pi$ is seen to be without boundary.
\end{proof}

\section{Approximating essential singularities by poles: the analytic case}

Let $R$ be the ring of integers of a field $K$ complete with respect to a on-trivial non-Archimedean absolute value.

\subsection{Statements} \label{sec:StatementsDensityPolarSection} Given a coherent sheaf $F$ on a $K$-analytic space $X$ the topology on its global sections  is defined as follows. If $X = \cM(A)$ is the Banach spectrum  of a $K$-affinoid algebra $A$, then the $A$-module $\Gamma(X, F)$ is finitely generated. The topology on $\Gamma(X, F)$ is  induced via some surjective homomorphism $A^n \to \Gamma(X, F)$  of $A$-modules by the one on $A$. It does not depend on the choice of said surjection. When $X$ is arbitrary, the topology on $\Gamma(X, F)$ is the coarsest one for which, given an affinoid domain $V \subset X$, the restriction map $\Gamma(X, F) \to \Gamma(V, F)$ is continuous. If $\cV$ is an affinoid $G$-cover of $X$, the restriction map induces a homeomorphism
\[ \Gamma(X, F) \stackrel{\sim}{\too} \projlim_{V \in \cV} \Gamma(V, F),\]
the target being endowed with the usual topology on the projective limit. If $X$ is reduced and $F = \cO_X$, then the so-defined topology coincides with the topology of uniform convergence over compact subsets of $X$. The main result of this section is the analogue of \cref{Thm:DensityMeromorphicSectionsModel} in the context of analytic spaces. In order to state it we need to construct the sheaf of functions with polar singularities. We begin with the following general lemma, which will be used repeatedly:

\begin{lemma} \label{lemma:InjectiveLimitSheavesGTopology} Let $X$ be a compact analytic space and $(F_i, \phi_{ij})_{i, j \in I}$ a directed system of sheaves of sets for the $G$-topology on $X$. Then, the canonical map is bijective
\[ \injlim_{i \in I} \Gamma(X, F_i) \stackrel{\sim}{\too} \Gamma(X, \injlim_{i \in I} F_i).\]
\end{lemma}

\begin{proof} Strictly speaking, one cannot apply directly \cite[\href{https://stacks.math.columbia.edu/tag/009F}{lemma 009F (4)}]{stacks-project} because the statement there is for a honest topological space, instead of a mere site. However, the proof goes through \emph{verbatim} replacing open subsets by compact analytic domains and open covers by $G$-covers.
\end{proof}

We are now in position to define the sheaf of functions with polar singularities:

\begin{lemma} \label{Lemma:MeromorphicSectionsAnalytic} Let $D$ be an effective Cartier divisor in a $K$-analytic space $X$ and $F$ a coherent sheaf of $\cO_X$-modules. Let $j \colon X\smallsetminus D \to X$ be the open immersion. 
\begin{enumerate}
\item If $X$ is compact, then the maps $F(dD) \to F(\ast D)$ induce an isomorphism
\[ \injlim_{d \in \bbN} \Gamma(X, F(dD)) \stackrel{\sim}{\too} \Gamma(X, F(\ast D)).\]

\item The restriction map $F(\ast D) \to j_\ast j^\ast F$ is injective.
\end{enumerate}
\end{lemma}

\begin{proof} (1) This follows from \cref{lemma:InjectiveLimitSheavesGTopology}. (2) The statement is local on $X$, thus we may suppose $X$ affinoid and $I_D$ is principal. Let $A = \Gamma(X, \cO_X)$, $f \in A$ a generator of $\Gamma(X, I_D)$, and $t \in \Gamma(X, F(\ast D)$. According to (1) the section $t$ comes from one of $F(iD)$ for some $i \in \bbN$ that we  denote again by $t$. Suppose that $t$ vanishes identically on $U = X \smallsetminus D$. Let $Y$ be the support of $t$, that is, the support of the coherent $\cO_X$-submodules $t \cO_X$ of $F$. If $J$ is the annihilator ideal of $t \cO_X$, then $Y = \cM(A/J)$. By assumption $Y$ is set-theoretically contained in $D$, hence $f$ vanishes identically on $Y$. According to \cite[proposition 2.1.4 (i)]{Berkovich90}, there is a positive integer $d \ge 1$ such that $f^d \in J$. In other words, $t$ lies in the kernel of the map $\Gamma(X, F(iD)) \to \Gamma(X, F((i + d)D))$, $s \mapsto f^d s$. Therefore the image of $t$ in $\Gamma(X, F(\ast D))$ vanishes.
\end{proof}

 We are now in position to state the main result.  Recall that an $\cO_X$-module over a ringed space $X$ is said to be \emph{semi-reflexive} if the homomorphism  $F \to F^{\vee\vee}$ is injective. 
 
\begin{theorem} \label{Thm:ComplementDivisorHolomorphicallyConvex} Let $D$ be an effective Cartier divisor in a separated holomorphically convex $K$-analytic space $X$. Assume there is a morphism without boundary $X \to S$ of $K$-analytic spaces with $S$ Stein. Then for any semi-reflexive coherent $\cO_{X}$-module~$F$ the natural map
\[ \Gamma(X, F(\ast D)) \too \Gamma(X \smallsetminus D, F)\]
is injective, and has a dense image.
\end{theorem}

It is sometimes useful to combine \cref{Thm:ComplementDivisorHolomorphicallyConvex} with the following non-density result. To state it, say that an open subset $U$ in a $K$-analytic space $X$ is \emph{scheme-theoretically dense} if the homomorphism of $\cO_X$-modules $\cO_X \to j_\ast \cO_U$ is injective, where $j \colon U \to X$ is the open immersion. If this is the case, for any coherent $\cO_X$-module $F$, the natural morphism $F^\vee \to j_\ast j^\ast F^\vee$ is injective. In particular, if $F$ is semi-reflexive, then $F \to j_\ast j^\ast F$ is injective.

\begin{proposition} \label{Thm:ClosednessImage} Let $Z$ be a closed analytic subspace of a Hausdorff $K$-analytic space $X$.  If $X \smallsetminus Z$ is scheme-theoretically dense, then for any  semi-reflexive coherent $\cO_{X}$-module $F$ the restriction map \[\rho \colon \Gamma(X, F) \too \Gamma(X \smallsetminus Z, F)\] is a homeomorphism onto its image and its image is a closed subset.
\end{proposition}

The rest of this section is devoted to the proof of these results.

\subsection{Extension of semi-reflexive modules} When $X$ is proper \cref{Thm:ComplementDivisorHolomorphicallyConvex} will be deduced from \cref{Thm:DensityMeromorphicSectionsModel} by passing to the generic fiber. However to do so we need to extend modules from the generic fiber to the formal scheme preserving the hypotheses needed to apply \cref{Thm:DensityMeromorphicSectionsModel}. To start with recall that coherent sheaves can be extended from the generic fiber to the formal scheme. More precisely, let $X$ be a quasi-compact admissible formal $R$-scheme and $F_K$ be a coherent $\cO_{X_\eta}$-module. Then there is a coherent $\cO_X$-module $F$ flat over $R$ and an isomorphism $F_\eta \cong F_K$. See \cite[lemma 2.2]{LutkebohmertProperness} in the discretely valued case and \cite[proposition 5.6]{BoschLutkebohmertFormalRigidI}, \cite[proposition 5.6]{AbbesEGR} or \cite[Corollary II.5.3.3]{FujiwaraKato} in the general case. Note that, for a coherent sheaf $F$ on a $X$, the natural homomorphism $ \Gamma(X, F) \otimes_R K \to \Gamma(X_\eta, F_\eta)$ is an isomorphism. The task here is to preserve the semi-reflexivity and the flatness on the Cartier divisor. We start with: 

\begin{lemma} \label{Lemma:AutomaticallySemiReflexive} Let $F$ be a coherent sheaf flat over $R$ on an admissible formal $R$-scheme $X$. If $F_\eta$ is semi-reflexive, then $F$ is semi-reflexive.
\end{lemma}

\begin{proof} The statement is local on $X$, thus one may suppose that $X = \Spf(A)$ is affine, where $A$ is a topologically finitely presented flat $R$-algebra. Let $M := \Gamma(X, F)$. Then $\Gamma(X_\eta, F_\eta) = M \otimes_R K$ and $\Gamma(X_\eta, F_\eta^{\vee\vee}) = (M \otimes_R K)^{\vee \vee}$.  Since $M$ is finitely generated and $K$ is flat over $R$, the natural homomorphism $M^{\vee \vee} \otimes_R K \to (M \otimes_R K)^{\vee \vee}$ is an isomorphism. Consider the following commutative diagram:
\begin{center}
\begin{tikzcd}
M \ar[r] \ar[d, "\alpha", hook] & M^{\vee\vee} \ar[d] \\
M \otimes_R K \ar[r, "\beta", hook] & M^{\vee\vee} \otimes_R K.
\end{tikzcd}
\end{center}
The homomorphism $\alpha$ is injective because $M$ is flat over $R$, while $\beta$ is injective because $F_\eta$ is semi-reflexive. It follows that $M \to M^{\vee\vee}$ is injective.
\end{proof}

\begin{lemma} \label{Lemma:RestrictionOfDualIsFlatFormal} Let $X$ be a quasi-compact admissible formal $R$-scheme. Let $D$ be an \emph{admissible} effective Cartier divisor in $X$. Let $F$ be a coherent $\cO_X$-module. 

Then, $F^\vee$ is $D$-torsion free, and $F^\vee$ and $F^\vee_{\rvert D}$ are flat over $R$.
\end{lemma}

\begin{proof} The statement is local on $X$. Therefore one may assume that $X = \Spf(A)$ is affine, for a topologically finitely presented flat $R$-algebra $A$, and the divisor $D$ is given by the ideal $f A$ for some nonzerodivisor $f \in A$. Let $M := \Gamma(X, F)$. The $A$-module $\Hom_A(M, A)$ is $D$-torsion free because $f$ is a nonzerodivisor in $A$.  In order to show that the $A$-modules $\Hom_A(M, A)$ and $\Hom_A(M, A) \otimes_A A / fA$ are flat, it suffices to show that they have no $\varpi$-torsion, for any $\varpi \in R \smallsetminus \{ 0 \}$. Let $\varpi \in R \smallsetminus \{ 0 \}$. The $A$-module $\Hom_A(M, A)$ is $\varpi$-torsion-free because $\varpi$ is a nonzerodivisor in $A$.  Concerning $\Hom_A(M, A) \otimes_A A / fA$, given a homomorphism of $A$-modules $\phi \colon M \to A$ such that $\varpi \phi = f \psi$ for some $\psi \in \Hom_A(M, A)$, one needs to show that there is $\phi' \in \Hom_A(M, A)$ such that $\phi = f \phi'$. Let $\bar{\phi} \colon M / fM \to A / fA$ be the homomorphism induced by $\phi$. If $\varpi \phi = f \psi$ for some $\psi \in \Hom_A(M, A)$, then $\varpi \bar{\phi}$ is $0$. However, multiplication by $\varpi$ on $A / f A$ is injective because the Cartier divisor $D$ is admissible. Therefore, the homomorphism $\bar{\phi}$ is $0$. Since the multiplication by $f$ on $A$ is injective, this means that $\phi$ can be written as $f \phi'$ for some $\phi' \in \Hom_A(M, A)$.
\end{proof}

\begin{lemma} \label{Lemma:SaturationOfSemiReflexiveFormal} Let $D$ be an effective Cartier divisor in a quasi-compact admissible formal $R$-scheme $X$. Let $F$ be a coherent $\cO_X$-module which is $D$-torsion free, flat over $R$, and such that $F_{\rvert D}$ is flat over $R$. Let $E \subset F$ be a coherent $\cO_X$-submodule. Then, the $\cO_X$-module \[E^{\sat} := \Ker(F \to [F(\ast D) / E(\ast D)]_K)\]
is coherent, $D$-torsion free, flat over $R$, $E^\sat_{\rvert D}$ is flat over $R$, $E_\eta (\ast D_\eta) \cong E_\eta^\sat(\ast D_\eta)$ and $E_{\rvert (X \smallsetminus D)_\eta} \cong E^\sat_{\rvert (X \smallsetminus D)_\eta}$.
\end{lemma}

\begin{proof} The statement is local on $X$. Therefore one may assume that $X = \Spf(A)$ is affine, where $A$ is a topologically finitely presented flat $R$-algebra, and $D$ is a principal Cartier divisor whose ideal is $fA$ for a nonzerodivisor $f \in A$. Let $M := \Gamma(X, E)$, $M' := \Gamma(X, E^\sat)$ and $N := \Gamma(X, F)$. Let $A_K := A \otimes_R K$. With this notation,
\[ M' = \Ker(N \to (N/M)\otimes_A A_K[1/f]).\]
The $A$-module $N$ has no $R$-torsion, thus $M'$ has no $R$-torsion too. It follows that $M'$ is flat over $R$. Analogously, the multiplication by $f$ on $N$ is injective, thus it is injective also on $M'$. By definition, $N/M'$ injects into $(N/M)\otimes_A A_K[1/f]$, hence $N/M'$ is flat over $R$. It follows from \cref{Lemma:FinitePresentationAndFlatnessFormalVersion} (2) that $M'$ is finitely generated. It remains to prove that $M' / f M'$ is $R$-torsion free. Suppose there are $m, m' \in M'$ and $\varpi \in R \smallsetminus\{0\}$ such that $\varpi m = f m'$. We want to show that there is $m'' \in M'$ such that $m = f m''$. Now, multiplication by $\varpi$ on $N / f N$ is injective, therefore there is $n \in N$ such that $m = f n$. Since $f$ is invertible in $A_K[1/f]$,  $n = m / f$ belongs to $M \otimes_A A_K[1/f]$. Therefore $n$ belongs to $M'$,  and $m'' = n$ does the job. The first isomorphism follows from the equality $M \otimes_A A_K[1/f] = M' \otimes_A A_K[1/f]$, which holds by definition of $M'$. To prove the second isomorphism, one has to show that the homomorphism $i_K \colon M \otimes_A A_K \{ 1/f\} \to M' \otimes_A A_K \{ 1/f\}$ is an isomorphism. The homomorphism $i_K$ is injective because the $A$-algebra $A_K \{ 1/f\}$ is flat. Let $m' \in M'$. By definition of $M'$, there are $\varpi \in R \smallsetminus \{ 0\}$ and a non-negative integer $n \in \bbN$ such that $\varpi f^n m'$ belongs to $M$. Since both $\varpi$ and $f$ are invertible in $A_K \{ 1/f\}$, one can consider the element $m:= \varpi f^n m' \otimes 1/ (\varpi f^n)$ of $M \otimes_A A_K \{ 1/f\}$. Clearly, $m$ maps to $m'$ via $i_K$. This shows that $i_K$ is surjective.
\end{proof}

\begin{lemma} \label{Prop:ExtendingSemiReflexiveModules} Let $D$ be an admissible effective Cartier divisor in a quasi-compact admissible formal $R$-scheme $X$ and $F_K$ a semi-reflexive coherent sheaf on~$X_\eta$. 

Then, there exist a coherent $\cO_X$-module $F$ flat over $R$ and $D$-torsion free  such that $F_{\rvert D}$ is flat over $R$, and an injective homomorphism of $\cO_{X_\eta}$-modules $F_K \into F_\eta$ inducing isomorphisms $F_K(\ast D_\eta) \cong F_\eta(\ast D_\eta)$ and $F_{K \rvert U_\eta} \cong F_{\eta \rvert U_\eta}$.
\end{lemma}

\begin{proof} Let $E$ be a coherent $\cO_X$-module flat over $R$ such that  $E_\eta \cong F_K$. By \cref{Lemma:AutomaticallySemiReflexive} $E$ is semi-reflexive because so is $F_K$. By \cref{Lemma:RestrictionOfDualIsFlatFormal} the coherent $\cO_X$-module $E^{\vee \vee}$ is $D$-torsion free, flat over $R$, and $E^{\vee\vee}_{\rvert D}$ is flat over $R$. According to \cref{Lemma:SaturationOfSemiReflexiveFormal}, the coherent $\cO_X$-module $F := \Ker(E^{\vee\vee} \to [E^{\vee\vee}(\ast D) / E(\ast D)]_K)$ does the job.
\end{proof}

\subsection{Proof of \cref{Thm:ClosednessImage}} We begin with the proof of \cref{Thm:ClosednessImage} which is somewhat easier and does not make use of \cref{Thm:DensityMeromorphicSectionsModel}. To reduce to strict spaces we will use the following:

\begin{lemma} \label{Lemma:ClosedenessExtensionOfScalars} Let $F$ be a coherent sheaf on a $K$-analytic space $X$ and $K'$ a complete valued extension of $K$. Then, the natural homomorphism $\Gamma(X, F) \to \Gamma(X', F')$, where $X'$ and $F'$ are obtained from $X$ and $F$ by extending scalars, is a homeomorphism onto its image and its image is closed.
\end{lemma}

\begin{proof} One may assume $F= \cO_X$ replacing $X$ by the total space $p \colon \bbV(F) \to X$ of the dual of $F$, i.e. the analytic space representing the functor associating to a morphism of $K$-analytic spaces $f \colon Y \to X$ the set of morphisms $F \to f_\ast \cO_Y$ of $\cO_X$-modules. Indeed, the universal property applied to $p$ yields a  morphism $F \to p_\ast \cO_{\bbV(F)}$ which identifies $\Gamma(X, F)$ with a closed subset of $\Gamma(\bbV(F), \cO_{\bbV(F)})$. To see this last claim, consider the morphisms $\sigma, \pr_2 \colon \bbA^1_K \times \bbV(F) \to \bbV(F)$ where $\sigma$ is the natural action by homotheties on fibers and $\pr_2$ the second projection. Then $F$ can be identified with the kernel of the morphism $f \mapsto \sigma^\ast f - z \pr_2^\ast f$, where $z$ is the coordinate on the affine line. Now, because of the definition of the topology on the global sections of a coherent sheaf, by taking an affinoid G-cover one reduces to the case where $X$ is affinoid. Once $X$ is supposed  to be affinoid, by embedding it as a closed subspace of a closed disc, one reduces to the case $X = \bbD(r_1, \dots, r_n)$ for $r_1, \dots, r_n \in \bbR_{> 0}$. The statement for the map
\[ \Gamma(X, \cO_X) = K \{ t_1 / r_1, \dots, t_n / r_n \} \too \Gamma(X', \cO_X') = K' \{ t_1 / r_1, \dots, t_n / r_n \},\]
is clearly true, the topology on both sides being defined by Gauss norm with respect to the radii $r_1, \dots, r_n$. This concludes the proof.
\end{proof}

The following lemma permits to enlarge your domain:

\begin{lemma} \label{Lemma:ExhaustionByBlowUp} Let $X$ be a quasi-compact admissible formal $R$-scheme. Let $Z$ be a closed admissible formal subscheme of $X$. Let $C$ be a compact subset of $X_\eta \smallsetminus Z_\eta$.

Then, there is $\epsilon > 0$ such that every $\varpi \in R \smallsetminus \{ 0\}$ with $|\varpi| < \epsilon$ has the following property: if $\pi \colon X' \to X$ denotes the admissible blow-up of $X$ along $\varpi \cO_{X} + I_Z$ and $Z'$ the strict transform of $Z$ in $X'$, then $\pi_\eta^{-1}(C)\subset (X' \smallsetminus Z')_\eta$.
\end{lemma}

\begin{proof} By taking a finite open affine covering of $X$, one reduces to the case where $X = \Spf(A)$ is the formal spectrum of a flat topologically finitely presented $R$-algebra $A$. Let  $f_1, \dots, f_r$ be generators of $I:= \Gamma(X, I_Z)$. Given $x \in X_\eta \smallsetminus Z_\eta$ there is at least one among $f_1, \dots, f_r$ which does not vanishes at $x$, thus by compactness of~$C$ we have $\epsilon := \min_{x \in C}  \max \{ |f_1(x)|, \dots, |f_r(x)|\} > 0$. Let $\varpi \in R \smallsetminus \{ 0\}$ be such that $|\varpi| < \epsilon$. Let $\pi \colon X' \to X$ be the admissible blow-up of $X$ along $\varpi \cO_{X} + I_Z$. By writing $f_0 = \varpi$ the blow-up $X'$ is covered by the affine open subsets $X'_i = \Spf A_i$ with
\[
A_i := A \{ f_j / f_i\}_{j \neq i} = [A \{ t_{ij} \}_{j \neq i} / (f_j - t_{ij} f_i)_{j \neq i}] / (f_i\textup{-torsion}).
\]
Let $Z'$ be the strict transform of $Z$ in $X'$.

\begin{claim} The strict transform $Z'$ is contained in $X_0'$. \end{claim}

\begin{proof}[Proof of the claim] The intersection of $Z'$ with $X_i'$ is the closed subscheme corresponding to the ideal $I_i = \Ker( A_i \to [A_i / I A_i] \otimes_R K)$. If $i \ge 1$, then $\varpi$ belongs to $I_i$ because it can be written as $\varpi = t_{i0} f_i \in I_i$. On the other hand,  $\varpi$ is invertible in $A_i \otimes_R K$, thus $1$ belongs to $I_i$. This means that $Z'$ does not meet $X_i'$ for all $i \ge 1$ and, consequently, $Z'$ is contained in $X'_0$.
\end{proof}

\begin{claim} The compact subset $\pi_\eta^{-1}(C)$ does not meet $(X'_0)_\eta$.
\end{claim}

\begin{proof}[Proof of the claim] For  $x \in (X'_0)_\eta$, one has $|f_i(x)| \le |\varpi| < \epsilon$ for all $i \ge 1$. Instead, by the choice of $\epsilon$, given a point $x \in \pi_\eta^{-1}(C)$ there exists $i \in \{ 1, \dots, n\}$ such that $|f_i(x)| \ge \epsilon$.
\end{proof}

From the previous claims one draws the following chain of inclusions:
\[ \pi_\eta^{-1}(C) \subset (X'_1 \cup \cdots \cup X'_r)_\eta \subset (X' \smallsetminus Z')_\eta, \]
which proves the statement.
\end{proof}

\begin{proof}[{Proof of \cref{Thm:ClosednessImage}}] According to \cref{Lemma:ClosedenessExtensionOfScalars} one may suppose that the $K$-analytic space $X$ is strict. In this case the proof consists of three steps of increasing generality.

\medskip

\emph{Step 1.} Suppose $X$ compact and $Z$ an effective Cartier divisor. In this step write $X_K$, $Z_K$ and $F_K$ instead of $X$, $Z$ and $F$. It suffices to prove the following: given a compact analytic domain $C \subset X_K \smallsetminus Z_K$, there is a compact analytic domain $C' \subset X_K \smallsetminus Z_K$ containing $C$ such that the restriction map $\Gamma(X_K, F_K) \to \Gamma(C', F_K)$
is a homeomorphism onto its image. Indeed, the image will be automatically closed because $\Gamma(X_K, F_K)$ and $\Gamma(C', F_K)$ are Banach spaces. According to \cref{Thm:FormalModelsOfCartierDivisors}, there are a quasi-compact quasi-separated admissible formal $R$-scheme $X$ and an admissible effective Cartier divisor $D$ in $X$ whose generic fibers are respectively $X_K$ and $Z_K$. Up to replacing $X$ by its blow-up along $\varpi \cO_X + I_D$ for some $\varpi \in R$ non-zero and $D$ by its strict transform, by \cref{Lemma:ExhaustionByBlowUp} and \cref{Prop:StrictTransformBlowUpIntersectionTwoCartierFormal} one may assume $C = U_\eta$ where $U = X \smallsetminus D$. According to \cref{Prop:ExtendingSemiReflexiveModules}, there are a coherent $\cO_{X}$-module $F$ and an injective homomorphism of $\cO_{X_K}$-modules $\iota \colon F_K \to F_{\eta}$ inducing an isomorphism of $\cO_{U_{\eta}}$-modules $F_{K \rvert U_{\eta}} \iso F_{\eta \rvert U_{\eta}}$. Now we have $\Gamma(X_K, F_{\eta})  = \Gamma(X, F) \otimes_R K$ and $\Gamma(U_{ \eta}, F_{\eta})  = \Gamma(U, F) \otimes_R K$ because $X$ and $U$ are quasi-compact and quasi-separated. Therefore by \cref{Lemma:MeromorphicInjectsInHolomorphicFormal} (4) the restriction homomorphism $\Gamma(X_K, F_{\eta}) \to \Gamma(U_{\eta}, F_{\eta})$
is an isometric embedding. On the other hand, the exact sequence
\[ 0 \too \Gamma(X_K, F_K) \too \Gamma(X_K, F_{\eta}) \too \Gamma(X_K, \coker \iota)\]
shows that the image of $\iota \colon \Gamma(X_K, F_K) \to \Gamma(X_K, F_{\eta})$ is closed. Banach's Open Mapping theorem implies $\iota$ maps $\Gamma(X_K, F_K)$ into a closed subset of $\Gamma(X_K, F_{\eta})$. This concludes the proof in this case.

\medskip

\emph{Step 2.} Suppose $X$ compact. Let $\pi \colon X' \to X$ be the blow-up of $X$ along $Z$. Let $Z' = \pi^{-1}(Z)$, and $U' = X' \smallsetminus Z'$, so that the induced map $\pi \colon U' \to U$ is an isomorphism. The coherent $\cO_{X'}$-module $\pi^\ast F$ may fail to be semi-reflexive. Nonetheless its image $F'$ in $(\pi^\ast F)^{\vee \vee}$ is semi-reflexive, being a sub-$\cO_{X'}$-module of a semi-reflexive one. The composite homomorphism of $\cO_{X}$-modules $\phi \colon F \to \pi_\ast \pi^\ast F \to \pi_\ast F'$ is injective. Indeed, let $j \colon U \to X$, $j' \colon U' \to X'$ be the open immersions. Then the following diagram of $\cO_{X}$-modules is commutative
\begin{center}
\begin{tikzcd}
F \ar[r] \ar[d, hook] & \pi_\ast F' \ar[d,hook]\\
j_\ast j^\ast F \ar[r, "\sim"]& \pi_\ast j'_\ast j'^\ast F'
\end{tikzcd}
\end{center}
where the vertical arrows are injective because $U$ (resp. $U'$) is scheme-theoretically dense in $X$ (resp. $X'$) and $F$ (resp. $F'$) is semi-reflexive, and the lower horizontal arrow is an isomorphism because $\pi \colon U' \to U$ is. It follows that $F \to \pi_\ast F'$ is injective. Taking global sections, one obtains the following commutative diagram
\begin{center}
\begin{tikzcd}
\Gamma(X, F) \ar[r, hook] \ar[d, hook] & \Gamma(X', F') \ar[d,hook]\\
\Gamma(U, F) \ar[r, "\sim"]& \Gamma(U', F')
\end{tikzcd}
\end{center}
where the lower horizontal arrow is a homeomorphism and the rightmost vertical is a homeomorphism onto its image which is closed in the target. To conclude, it suffices to show that $\Gamma(X, F) \to \Gamma(X', F')$ is a homeomorphism onto its image and its image is closed. To do so, remark that the $\cO_{X}$-module $\pi_\ast F'$ is coherent because $\pi$ is proper. From the exact sequence
 \[ 0 \too \Gamma(X, F) \too \Gamma(X', F') \too \Gamma(X, \coker \phi)\]
one deduces that  $\Gamma(X, F) \to \Gamma(X', F')$ is injective and has closed image. By Banach's Open Mapping theorem, it induces a homeomorphism of $\Gamma(X, F)$ with its image.

\medskip

\emph{Step 3}. Let $\cV$ be an affinoid G-cover of $X$. Then $\{ V \smallsetminus Z \}_{V \in \cV}$ is a G-cover by analytic domains of $X \smallsetminus Z$ and
\begin{align*}
\Gamma(X, F) &= \projlim_{V \in \cV} \Gamma(V, F), \\
\Gamma(X \smallsetminus Z, F) &= \projlim_{V \in \cV} \Gamma(V \smallsetminus Z, F).
\end{align*}
The statement follows by applying the preceding case.
\end{proof}

\subsection{Proof of \cref{Thm:ComplementDivisorHolomorphicallyConvex} in the compact case} \label{sec:ProofComplementDivisorCompactCase} We will first prove \cref{Thm:ComplementDivisorHolomorphicallyConvex} when the $K$-analytic space is compact. To do this it is convenient to say that a $K$-analytic space $X$ is \emph{proper over affinoid} if there is  a proper morphism of $K$-analytic spaces $X \to Y$ with $Y$ affinoid. Here we will prove the following:

\begin{theorem} \label{Thm:DensityMeromorphicSectionsAnalytic} Let $D$ be an effective Cartier divisor in a $K$-analytic space $X$ which is proper over affinoid. Then, for any semi-reflexive coherent $\cO_{X}$-module~$F$, the natural map
\[ \Gamma(X, F(\ast D)) \too \Gamma(X \smallsetminus D, F)\]
is injective, and has a dense image.
\end{theorem}

We reduce first of all to case where the $K$-analytic space in question is strict. For, consider $r_1, \dots, r_n > 0$ linearly independent in the $\bbQ$-vector space $\bbR_{> 0} / \sqrt{|K^\times|}$. Then the completion  $K_r$ of the $K$-algebra $K[t_1^{\pm}, \dots, t_n^{\pm}]$ with respect to the norm
\[ \left\|\sum_{a \in \bbZ^n } c_a t^a  \right\|_r := \max_{a \in \bbZ^n} |c_a| r^a \] 
is a field by \cite[p. 22]{Berkovich90} where $t^a= t_1^{a_1} \cdots t_n^{a_n}$ and $r^a = r_1^{a_1} \cdots r_n^{a_n}$. For a $K$-analytic space $X$ let $X_r$ be the $K_r$-analytic space obtained from $X$ by extending scalars to $K_r$. Given a coherent sheaf $F$ on $X$ let $F_r$ the coherent sheaf on $X_r$ deduced from $F$. Suppose $X$ compact and let $\| \cdot \|_X$ be a norm defining the topology of $\Gamma(X, F)$. Then and $\Gamma(X_r, F_r) = \Gamma(X, F) \hotimes_K K_r$ and  a norm $\| \cdot \|_{X_r}$ defining the topology of $\Gamma(X_r, F_r)$ is given by
\[ \left\| \sum_{a \in \bbZ^n } f_a t^a \right\|_{X_r} = \max_{a \in \bbZ^n} \|f_a\|_{X} r^a. \]
To reduce to strict spaces we need to be able to descend density in global sections:

\begin{lemma} \label{Lemma:DescendingDensityAlongScalarExtensions} Let $D$ be an effective Cartier divisor in a compact $K$-analytic space~$X$ and $F$ a coherent $\cO_X$-module. If the restriction $\Gamma(X_r, F_r(\ast D_r)) \to \Gamma(X_r \smallsetminus D_r, F_r)$ is injective and has a dense image, then so does  $\Gamma(X, F(\ast D)) \to \Gamma(X \smallsetminus D, F)$.
\end{lemma}

\begin{proof} Let $U = X \smallsetminus D$, $f \in \Gamma(U, F)$ and $C \subset U$ a compact analytic domain. For $d \in \bbN$ let $\| \cdot \|_{X, d}$ and $\| \cdot \|_{C} $ be norms defining the topology respectively on $\Gamma(X, F(dD))$ and on $\Gamma(C, F)$. Via the isomorphism $\Gamma(C, F) \cong \Gamma(C, F(dD))$
one deduces a norm on $\Gamma(C, F(dD))$ defining its topology: with an abuse of notation, denote it again by $\| \cdot \|_C$. Let $\epsilon > 0$. By hypotheses, there exists $g \in \Gamma(X_r, F_r(\ast D_r))$ such that $\| f - g\|_{C_r} < \epsilon$. Now $g \in\Gamma(X_r, F_r(d D_r))$ for some $d \in \bbN$ because
\[ \Gamma(X_r, F_r(\ast D_r)) = \injlim_{d \in \bbN} \Gamma(X_r, F(d D_r))\]
by compactness of $X_r$ and \cref{lemma:InjectiveLimitSheavesGTopology}. Write  $g = \sum_{a \in \bbZ^n} g_a t^a$
with sections $g_a \in \Gamma(X, F(dD))$ such that $ \|g_a\|_{X, d} r^a \to 0$ as $|a| \to \infty$. Since the restriction map $\Gamma(X, F(dD)) \to \Gamma(C, F(dD))$ is continuous, there is a real number $R > 0$, such that $\| g_a \|_{C} \le R \| g_a \|_{X, d} $ for all $a \in \bbZ^n$. Thus the series $\sum_{a \in \bbZ^n} g_a t^a$ converges on~$C_r$ and 
\[ \| f - g\|_{C_r} = \max_{a \in \bbZ^n \smallsetminus \{ 0 \}} \left \{ \| f- g_0 \|_C, \| g_a \|_C r^a \right\}. \]
In particular $\| f - g\|_{C_r} < \epsilon$ implies $\| f- g_0 \|_C < \epsilon$. This shows that $\Gamma(X, F(\ast D))$ is dense in $\Gamma(U, F)$.
\end{proof}

\begin{proof}[{Proof of \cref{Thm:DensityMeromorphicSectionsAnalytic}}] Write $X_K$ instead of $X$ and similarly for $D$ and $F$. Thanks to \cref{Lemma:DescendingDensityAlongScalarExtensions}, up to extending scalars, one may suppose that $X_K$ admits a proper morphism with strictly affinoid target. Let $C \subset X_K \smallsetminus D_K$ be a compact analytic domain. We want to show that $\Gamma(X_K, F_K(\ast D_K))$ is dense in $\Gamma(X_K \smallsetminus D_K, F_K)$ with respect to a norm $\| \cdot \|_C$ defining the topology of $\Gamma(C, F_K)$. According to \cref{Thm:FormalModelsOfCartierDivisors} there are a proper over affine admissible formal $R$-scheme, an admissible effective Cartier divisor $D$ in $X$ and an isomorphism of $K$-analytic spaces $\phi \colon X_\eta \to X_K$ inducing an isomorphism $D_\eta \cong D_K$. Let $U = X \smallsetminus D$. Thanks to \cref{Prop:ExtendingSemiReflexiveModules}, there are a coherent $\cO_{X}$-module $F$, flat over $R$, $D$-torsion free and such that $F_{\rvert D}$ is flat over $R$, and an injective homomorphism of $\cO_{X_\eta}$-modules $F_K \into F_\eta$ inducing isomorphisms $F_K(\ast D_\eta) \cong F_\eta(\ast D_\eta)$ and $F_{K \rvert U_\eta} \cong F_{\eta \rvert U_\eta}$. For $d \in \bbN$, the natural homomorphisms
\[
\Gamma(X, F(dD)) \otimes_R K \too \Gamma(X_K, F_K(dD_K)), \quad
\Gamma(U, F) \otimes_R K  \too \Gamma(U_\eta, F_K),
\]
are isomorphisms. In what follows, these identifications are understood.

\begin{claim}
Up to replacing $X$ by its blow-up along $\varpi \cO_X + I_D$ for some $\varpi \in R \smallsetminus \{ 0 \}$, one may assume $C = U_\eta$ and $\| \cdot \|_C$ to be the norm induced by the $R$-module $\Gamma(U, F)$.
\end{claim}

\begin{proof}[Proof of the claim]
According to \cref{Lemma:ExhaustionByBlowUp}, there is $\varpi \in R \smallsetminus \{ 0 \}$ with the following property. Let $\pi \colon X' \to X$ be the admissible blow-up of $X$ along $\varpi \cO_X + I_D$. Let $D' \subset X'$ be the strict transform of $D$ and $U' = X' \smallsetminus D'$. Then $\pi_\eta^{-1}(C)$ is contained in $U'_\eta$. Since the restriction map $\Gamma(U'_\eta, \pi^\ast F_K) \to \Gamma(\pi_\eta^{-1}(C), \pi^\ast F_K)$ is continuous, one may suppose $\pi_\eta^{-1}(C) = U'_\eta$. Moreover, by \cref{Prop:StrictTransformBlowUpIntersectionTwoCartierFormal}, the strict transform $D'$ of $D$ in $X'$ in an admissible effective Cartier divisor.  Concerning the choice of the norm, density in $\Gamma(U'_\eta, \pi^\ast F_K)$ only depends on its topology. Therefore, it suffices to test density for a norm defining its topology: the norm induced by the $R$-module $\Gamma(U', \pi^\ast F)$ definitely falls into this category. 
\end{proof}

Henceforth suppose $C = U_\eta$. The natural map $\Gamma(X, F(\ast D)) \to \Gamma(U, F)$
is injective and has dense image (\cref{Thm:DensityMeromorphicSectionsModel}). By extending scalars to $K$, one deduces that the map $\Gamma(X_K, F_K(\ast D_K)) \to \Gamma(U_K, F_K)$ is injective with dense image. This concludes the proof.
\end{proof}

\subsection{Proof of \cref{Thm:ComplementDivisorHolomorphicallyConvex} in the general case} We can now complete the proof of \cref{Thm:ComplementDivisorHolomorphicallyConvex}.

\begin{lemma} \label{Lemma:InjectivityRestrictionMapAnalyticSpace} Let $D$ be an effective Cartier divisor in a $K$-analytic space $X$.  Then the restriction homomorphism $\rho \colon \Gamma(X, F(\ast D)) \to \Gamma(X \smallsetminus D, F)$ is injective for any semi-reflexive $\cO_X$-module $F$.
\end{lemma}

\begin{proof} By taking an affinoid G-cover it suffices to prove the statement when $X$ is compact. In this case $ \Gamma(X, F(\ast D)) = \injlim_{d \in \bbN} \Gamma(X, F(d D))$ by compactness of $X$ and \cref{lemma:InjectiveLimitSheavesGTopology}. For $d \in \bbN$ the $\cO_X$-module $F(dD) = F \otimes \cO_X(dD)$ is semi-reflexive because $F$ is. The complement $U:= X \smallsetminus D$ is scheme-theoretically dense in $X$ thus the restriction homomorphism $\rho_d \colon \Gamma(X, F(dD)) \to \Gamma(U, F(dD))$ is injective. For an integer $d' \ge d$, one has the following commutative diagram
\begin{center}
\begin{tikzcd}
\Gamma(X, F(dD)) \ar[r] \ar[d, "\rho_{d}", hook] & \Gamma(X, F(d'D)) \ar[d, "\rho_{d'}", hook]\\
\Gamma(U, F(dD)) \ar[r, "\sim"] & \Gamma(U, F(d'D)),
\end{tikzcd}
\end{center}
the lower horizontal arrow being an isomorphism. By taking the injective limit one sees that $\rho$ is injective. This concludes the proof.
\end{proof}

\begin{proof}[{Proof of \cref{Thm:ComplementDivisorHolomorphicallyConvex}}] The restriction homomorphism $\Gamma(X, F(\ast D)) \to \Gamma(X \smallsetminus D, F)$ is injective by \cref{Lemma:InjectivityRestrictionMapAnalyticSpace}. Let $U := X \smallsetminus D$, $s \in \Gamma(U, F)$, $C \subset U$ a compact analytic domain,  $\| \cdot \|$ a norm on $\Gamma(C, F)$ defining its topology. In order to prove the statement we have to show that for each $\epsilon > 0$ there are $d \in \bbN$ and $t \in \Gamma(X, F(dD))$ such that $\| s - t\| < \epsilon$.  According to \cref{thm:Remmert}, there are a Stein $K$-analytic space $S$ and a proper morphism $\sigma \colon X \to S$. Let $\{ S_i\}_{i \in \bbN}$ be an affinoid Stein exhaustion of $S$. For $i \in \bbN$ set $X_i := X \times_S S_i$. For $i \in \bbN$ big enough the analytic domain $X_i$ contains~$C$. Therefore, we may suppose $C = X_i$. The induced morphism $\sigma \colon X_i \to S_i$ is proper. The closed analytic subspace $D_i := D \times_X X_i$ of $X_i$ is an effective Cartier divisor. This allows to apply \cref{Thm:DensityMeromorphicSectionsAnalytic} to $X_i$. It follows that there are $d \in \bbN$ and $t \in \Gamma(X_i, F(dD))$ such that $\| t - s\| < \epsilon$. In order to conclude, note that the $\cO_X$-module $M:= \sigma_\ast F(dD)$ is coherent because the morphism $\sigma$ is proper. Then, as recalled above, the restriction map
\[ \Gamma(X, F(dD)) = \Gamma(S, M) \too \Gamma(S_i, M) = \Gamma(X_i, F(dD)) \]
has a dense image, because $S$ is Stein. This concludes the proof.
\end{proof}

\section{Approximating analytic functions by rational functions}

Let $K$ be a complete non-trivially valued  non-Archimedean field.

\subsection{Statements} In this section we deduce density results of algebraic functions by using the results of the previous sections. Let $X$ be a separated finite type scheme over an affinoid $K$-algebra~$A$ and $F$ a semi-reflexive coherent $\cO_X$-module. The main result is the following:

\begin{theorem} \label{Thm:DensityAlgebraicSections} The subset $\Gamma(X, F) \subset \Gamma(X^\an, F^\an)$ is dense.
\end{theorem}

\begin{corollary} \label{Cor:SectionsAffineOpen} For a a scheme-theoretically dense open subset $U \subset X$ the following are equivalent:
\begin{enumerate}
\item $\Gamma(X, F) \subset \Gamma(U^\an, F^\an)$ is dense;
\item $\Gamma(X^\an, F^\an) \subset \Gamma(U^\an, F^\an)$ is dense;
\item $\Gamma(X^\an, F^\an) = \Gamma(U^\an, F^\an)$;
\item $\Gamma(X, F) = \Gamma(U, F)$.
\end{enumerate}
\end{corollary}

\subsection{Scheme-theoretic closure and analytification} We need some results on the analytification of quasi-coherent sheaves. In this section let $f \colon X \to Y$ be a separated morphism between $A$-schemes of finite type. For an $\cO_X$-module $E$ define~$E^\an$ as the pull-back of $E$ along the canonical morphism $X^\an \to X$ of locally ringed spaces. Also, let
\[ \theta_{E, f} \colon (f_\ast E)^\an \too f^\an_\ast E^\an\] 
be the canonical homomorphism of $\cO_{Y^\an}$-modules.

\begin{lemma} \label{Lemma:PushForwardDirectLimitOfCoherent} Let $\{ E_i \}_{i \in I}$ be a directed system of $\cO_X$-modules and $E = \injlim E_i$. If all transition maps are injective and~$\theta_{E_i, f}$ is injective for all $i\in I$, then $\theta_{E, f}$ is injective.
\end{lemma}

\begin{proof} The homomorphism of $\cO_{Y^\an}$-modules
\[ \tilde{\theta} := \injlim_{i \in I} \theta_{E_i, f} \colon \injlim_{i \in I} (f_\ast E_i)^\an \too \injlim_{i \in I} f^\an_\ast E^\an_i\]
is injective. Since $f$ is quasi-compact and quasi-separated, pushing-forward along~$f$ commutes with injective limits \cite[\href{https://stacks.math.columbia.edu/tag/009E}{lemma 009E}]{stacks-project} thus 
\[\injlim_{i \in I} (f_\ast E_i)^\an = (\injlim_{i \in I} f_\ast E_i)^\an = (f_\ast E)^\an.\]
Here we used that $M \rightsquigarrow M^\an$ commutes with colimits because it is the pullback along a morphism of locally ringed spaces, hence a left adjoint functor.
On the other hand $\theta_{E, f} = \psi \circ \tilde{\theta}$  where 
$ \psi \colon \injlim_{i \in I} f^\an_\ast E^\an_i \to f^\an_\ast E^\an$
is the canonical morphism. To conclude the proof it remains to prove that $\psi$ is injective. Since all transition maps are injective, for any $i \in I$ the canonical maps $E_i \to E$ are injective. The morphism of locally ringed spaces $X^\an \to X$ being faithfully flat \cite[Proposition 2.6.2]{BerkovichIHES}, the induced morphism $E_i^\an \to E^\an$ is injective and, by left exactness of the push-forward, so is $f^\an_\ast E_i^\an \to f^\an_\ast E^\an$. Passing to the limit, one sees that $\psi$ is injective.
\end{proof}

\begin{lemma} \label{Prop:InjectivityPushingForward} If $E$ is quasi-coherent,  then $\theta_{E, f}$ is injective.
\end{lemma}

\begin{proof} The module $E$ is the injective limit of its coherent $\cO_X$-submodules by \cite[\href{https://stacks.math.columbia.edu/tag/01PG}{lemma 01PG}]{stacks-project}. By \cref{Lemma:PushForwardDirectLimitOfCoherent} one may suppose $E$ coherent henceforth. When $f$ is proper $\theta_{E, f}$ is an isomorphism by the non-Archimedean GAGA theorem \cite{Kopf}, \cite[Example 3.2.6]{ConradAmpleness}, \cite[Th\'eor\`eme A.1]{PoineauRaccord}. In the general case, by Nagata's compactification theorem (see \cite{LutkebohmertNagata},  \cite{ConradNagata} and \cite{ConradNagataErratum}, or \cite{TemkinRiemannZariski}) there are an $A$-scheme~$X'$ of finite type together with a proper morphism $\pi \colon X' \to Y$ and open immersion $i \colon X \to X'$ such that $f = \pi \circ i$. Up to blowing up, one may assume that $X' \smallsetminus i(X)$ supports an effective Cartier divisor $D$. Let $E'$ be a coherent sheaf on $X'$ extending $E$ \cite[\href{https://stacks.math.columbia.edu/tag/01PI}{lemma 01PI}]{stacks-project}. Then $i_\ast E = \injlim_{d \in \bbN} E'(dD)$. Therefore $(i_\ast E)^\an = \injlim_{d \in \bbN} E'(dD)^\an$ is the subsheaf of sections of $E^\an$ with polar singularities along $D$. In particular, the natural map
\[ \theta_{E, i} \colon (i_\ast E)^\an = \injlim_{d \in \bbN} E'(dD)^\an \too i^\an_\ast i^{\an \ast} E'^\an = i^\an_\ast E^\an, \]
is injective by \cref{Lemma:MeromorphicSectionsAnalytic}. According to the proper case, the natural map
\[ \theta_{i_\ast E, \pi} \colon (f_\ast E)^\an =  (\pi_\ast i_\ast E)^\an \too \pi^\an_\ast (i_\ast E)^\an\]
is injective. Since pushing forward is a left exact functor $\pi^\an_\ast \theta_{E, i}$ is injective, thus so is $\theta_{E, f} = \pi^\an_\ast \theta_{E, i} \circ \theta_{i_\ast E, \pi}$.
 \end{proof}

\begin{lemma} \label{Cor:InjectivityPushForward} For an $\cO_Y$-module $F$ we have a natural isomorphism:
\[ \ker(F \to f_\ast f^\ast F)^\an \stackrel{\sim}{\too}\ker(F^\an \to f^\an_\ast f^{\an \ast} F^\an). \]
\end{lemma}

\begin{proof} Let $K := \ker(F \to f_\ast f^\ast F)$. The functor $M \rightsquigarrow M^\an$ on $\cO_Y$-modules is exact because the morphism of locally ringed spaces $Y^\an \to Y$ is surjective and flat by \cite[2.6.2]{BerkovichIHES} thus $K^\an = \ker(F^\an \to (f_\ast f^\ast F)^\an)$. This yields the following commutative diagram of $\cO_{Y^\an}$-modules:
\begin{center}
\begin{tikzcd}
0 \ar[r] & K^\an \ar[r] \ar[d] & F^\an \ar[r] \ar[d, equal] & (f_\ast f^\ast F)^\an \ar[d, "\theta_{f^\ast F, f}"] \\
0 \ar[r] & K' \ar[r] & F^\an \ar[r] & f^\an_\ast f^{\an \ast} F^\an
\end{tikzcd}
\end{center}
where $K' := \ker(F^\an \to f^\an_\ast f^{\an \ast} F^\an)$. Since $\theta_{f^\ast F, f}$ is injective by \cref{Prop:InjectivityPushingForward}, one concludes by the Snake lemma.
\end{proof}

\begin{corollary} \label{Cor:HolomorphicExtensionOfMeromorphicSections} Let $U$ be an open subset of an $A$-scheme of finite type $X$. Let $F$ be a quasi-coherent $\cO_X$-module such that $F \to j_\ast j^\ast F$ is injective where $j \colon U \to X$ is the open immersion. If $f \in \Gamma(X^\an, F^\an)$ is such that  $f_{\rvert U^\an}\in \Gamma(U, F)$, then $f \in \Gamma(X, F)$.
\end{corollary}

\begin{proof} Let $\pi \colon X^\an \to X$ be the canonical morphism of locally ringed spaces. For a sheaf of abelian groups $E$ on $X$ let $\pi^{-1}E$ be the pull-back of $E$ as sheaf of abelian groups. For each $x \in X^\an$, $\cO_{X^\an, x}$ is a faithfully flat $\cO_{X, \pi(x)}$-algebra \cite[2.6.2]{BerkovichIHES}. Thus, given an $\cO_X$-module $E$, the natural homomorphism $\phi_E \colon \pi^{-1}E \to E^\an$ is injective. Let $C$ be the cokernel of the homomorphism $F \to j_\ast j^\ast F$. Then the following diagram of $\pi^{-1}\cO_X$-modules is commutative and exact:
\begin{center}
\begin{tikzcd}
0 \ar[r] & \pi^{-1} F \ar[r] \ar[d, hook, "\phi_F"] & \pi^{-1} j_\ast j^\ast F \ar[r] \ar[d, hook, "\phi_{j_\ast j^\ast F}"] & \pi^{-1} C \ar[r] \ar[d, hook, "\phi_C"] & 0 \\
0 \ar[r] & F^\an \ar[r] & (j_\ast j^\ast F)^\an \ar[r]  & C^\an \ar[r] & 0
\end{tikzcd}
\end{center}
The Snake lemma implies that the homomorphism $\coker \phi_F \to \coker \phi_{j_\ast j^\ast F}$ is injective. \Cref{Prop:InjectivityPushingForward} applied to the quasi-coherent $\cO_U$-module $j^\ast F$ and to the morphism $j$ states the homomorphism $(j_\ast j^\ast F)^\an \to j^\an_\ast j^{\an \ast} F^\an$ is injective. In particular, the natural maps $F^\an \to j^\an_\ast j^{\an \ast} F^\an$,
\[ \coker \phi_F \too \coker (\pi^{-1} j_\ast j^\ast F \to j^\an_\ast j^{\an \ast} F^\an)\]
are injective. This concludes the proof.
\end{proof}

\subsection{Proofs} We begin with a result of extension of semi-reflexive sheaves:

\begin{lemma} \label{Lemma:ExtendingSemiReflexiveToCompactification} Let $D$ be a Cartier divisor in a Noetherian scheme  $X$ and $U = X \smallsetminus D$. Let $F$ be a coherent $\cO_X$-module such that $F_{\rvert U}$ is semi-reflexive. Then the $\cO_X$-module \[ F^\sat = \Ker(F^{\vee \vee} \to (F^{\vee\vee} / F)_{\rvert U})\]  is $D$-torsion free semi-reflexive and such that $ F_{\rvert U} \cong F^\sat_{\rvert U}$.
\end{lemma}

\begin{proof} The statement is local on $X$, therefore one may assume that $X = \Spec A$ is affine and the Cartier divisor $D$ is given by the ideal $f A$ for a nonzerodivisor $f \in A$. Let $M := \Gamma(X, F)$ and $M' := \Gamma(X, F^\sat)$. Since $f$ is a nonzerodivisor in $A$, the $A$-module $M^{\vee\vee}$ has no $f$-torsion and the localization map $M^{\vee\vee} \to M^{\vee\vee}_f$ is injective. Since $M'$ is a submodule of $M^{\vee\vee}$, the map $M' \to M'_f$ is injective. By flatness of the $A$-algebra $A_f$, one has $M'_f = \Ker(M^{\vee\vee}_f \to M_f^{\vee\vee} / M_f) = M_f$,
the last equality being true because $F_{\rvert U}$ is semi-reflexive. Therefore, in the following commutative diagram of $A$-modules
\begin{center}
\begin{tikzcd}
M' \ar[r] \ar[d, hook, "\alpha"] & M^{\vee\vee} \ar[d]\\
M'_f \ar[r, hook, "\beta"] & M_f^{\vee\vee}
\end{tikzcd}
\end{center}
the maps $\alpha$ and $\beta$ are injective. The injective map $\beta \circ \alpha$ coincide with the composite map $M' \to M'^{\vee\vee} \to M^{\vee\vee} \to M^{\vee\vee}_f.$ It follows that $M' \to M'^{\vee\vee}$ is injective.
\end{proof}

\begin{proof}[{Proof of \cref{Thm:DensityAlgebraicSections}}] To be coherent with the notation elsewhere in this paper, write $U$ for the $A$-scheme in the statement of \cref{Thm:DensityAlgebraicSections}. By Nagata's compactification theorem  there is a proper $A$-scheme $X$ containing $U$ as a scheme-theoretically dense open subscheme. Up to blowing-up the complement of $U$, one may suppose that $D = X \smallsetminus U$ is an effective Cartier divisor. Let $E$ be a coherent $\cO_X$-module together with an isomorphism $E_{\rvert U} \iso F$. According to \cref{Lemma:ExtendingSemiReflexiveToCompactification}, the coherent $\cO_X$-module 
\[ E' = \Ker(E^{\vee\vee} \to (E^{\vee\vee} / E)_{\rvert U})\]
is semi-reflexive and $E'_{\rvert U}$ is isomorphic to $F$. \Cref{Thm:DensityMeromorphicSectionsAnalytic}, applied to the analytification of $X$ and the coherent $\cO_{X}^\an$-module induced by $E$, implies that $\Gamma(X^\an, E'^\an(\ast D^\an))$ is dense in $\Gamma(U^\an, F^\an)$. Since $X^\an$ is compact, by \cref{Lemma:MeromorphicSectionsAnalytic}, the natural map
\[ \injlim_{d \in \bbN} \Gamma(X^\an, E'^\an (d D^\an)) \to \Gamma(X^\an, E'^\an(\ast D^\an)),\]
is an isomorphism. On the other hand, for every $d \in \bbN$, the GAGA principle over affinoid algebras \cite{Kopf}, \cite[Example 3.2.6]{ConradAmpleness}, \cite[Th\'eor\`eme A.1]{PoineauRaccord} yields an isomorphism
\[ \Gamma(X, E'(dD)) \stackrel{\sim}{\too} \Gamma(X^\an, E'^\an(d D^\an)).\]
Thus $\Gamma(U, F) = \injlim \Gamma(X, E' (d D)) \to \Gamma(U^\an, F^\an)$ has a dense image.
\end{proof}

\begin{proof}[{Proof of \cref{Cor:SectionsAffineOpen}}] (1) $\Rightarrow$ (2) Clear.   (2) $\Rightarrow$ (3) comes from the fact that the image of the restriction map $\Gamma(X^\an, F^\an) \to \Gamma(U^\an, F^\an)$ is closed by \cref{Thm:ClosednessImage}.  (3) $\Rightarrow$  (4) Let $f \in \Gamma(U, F)$. By hypothesis, $f$ is the restriction of a section of $F^\an$ on $X^\an$. Therefore, according to \cref{Cor:HolomorphicExtensionOfMeromorphicSections}, $f$ comes by restriction from a section of $F$ on $X$. (4) $\Rightarrow$ (1) By \cref{Thm:DensityAlgebraicSections}, $\Gamma(U, F)$ is dense in $\Gamma(U^\an, F^\an)$.
\end{proof}

\section{Algebraic holomorphic convexity and stable codimension} \label{sec:GoodmanLandman}

In this section we first gather some results on the algebraic counter part of holomorphic convexity. For schemes of finite type over a field, some of these appear in \cite{GoodmanLandman}. We next present some results of Brenner \cite{Brenner} needed for the comparison between affine and Stein varieties.

\subsection{Proper over affine schemes}

For a scheme $X$ we let $\sigma_X \colon X \to \Spec \cO(X)$ be the canonical morphism.

\begin{lemma} \label{Lemma:SteinFactorizationOpenImmersion} Let $U$ be an open subset of an affine scheme $X$. Then $\sigma_U$ is an open immersion.
\end{lemma}

\begin{proof} Write $U$ as a union of principal open subsets $D_X(f_i)$ of $X$, for $i \in I$ and $f_i \in \cO(X)$. For each $i \in I$, the homomorphism $\cO(X)[1/f_i] \to \cO(U)[1/f_i]$ induced by the restriction map is an isomorphism. It follows that the morphism $\sigma_U$ identifies $U$ with the union of the principal open subsets $D_{\Spec \cO(U)}(f_i) \subset \Spec \cO(U)$ for $i \in I$.
\end{proof}

\begin{lemma} \label{Lemma:SteinFactorizationDominant} For a scheme $X$ the scheme-theoretic image of $\sigma_X$ is $\Spec \cO(X)$.
\end{lemma}

\begin{proof} Let $Y$ be the scheme-theoretic image of $\sigma_X$, $i \colon Y \to \Spec \cO(X)$ the closed immersion and $\pi \colon X \to Y$ be the unique morphism such that $\sigma_X = i \circ \pi$. The composite map $\cO_{\Spec \cO(X)} \to i_\ast \cO_Y \to \sigma_{X \ast} \cO_X$
is injective, thus $i^\sharp \colon \cO_{\Spec \cO(X)} \to i_\ast \cO_Y$ is injective and $i$ an isomorphism.
\end{proof}

\begin{lemma} \label{Lemma:SeparationThroughFiniteGeneration} Let $X$ be a scheme of finite type over a Noetherian affine scheme~$S$. Then there is a morphism of $S$-schemes $p \colon \Spec \Gamma(X, \cO_X) \to T$ with $T$ affine of finite type over $S$ such that $\sigma_X^{-1}(\sigma_X(x)) = q^{-1}(q(x))$ for all $x \in X$ where 
$q = p \circ \sigma_X$.
\end{lemma}

\begin{proof} For $h \in \cO(X)$, consider the function $\tilde{h}$ on $X \times_S X$ defined as $\pr_1^\ast h - \pr_2^\ast h$. Let $U_h \subset X \times_S X$ be the open subset on which $\tilde{h}$ is invertible. The scheme $X \times_S X$ is Noetherian thus the open subset $U := \bigcup_{h \in \cO(X)} U_h$ is quasi-compact. Therefore there are $h_1, \dots, h_n \in \cO(X)$ such that $U$ is the union of the open subsets $U_{h_1}, \dots, U_{h_n}$.  Let $A$ be the $\cO(S)$-subalgebra of $\cO(X)$ generated by $h_1, \dots, h_n$, $T := \Spec A$ and $p \colon \Spec \cO(X) \to T$ the morphism induced by the inclusion $A \subset \cO(X)$ and $\tilde{p} = p \circ \sigma_X$. In order to conclude the proof it suffices to prove that, for $x_1, x_2 \in X$, we have
\[\sigma_X(x_1) = \sigma_X(x_2) \iff q(x_1) = q(x_2).\]
The implication ``$\Rightarrow$'' is clear. For the converse assume  $\sigma_X(x_1) \neq \sigma_X(x_2)$. Up to swapping $x_1$ and $x_2$, then there is $h \in \cO(X)$ such that $h(x_1) = 0$ and $h(x_2) \neq 0$. Also, we may assume that $x_1$ and $x_2$ lie over the same point of $S$, the conclusion being trivial otherwise. Let $\xi$ be a point of $X \times_{S} X$ such that $\pr_i(\xi) = x_i$ for $i = 1, 2$. Then $\tilde{h}(\xi) = \pr_1^\ast h(x_1) - \pr_2^\ast h(x_2) \neq 0$. In other words the point $\xi$ belongs to $U$. Therefore, there is $g \in A$ such that $\tilde{g}$ does not vanish at $\xi$. In particular, $q(x_1) \neq q(x_2)$. This concludes the proof.
\end{proof}

The first result characterizes schemes $X$ for which the morphism $\sigma_X$ has proper fibers.

\begin{proposition} \label{Prop:ProperOverQuasiAffine} Let $X$ be a finite type and separated scheme over an affine Noetherian scheme $S$. Then the following are equivalent:

\begin{itemize}
\item[PQA$_1$.] For all $x \in X$, the  $\kappa(\sigma_X(x))$-scheme $\sigma_X^{-1}(\sigma_X(x))$ is proper;
\item[PQA$_2$.] For all $s \in S$ and all $x \in X_s$ closed, the $\kappa(s)$-scheme $\sigma_X^{-1}(\sigma_X(x))$ is proper;
\item[PQA$_3$.] There are an affine $S$-scheme of finite type $g \colon Y \to S$, an open subset $V \subset Y$ and a proper surjective morphism of $S$-schemes $f \colon X \to V$ such that the homomorphism $f^\sharp \colon \cO_V \to f_\ast \cO_X$ induced by $f$ is an isomorphism. Moreover, for any $x \in X$, we have
\[ f^{-1}(f(x)) = \sigma_X^{-1}(\sigma_X(x)).\]
\end{itemize}
\end{proposition}

\begin{proof} (PQA$_1$) $\Rightarrow$ (PQA$_2$) Clear.

(PQA$_2$) $\Rightarrow$ (PQA$_3$) The argument here is inspired from \cite[proposition C.1.3]{BostCharles}. According to \cref{Lemma:SeparationThroughFiniteGeneration} there is an affine $S$-scheme of finite type $Y$ and a morphism of $S$-schemes $f \colon X \to Y$ such that, for each $x \in X$, the set-theoretical identity $\sigma_X^{-1}(\sigma_X(x)) = f^{-1}(f(x))$ holds. By Nagata's compactification (see \cite{LutkebohmertNagata},  \cite{ConradNagata} and \cite{ConradNagataErratum}, or \cite{TemkinRiemannZariski}) there are a  proper morphism $\bar{f} \colon \bar{X} \to Y$  of $S$-schemes and a scheme-theoretically dense open immersion $j \colon X \to \bar{X}$ such that $\bar{f} \circ j = f$. The Stein factorization $Y'$ of $\bar{f}$ is finite over $Y$. Since $S$ is Noetherian and $Y$ is of finite type over $S$, the scheme $Y'$ is Noetherian. Therefore, up to replacing $Y$ by $Y'$ we may assume $Y=\Spec \cO(\bar{X})$. 
The notation is resumed in the following diagram:
\begin{center}
\begin{tikzcd}
X \ar[d, "j"', hook] \ar[r, "\sigma_X"] \ar[dr, "f"]&  \Spec \cO(X) \ar[d] \\
\bar{X} \ar[r, "\bar{f}"] &Y
\end{tikzcd}
\end{center}
where $\bar{f} = \sigma_{\bar{X}} \colon \bar{X} \to Y= \Spec \cO(\bar{X})$ is proper. Now, for each $x \in X$, we have \[\sigma_{\bar{X}}^{-1}(y) = j(f^{-1}(y))\] where $y= f(x)$. Indeed we have  $f^{-1}(y)= \sigma_X^{-1}(y)$ by construction thus the $\kappa(y)$-scheme $f^{-1}(y)$ is proper by hypothesis, where $\kappa(y)$ is the residue field at $y$. Therefore $j(f^{-1}(y))$ is a clopen subset of $\bar{f}^{-1}(y)$. Since  the $\kappa(y)$-scheme $\bar{f}^{-1}(y)$ is geometrically connected \cite[\href{https://stacks.math.columbia.edu/tag/03H0}{theorem 03H0}]{stacks-project} we have the wanted identity. In particular we have $f(X) = \bar{f}(\bar{X}) \smallsetminus \bar{f}(\bar{X} \smallsetminus j(X))$
because $\sigma^{-1}_{\bar{X}}(\bar{f}(x))$ is empty for $x \in \bar{X} \smallsetminus j(X)$. Note that $\bar{f}$ is proper by construction and has a dense image by \cref{Lemma:SteinFactorizationDominant}, thus is surjective. It follows that $V := f(X) \subset Y$  is open because $ j(X) \subset \bar{X}$ is so, and $X = f^{-1}(V)$ is proper over $V$.

(PQA$_3$) $\Rightarrow$ (PQA$_1$) Consider the following commutative diagram:
\begin{center}
\begin{tikzcd}
X \ar[r, "\sigma_X"] \ar[d, "f"']& \Spec \cO(X) \ar[d, "\wr"] \\
V \ar[r, "\sigma_{V}"] & \Spec {\cO(V)}
\end{tikzcd}
\end{center}
The morphism $\sigma_V$ is an open immersion by \cref{Lemma:SteinFactorizationOpenImmersion}. For each $x \in X$ the fiber of $\sigma_X$ at $\sigma_X(x)$ coincides with the fiber of $f$ at $f(x)$, thus it is proper.
\end{proof}

\begin{definition} Under the hypotheses of \cref{Prop:ProperOverQuasiAffine}, the scheme $X$ is \emph{proper over quasi-affine} if it satisfies one of the equivalent properties PQA$_{1-3}$.
\end{definition}

\begin{corollary} \label{Cor:CharacterizationQuasiAffine}Let $X$ be a finite type and separated scheme over an affine Noetherian scheme $S$. Then the following are equivalent:
\begin{enumerate}
\item the $S$-scheme $X$ is quasi-affine, that is, there is an open embedding $X \into Y$ of $S$-schemes with $Y$ affine of finite type over $S$;
\item the morphism $\sigma_X$ is injective.
\end{enumerate}
\end{corollary}

\begin{proof} (1) $\Rightarrow$ (2) Clear. (2) $\Rightarrow$ (1) With the notation of \cref{Prop:ProperOverQuasiAffine} the morphism $f$ is finite, thus it identifies $X$ with spectrum over $V$ of the quasi-coherent $\cO_V$-algebra $f_\ast \cO_X$. As $\cO_V\to f_\ast \cO_X$ is an isomorphism $f$ is an isomorphism too.
\end{proof}

\begin{proposition} \label{Prop:Nullstellensatz} Let $S$ be an affine Noetherian scheme. Let $g \colon Y \to S$ be an affine $S$-scheme of finite type, $j \colon V \to Y$ an open immersion and  $\pi \colon X \to V$ a proper surjective morphism of $S$-schemes such that the homomorphism $\pi^\sharp \colon \cO_V \to \pi_\ast \cO_X$ induced by $\pi$ is an isomorphism. Consider the morphism $f:= g \circ j \circ \pi \colon X \to S$.  Then, the following conditions are equivalent:
\begin{itemize}
\item[N$_1$.] the morphism $\sigma_X$ is surjective; 
\item [N$_2$.] the image of $\sigma_X$ contains all the closed points of $\Spec \Gamma(X, \cO_X)$.
\end{itemize}
Furthermore, if one of the equivalent preceding conditions holds, then the $\Gamma(S, \cO_S)$-algebra $\Gamma(X, \cO_X)$ is finitely generated.\end{proposition}

\begin{proof} (N$_1$) $\Rightarrow$ (N$_2$) Clear. (N$_2$) $\Rightarrow$ (N$_1$) Consider the following commutative diagram of $S$-schemes:
\begin{center}
\begin{tikzcd}
X \ar[r, "\sigma_X"] \ar[d, "\pi"]& \Spec \cO(X) \ar[d, "\wr"] \\
V \ar[r, "\sigma_{V}"] & \Spec {\cO(V)}.
\end{tikzcd}
\end{center}
\Cref{Lemma:SteinFactorizationOpenImmersion} states that $\sigma_{V}$ is an open immersion. Hypothesis N$_2$ implies that the closed subset $\Spec \cO(V) \smallsetminus \sigma_{V}(V)$ does not contain any closed point, thus must be empty. It follows that $\sigma_{V}$ is an isomorphism. Since $\pi$ is surjective, one concludes that $\sigma_X$ is surjective too. It remains to show that  $\cO(S)$-algebra $\cO(X)$ is of finite type. Since $\sigma_{V}$ is an isomorphism and $V$ is of finite type over $S$, the $\cO(S)$-algebra $\cO(V)$ is of finite type. In particular, the $\cO(S)$-algebra $\cO(X)$ is of finite type.
\end{proof}

\begin{definition} Under the hypotheses of \cref{Prop:Nullstellensatz}, the morphism $f \colon X \to S$ is said to \emph{satisfy Nullstellensatz} if one of the equivalent conditions N$_{1-2}$ holds.
\end{definition}

\begin{proposition} \label{Prop:ProperOverAffine} Let $f \colon X \to S$ be a separated, finite type morphism of schemes with $S$ affine and Noetherian. Then the following are equivalent:
\begin{itemize}
\item[PA$_1$.] There is a proper morphism of $S$-schemes $\pi \colon X \to Z$ with $Z$ affine and finite type over $S$;
\item[PA$_2$.] The morphism $\sigma_X \colon X \to \Spec \Gamma(X, \cO_X)$ is proper;
\item[PA$_3$.] The morphism $f$ is proper over quasi-affine and satisfies Nullstellensatz.
\end{itemize}
Furthermore, if one of the equivalent preceding conditions holds, then the $\Gamma(S, \cO_S)$-algebra $\Gamma(X, \cO_X)$ is finitely generated. 
\end{proposition}

\begin{definition} Let $S$ be an affine Noetherian scheme. A separated morphism of finite type $f \colon X \to S$ is \emph{proper over affine} if one of the equivalent conditions PA$_{1-3}$ holds.
\end{definition}

\begin{proof} (PA$_1$) $\Rightarrow$ (PA$_2$) Consider the following commutative diagram of $S$-schemes:
\begin{center}
\begin{tikzcd}
X \ar[r, "\sigma_X"] \ar[d, "\pi"]& \Spec \cO(X) \ar[d] \\
Z \ar[r, "\sigma_Z", "\sim"'] & \Spec \cO(Z)
\end{tikzcd}
\end{center}
where the rightmost vertical arrow is induced by the homomorphism $\cO(Z) \to \cO(X)$. Since $\sigma_Z \circ \pi$ is proper, the morphism $\sigma_X$ is proper.

(PA$_2$) $\Rightarrow$ (PA$_3$) Condition PQA$_1$ is satisfied because being proper is stable under base change. Condition N$_1$ is satisfied because $X \to \Spec \cO(X)$ has a dense image; since by hypothesis $\sigma_X$ is proper, its image is also closed, whence the surjectivity.

(PA$_3$) $\Rightarrow$ (PA$_1$) By property PQA$_3$, there are an affine $S$-scheme of finite type $Y$, an open immersion  $j \colon V \to Y$ and a proper morphism of $S$-schemes $\pi \colon X \to V$ such that the homomorphism $\pi^\sharp \colon \cO_V \to \pi_\ast \cO_X$ induced by $\pi$ is an isomorphism. Consider the following commutative diagram of $S$-schemes:
\begin{center}
\begin{tikzcd}
X \ar[d, "\sigma_X"] \ar[r, "\pi"] & V  \ar[d, "\sigma_{V}"] \ar[r, "j"] & Y \ar[d, "\sigma_Y", "\wr"'] \\
\Spec \cO(X) \ar[r, "\sim"] & \Spec { \cO(V)} \ar[r] & \Spec \cO(Y)
\end{tikzcd}
\end{center}
According to \cref{Lemma:SteinFactorizationOpenImmersion}, the morphism $\sigma_V$ is an open immersion. By hypothesis $\sigma_X$ is surjective, thus $\sigma_V$ is surjective too. It follows that $\sigma_V$ is an isomorphism, therefore $V$ is affine and $X$ is proper over affine.
\end{proof}

\begin{lemma} \label{Lemma:ClosedSubsetPassingThrough} Let $X$ be an integral Noetherian scheme and $s \in \Spec \cO(X)$ non generic. Then, there is an integral closed subscheme $Y \subsetneq X$ such that $s \in \overline{\sigma_X(Y)}$.
\end{lemma}

\begin{proof} If $s = \sigma_X(x)$ for some non generic $x \in X$, then $Y = \overline{\{ x \}}$ will do. If $s$ is not in the image of $\sigma$, let $\frp \subset \cO(X)$ be the ideal corresponding to $s$ and $f \in \frp \smallsetminus\{0 \}$. The locus $Z:= V_X(f)$ of $X$ where $f$ vanishes is non-empty (otherwise $f$ would be invertible, contradicting the fact that it belongs to a prime ideal) and not the whole $X$ (because $f \neq 0$ is nonzero, so it does not vanish at the generic point of $X$).   Let $i \colon Z \to X$ be the closed immersion. The morphism $\Spec \cO(Z) \to \Spec \cO(X)$ induced by $i^\sharp \colon \cO(X) \to \cO(Z)$ factors through a morphism $h \colon \Spec \cO(Z) \to \Spec \cO(X) / f \cO(X)$. Taking global sections of the short exact sequence $0 \to f \cO_X \to \cO_X \to \cO_Z \to 0$
shows that $\cO(X) / f \cO(X) \to \cO(Z)$ is injective, thus $h$ is dominant. Let $\eta_1, \dots, \eta_r$ be the generic points of $Z$. Since $\sigma_Z$ is dominant the scheme $\Spec \cO(Z)$ is the closure of $\{ \sigma_Z(\eta_1), \dots, \sigma_Z(\eta_r)\}$. The point $s$ is contained in the closed subset $\Spec \cO(X) / f \cO(X)$ because $f$ belongs to $\frp$. Since the map $h$ is dominant, the point $s$ is in the closure of $\{ h(\sigma_Z(\eta_1)), \dots, h(\sigma_Z(\eta_r))\}$. It follows that there is $i = 1, \dots, n$ such that $s$ belongs to the closure of $h(\sigma_Z(\eta_i))$. Taking $Y$ to be the closure of $\eta_i$ does the job.
\end{proof}

\begin{proposition} \label{Prop:AffineIFFRestrictionToSubvarietySurjective} Let $X$ be a finite type and separated scheme over an affine Noetherian scheme $S$. Then the following are equivalent:
\begin{enumerate}
\item $X$ is affine;
\item $\sigma_X$ is injective and for each integral closed subscheme $Y\subset X$ not an irreducible component of $X$, the restriction map $\rho_Y \colon \Gamma(X, \cO_X) \to \Gamma(Y, \cO_Y)$ is integral.
\end{enumerate}
\end{proposition}

\begin{proof} The proof is a slight variation of the arguments in \cite[lemma C.2.11 and proposition C.2.12]{BostCharles}. (1) $\Rightarrow$ (2) Clear.  (2) $\Rightarrow$ (1). Since a scheme is affine if and only if its reduction is, we may assume $X$ reduced. Moreover, the scheme $X$ can be assumed to be integral, as a Noetherian scheme  is affine if and only if every irreducible component is. To see this, let $Y$ be the disjoint of the irreducible components of $X$; then the natural morphism $Y \to X$ is finite and surjective, hence the conclusion follows from  \cite[\href{https://stacks.math.columbia.edu/tag/01YQ}{Lemma 01YQ}]{stacks-project}. Property (2) being inherited by any closed subscheme of $X$, by Noetherian induction we may assume that every closed subscheme of $X$ different from $X$ is affine. 

\begin{claim} The scheme $X$ is proper over affine. \end{claim}

\begin{proof}[Proof of the claim.] According to \cref{Prop:ProperOverAffine}, it suffices to show that $f$ is proper over quasi-affine and it satisfies Nullstellensatz. Since $\sigma_X$ is injective, the morphism $f$ satisfies condition PQA$_1$ and it is therefore proper over quasi-affine. In order to see that $\sigma_X$ is surjective, let $s$ be a point in $\Spec \cO(X)$. Since $\sigma_X$ is dominant, the image of the generic point of $X$ is that of $\Spec \cO(X)$. Thus, if $s$ is the generic point of $\Spec \cO(X)$ we are done. If this is not the case, according to \cref{Lemma:ClosedSubsetPassingThrough}, there is an integral closed subscheme $Y \neq X$ such that $s$ belongs to the closure of $\sigma_X(Y)$. Consider the commutative diagram of $S$-schemes
\begin{center}
\begin{tikzcd}
Y \ar[r, "\sigma_Y", "\sim"'] \ar[d] & \Spec \cO(Y) \ar[d, "\Spec \rho_Y"]\\
X \ar[r, "\sigma_X"]& \Spec \cO(X),
\end{tikzcd}
\end{center}
the upper horizontal arrow being an isomorphism because $Y$ is affine by Noetherian induction. By hypothesis the restriction map $\rho_Y \colon \Gamma(X, \cO_X) \to \Gamma(Y, \cO_Y)$ is integral, thus the corresponding morphism $\Spec \rho_Y$ between spectra is closed. Since $s$ belongs to the closure of $\sigma_X(Y)$, it belongs to the image of $\Spec \rho_Y$. The morphism $\sigma_Y$ being bijective, there is $y \in Y$ such that $s = \Spec \rho_Y(\sigma_Y(y))$, thus $\sigma_X(y) =s$.
\end{proof}

By Zariski's main theorem \cite[\href{https://stacks.math.columbia.edu/tag/05K0}{lemma 05K0}]{stacks-project}, there is a finite morphism $\pi \colon Z \to \Spec \cO(X)$ and a scheme-theoretically dense open immersion $i \colon X \to Z$ such that $\sigma_X = \pi \circ i$. Consider the following commutative diagram:
\[
\begin{tikzcd}
X \ar[r, "i"]  \ar[d, "\sigma_X"] & Z \ar[r, "\pi"] \ar[d, "\sigma_Z", "\wr"'] & \Spec \cO(X) \ar[d, equal]\\
\Spec \cO(X) \ar[r, "h"] & \Spec \cO(Z) \ar[r, "\pi"] & \Spec \cO(X)
\end{tikzcd}
\]
where $h$ is the morphism induced by the ring homomorphism $i^\sharp \colon \cO(Z) \to \cO(X)$ induced by $i$. Note that the scheme $Z$ is affine because the morphism $\pi$ is finite. Since $\pi \circ h$ is the identity, the morphism $h$ is a closed immersion. Therefore $i = h \circ \sigma_X$ is proper because $X$ is proper over affine, hence an isomorphism because $i$ is a scheme-theoretically dense open immersion. Therefore the scheme $X$ is affine because $Z$ is affine. 
\end{proof}

\subsection{Stable codimension and affineness}

Let $X$ be a Noetherian scheme. The codimension $\codim_X(x)$ of a point $x \in X$ is the dimension of the local ring $\cO_{X, x}$. Let $Z$ be a closed subset of $X$. A generic point of $Z$ is a generic point of one of its irreducible components. The codimension of $Z$ in $X$ is supremum of the codimension of the generic points:
\[ \codim_X(Z) := \sup \{ \codim_X(x) : x \in Z \textup{ generic}\},\]
with the convention that it is $-\infty$ when $Z$ is empty.

\begin{definition} Let $X$ be a Noetherian scheme and $Z$ a closed subscheme of $X$.
\begin{enumerate}
\item The \emph{stable codimension} of $Z$ in $X$ is  $\codim^\s_X(Z) := \sup  \codim_{X'}(\pi^{-1}(Z))$ the supremum ranging on the morphisms $\pi \colon X' \to X$ with $X'$ Noetherian. 
\item The \emph{finite type stable codimension} is $ \codim^\sft_X(Z) := \sup  \codim_{X'}(\pi^{-1}(Z))$ the supremum ranging on the finite type morphisms $\pi \colon X' \to X$.
\end{enumerate}
\end{definition}
For affine Noetherian schemes in the literature the stable codimension (resp. the finite type stable codimension) is called the super height of the ideal defining the closed subset (with the reduced structure, say). Notwithstanding we lean for geometric-flavoured term codimension. Given a closed subset $Z$ of a Noetherian scheme $X$ and a finite open cover $X= \bigcup_{i=1}^n X_i$ we have 
\[
\codim^\star_X(Z) = \max_{i =1, \dots, n} \codim^\star_{X_i}(Z \cap X_i) , 
\quad \star \in \{ \s,\sft  \}.
\]
Moreover the stable codimension $\codim^\s_X(Z)$ of $Z$ equals $\sup \dim A$ the supremum ranging on Noetherian complete local rings $A$ with a morphism  $\pi \colon \Spec A \to X$ such that $\pi^{-1}(Z)$ is the unique closed point of $\Spec A$. Needless to say, \[\codim_X(Z) \le \codim^\sft_X(Z) \le \codim^\s_X(Z).\] 
The first inequality is in general strict, as the example of a non-irreducible scheme $X$ shows with $Z$ being one of the irreducible components. Instead, the key result is a theorem of Koh \cite[theorem 1]{Koh} saying that the second is an equality for varieties over a field:

\begin{theorem} \label{thm:Koh}If $X$ is of finite type over a field, then $\codim^\sft_X(Z) = \codim^\s_X(Z)$.
\end{theorem}

Suppose $X$ Noetherian and let $U \subset X$ be an open subset. The open immersion $i \colon U \to X$ induces a  open immersion $U_{\red} \to X':= \Spec i_\ast \cO_{U_{\red}}$ which is scheme-theoretically dense. We have $X ' \smallsetminus U_{\red} = \pi^{-1}(X \smallsetminus U)$ where $\pi \colon X' \to X$ is the natural morphism.

\begin{proposition} Suppose $X$ and $X'$ Noetherian. If $i$ is not an affine morphism, then $\codim_{X'} (X' \smallsetminus U_{\red}) \ge 2$. In particular $\codim^\s_X(X \smallsetminus U) \ge 2$.
\end{proposition}

\begin{proof}  The argument is borrowed from \cite[theorem 3.2]{Brenner}. One can clearly suppose that $X$, hence $U$, are reduced. Also, one may suppose $X$ affine. Under this hypothesis, the scheme $X'$ being Noetherian is equivalent to the ring $A:= \Gamma(U, \cO_U)$ being Noetherian. Note also that $X' \smallsetminus U$ is non-empty because $U$ is not affine.
Since $U$ is dense in $X'$, it contains all of its the generic points.  Let $\frp = (f_1, \dots, f_n)$ be a prime ideal of height $1$ of $A$. Suppose by contradiction that $\frp$ does not belong to $U$, that is, the open subset $V:= X' \smallsetminus V(\frp)$ contains $U$.  Let $f \in \frp$ be a nonzerodivisor: it exists because $\frp$ is not contained in the union of the minimal primes of $A$. The ring $A_\frp$ is local Noetherian of dimension $1$ thus $\sqrt{f A_\frp} = \frp A_\frp$. Therefore there is $N \in \bbN$ and, for each $i = 1, \dots, n$, there are $g_i \in A$, $h_i, r_i \in A \smallsetminus \frp$ such that $r_i(h_i f_i^N - g_i f) = 0$. Upon setting $h := h_1 \cdots h_n$ and $r = r_1 \cdots r_n$, one can write 
\begin{equation} \label{eq:BrennerElement}r(h f_i^N - g_i' f) = 0\end{equation}
for some $g_i' \in A$.  Since $f$ is a nonzerodivisor, the homomorphism $\cO_{X'} \to \cO_{X'}$ given by the multiplication by $f$ is injective. In particular, for each $i = 1, \dots, n$, the multiplication-by-$f$ map on $\Gamma(D(f_i), \cO_{X'})$ is injective. Coupled with the equality \eqref{eq:BrennerElement}, this shows that there is a unique $a_i \in \Gamma(D(f_i), \cO_{X'})$ such that $f a_i = rh$. By uniqueness, the functions $a_i$ glue to give a function $a \in \Gamma(V, \cO_{X'})$ such that $af = rh$. On the other hand, the restriction map $A = \Gamma(X', \cO_{X'}) \to \Gamma(V, \cO_{X'})$ is an isomorphism because the restriction map $\Gamma(V, \cO_{X'}) \to \Gamma(U, \cO_{U})$ is injective (by density of $U$ in $X'$) and the restriction map $A = \Gamma(X', \cO_{X'}) \to \Gamma(U, \cO_{U}) = A$ is the identity. This implies that the function $a$ belongs to $A$, thus $rh = af$ belongs to $\frp$, contradicting the fact that $r$ and $h$ do not belong to $\frp$.
\end{proof}

Let $K$ be a non-trivially valued complete non-Archimedean field, $X$ be a $K$-analytic space and $Z \subset X$ a closed analytic subset defined by a coherent sheaf of ideals $I$.  The codimension $\codim_X (Z)$ of $Z$ in $X$ is
\[ \codim_{X}(Z) = \sup_D \codim_{\Spec \Gamma(D, \cO_X)}(  \Spec \Gamma(D, \cO_X)  / \Gamma(D, I)), \]
the supremum ranging on affinoid domains $D$ of $X$.

\begin{definition} The \emph{stable codimension} of $Z$ is $\codim^{\s}_X(Z) = \sup \codim_{X'} (\pi^{-1}(Z))$ the supremum ranging on the morphisms $\pi \colon X' \to X$ of $K$-analytic spaces.
\end{definition}

\begin{proposition}[{\cite[theorem 4.2]{Brenner}, \cite[theorem 5.2]{Neeman}}] Let $X$ be a $K$-analytic space equipped with a morphism without boundary $X \to \cM(A)$ with $A$ an affinoid $K$-algebra. Let $Z$ be a closed analytic subset of $X$. If the open subset $X \smallsetminus Z$ is Stein, then $\codim^{\s}_X(Z) \le 1$.
\end{proposition}

\begin{proof} Suppose, by contradiction, that the stable codimension of $Z$ is $\ge 2$. Then, there is a morphism of $K$-analytic spaces $\pi \colon X' \to X$ such that the codimension of $\pi^{-1}(Z)$ in $X'$ is $\ge 2$. Since the codimension is just a local matter, one may assume $X' = \cM(B)$ for an affinoid $K$-algebra  $B$ and $\pi^{-1}(Z)$ non-empty and purely of codimension $\ge 2$. Moreover, the affinoid $K$-algebra $B$ may be assumed to be reduced, for the codimension only depends on the underlying reduced space. Let $X'_1, \dots, X'_n$ denote the normalization of the irreducible components of $X'$. The morphism  $ \bigsqcup_{i = 1}^n X_i' \to X'$ being finite, one may treat each irreducible component separately and, finally, suppose $X'$ reduced, irreducible and normal, and $\pi^{-1}(Z)$ non-empty and purely of codimension $\ge 2$. By Hartogs' phenomenon \cite{Bartenwerfer}, the restriction map $\rho \colon \Gamma(X', \cO_{X'}) \to \Gamma(X' \smallsetminus \pi^{-1}(Z), \cO_{X'})$ is an isomorphism. Because of \cref{Thm:ClosednessImage}, it is actually a homeomorphism. On the other hand, the open subset $\pi^{-1}(X \smallsetminus Z) = X' \times_X (X \smallsetminus Z)$ is a closed analytic subspace of the product $X' \times_A (X \smallsetminus Z)$. The product of Stein spaces being Stein, and closed analytic subspaces of Stein spaces being Stein, the $K$-analytic space $X ' \smallsetminus \pi^{-1}(Z)$ is Stein. Since the restriction map $\rho$ is a homeomorphic isomorphism, \cref{Prop:YonedaStein} implies that the open immersion $X' \smallsetminus \pi^{-1}(Z) \to X'$ is an isomorphism, contradicting the non-vacuity of $\pi^{-1}(Z)$.
\end{proof}

\begin{proposition} Let $Z$ be a closed subset of an $A$-scheme of finite type $X$. Then, \[\codim_X^{\sft}(Z) \le \codim_{X^\an}^{\s}(Z^\an) \le \codim_{X}^\s(Z).\]
\end{proposition}

\begin{proof} The inequality $\codim_X^{\sft}(Z) \le \codim_{X^\an}^{\s}(Z^\an)$ is clear, so that only the inequality $\codim_{X^\an}^{\s}(Z^\an) \le \codim_{X}^\s(Z)$ is left to prove. Let $\pi \colon X' \to X^\an$ be a morphism of $K$-analytic spaces. Let $I$ be the coherent sheaf of ideals on $X'$ defining the closed analytic subspace $\pi^{-1}(Z^\an)$ with its reduced structure. By definition,
\[ \codim_{X'} (\pi^{-1}(Z^\an)) = \sup_V \codim_{\Spec \Gamma(V, \cO_X)}(  \Spec \Gamma(V, \cO_X)  / \Gamma(D, I)), \]
the supremum ranging on affinoid domains $V$ of $X'$. Since affinoid $K$-algebras are Noetherian rings, the wanted inequality follows.
\end{proof}

Together with \cref{thm:Koh} the preceding result yields a non-Archimedean analogue of \cite[theorem 4.2]{Brenner}:

\begin{corollary} Let $Z$ be a closed subset of  a $K$-scheme of finite type $X$. Then \[\codim_X^{\sft}(Z) = \codim_{X^\an}^{\s}(Z^\an) = \codim_{X}^\s(Z).\]
\end{corollary}

\begin{corollary} \label{Cor:QuasiAffineSteinNoetherianIsAffine} Let $U$ be an open subset of a $K$-scheme of finite type $X$ such that $U^\an$ is Stein. Then $\codim_X^\s (X \smallsetminus U) \le 1$. Moreover, if $X$ is affine and the ring $\Gamma(U_{\red}, \cO_U)$ is Noetherian, then $U$ is affine.
\end{corollary}

\section{Stein \emph{vs.} affine varieties} Let $K$ be a complete non-Archimedean field and $A$ a strictly affinoid $K$-algebra.

\subsection{Statements} We arrive at last with the original theme of this paper. The main results of this section concern the comparison between affine and Stein varieties. More precisely, for a separated $A$-scheme of finite type $X$ we have :

\begin{theorem} \label{Thm:SteinImpliesQuasiAffine} The $A$-scheme $X$ is quasi-affine iff $X^\an$ is holomorphically separable.
\end{theorem}
 
Recall that by \cref{Cor:CharacterizationQuasiAffine} the $A$-scheme $X$ is quasi-affine as an $A$-scheme if there is an open embedding $X \into Y$ of $A$-schemes with $Y$ affine of finite type.

\begin{theorem} \label{Thm:AffineIfFiniteType} Suppose $X^\an$ is Stein and (at least) one of the following:
\begin{enumerate}
\item the $A$-algebra $\Gamma(X, \cO_X)$ is of finite type;
\item $A = K$ and the ring $\Gamma(X_{\red}, \cO_X)$ is Noetherian.
\end{enumerate}
Then  $X$ is affine.
\end{theorem}

Under hypothesis (2) \cref{Thm:AffineIfFiniteType} is the conjunction of \cref{Thm:SteinImpliesQuasiAffine} and \cref{Cor:QuasiAffineSteinNoetherianIsAffine}. It is plausible that \cref{thm:Koh}  holds for excellent rings and \cref{Cor:QuasiAffineSteinNoetherianIsAffine} holds over affinoid $K$-algebras \cite[Th\'eor\`eme 2.13]{DucrosExcellent}. However, in the case of an arbitrary affinoid $K$-algebra, we prefer not to go down this road and content ourselves with the finite generation hypothesis, sticking to techniques much closer to the ones already present in this article. As just recalled, an affinoid $K$-algebra $A$ is an excellent ring  thus is universally catenary (so that we can speak about the Krull dimension of $A$-algebras of finite type) and universally japanese \cite[\href{https://stacks.math.columbia.edu/tag/07QV}{lemma 07QV}, \href{https://stacks.math.columbia.edu/tag/0334}{proposition 0334}]{stacks-project}.

\begin{theorem} \label{Thm:RestrictionToCurvesIsSurjective}If $X^\an$ is Stein, then the restriction map $\Gamma(X, \cO_X) \to \Gamma(Y, \cO_Y)$ is surjective for all closed subscheme $Y \subset X$ dimension $\le 1$.
\end{theorem}

Combined with proposition
\ref{Prop:AffineIFFRestrictionToSubvarietySurjective} this yields:

\begin{corollary} \label{Thm:SteinIFFAffineSurfaces} If $X^\an$ is Stein and $\dim X \le 2$, then is $X$ is affine. In particular, an algebraic surface $S$ over $K$ is affine if and only if $S^\an$ is Stein.
\end{corollary}

The rest of this section is devoted to the proof of \cref{Thm:SteinImpliesQuasiAffine,Thm:AffineIfFiniteType,Thm:RestrictionToCurvesIsSurjective,Thm:SteinIFFAffineSurfaces}. But before that, a last remark on the strictness assumption for $A$. Faithful flatness of scalar extension \cite[lemma 2.1.2]{BerkovichIHES} and fpqc descent \cite[\href{https://stacks.math.columbia.edu/tag/041Z}{Lemma 041Z}]{stacks-project}, permits to prove \cref{Thm:AffineIfFiniteType} when $A$ is not necessarily strict. However, in the non strict case, the difference between Krull and analytic dimension prevents from proving \cref{Thm:RestrictionToCurvesIsSurjective,,Thm:SteinIFFAffineSurfaces} with the current techniques.

\subsection{Proof of \cref{Thm:SteinImpliesQuasiAffine}} The implication  (1) $\Rightarrow$ (2) is clear, thus only the converse one needs to be proved. According to \cref{Cor:CharacterizationQuasiAffine} it suffices to show that the morphism $\sigma_X \colon X \to \Spec \cO(X)$ is injective. 
Suppose by contradiction that $\sigma_X$ is not injective. Then trivially $\sigma_X$ is not universally injective, therefore the diagonal morphism 
\[ \Delta \colon X \too Z := X \times_{\Spec \cO(X)} X \subset X \times_S X\]
is not surjective by \cite[\href{https://stacks.math.columbia.edu/tag/01S4}{lemma 01S4}]{stacks-project}. Let $z \in Z \smallsetminus \im \Delta$ be a closed point and write $x_i = \pr_i(z) \in X$ for $i = 1, 2$. The closed points $x_1, x_2 \in X$ are distinct and they correspond to distinct rigid points of $X^\an$ by \cite[lemma 2.7.1]{DucrosFlatness}. Let $I \subset \cO_X$ be the sheaf of functions vanishing at $x_1$. Since $X^\an$ is holomorphically separated there is $f \in \Gamma(X^\an, I^\an)$ such that $f(x_2) \neq 0$, see \cite[lemma 3.2]{notions}. The subset of $ \Gamma(X^\an, I^\an)$ made of functions not vanishing at $x_2$ is open, thus there is $g \in \Gamma(X, I)$ such that $g(x_2) \neq 0$ by density of $\Gamma(X, I) \subset \Gamma(X^\an, I^\an)$. Note that  $I$ is semi-reflexive, thus \cref{Thm:DensityAlgebraicSections} can be applied. In particular $\sigma_X(x_1) \neq \sigma_X(x_2)$ thus the point $z$ cannot lie in $Z$. Contradiction. \qed

\subsection{Proof of \cref{Thm:AffineIfFiniteType}} The statement is a consequence of \cref{Thm:SteinImpliesQuasiAffine} and the maximality property of Stein spaces explained in the following two results:

\begin{lemma} \label{Lemma:NonExtensionOfAnalyticFunctions} Let $Z$ be a closed subspace of a strict, Hausdorff and countable at infinity $K$-analytic space $X$ with a morphism $X \to \cM(A)$ without boundary. If~$X \smallsetminus Z$ is Stein and scheme-theoretically dense in $X$, and $\Gamma(X, \cO_X) \to \Gamma(X \smallsetminus Z, \cO_X)$ is surjective, then $Z = \emptyset$.
\end{lemma}

\begin{proof} Suppose $Z \neq \emptyset$ by contradiction. Let  $x \in Z$ be a rigid point and $U := X \smallsetminus Z$.
Then there is a sequence of rigid points $S = \{ u_i\}_{i \in \bbN}$ whose limit is $x$. To see this, one can either argue that rigid points have a countable basis of neighborhoods or apply \cite[Th\'eor\`eme 5.3]{PoineauAngeliques} using that rigid points in $U$ are dense because $U$ is strict.
Up to extracting a subsequence, we may suppose that $S$ is discrete in $U$. For, let $\{ U_i \}_{i \in \bbN}$ be an increasing $G$-cover of $U$ by compact analytic domains with $u_0 \in U_0$. Since $z \not \in U$ up to passing to a subsequence of the points $u_i$ and the domains $U_i$ we may assume that $u_i \in U_i$ and $u_{i + 1} \not\in U_i$ for all $i\in \bbN$. Such a subsequence will do. Now $S \subset U$ is a closed analytic subset. For each $i \in \bbN$, let $\lambda_i \in K$ be such that $|\lambda_i| \ge i$.  Since $U$ is Stein, the restriction homomorphism $\cO(U) \to \cO(S)$ is surjective, thus there is $f \in \cO(U)$ such that $f(u_i) = \lambda_i$. Such a $f$ cannot extend to $X$ because $\lim_{i \to \infty} |f(u_i)| = + \infty$.
\end{proof}

\begin{lemma} \label{Prop:Maximality} Let $U$ be a scheme-theoretically dense open subset of a separated $A$-scheme $X$ of finite type. If $U^\an$ is Stein and  $\Gamma(X, \cO_X) \to \Gamma(U, \cO_X)$ is surjective, then $U = X$. 
\end{lemma}

\begin{proof} The open subset $U^\an \subset X^\an$ is scheme-theoretically dense by \cref{Cor:InjectivityPushForward}. Therefore the image of the restriction map $\cO(X^\an) \to \cO(U^\an)$ is closed by \cref{Thm:ClosednessImage} and $\cO(U) \subset \cO(U^\an)$ is dense by \cref{Thm:DensityAlgebraicSections}.  The hypothesis then implies that $\cO(X^\an) \to \cO(U^\an)$ is surjective. Since $U^\an$ is Stein, this implies that $X^\an \smallsetminus U^\an$ is empty by \cref{Lemma:NonExtensionOfAnalyticFunctions}.
\end{proof}

\begin{proof}[{Proof of \cref{Thm:AffineIfFiniteType}}] Assuming (2) \cref{Thm:AffineIfFiniteType} follows from \cref{Thm:SteinImpliesQuasiAffine} and \cref{Cor:QuasiAffineSteinNoetherianIsAffine}. Suppose that $\cO(X)$ is a finitely generated $A$-algebra. The $A$-scheme $X$ is quasi-affine by \cref{Thm:SteinImpliesQuasiAffine} thus  $\sigma_X \colon X \to Y:= \Spec \cO(X)$ is a scheme-theoretically dense open immersion (\cref{Lemma:SteinFactorizationOpenImmersion,,Lemma:SteinFactorizationDominant}). The restriction map~$\cO(Y) \to \cO(X)$ is an isomorphism by construction. Since $X^\an$ is Stein we have $X = Y$ by \cref{Prop:Maximality}.
\end{proof}

\subsection{Proof of \cref{Thm:RestrictionToCurvesIsSurjective}} 

First we need this simple consequence of the Stein property and \cref{Thm:DensityAlgebraicSections}:

\begin{lemma} \label{Prop:DensityOnClosedSubvarieties} Let $Y$ be a closed subscheme of a separated $A$-scheme $X$ of finite type and $F$ a semi-reflexive coherent $\cO_X$-module. If $X^\an$ is Stein, then \[
\im(\Gamma(X, F) \to \Gamma(Y, F)) \subset \Gamma(Y^\an, F^\an)\]
is dense.
\end{lemma}

\begin{proof} Let $f \in \Gamma(Y^\an, F^\an)$. The statement boils down to approximating $f$ by elements of $\Gamma(X, F)$ on any compact analytic domain of $Y^\an$. Since $X^\an$ is Stein and $Y^\an$ is a closed subspace, there is $g \in \Gamma(X^\an, F^\an)$ such that $g_{\rvert Y^\an} = f$. Moreover, it suffices to approximate $g$ by elements of $\Gamma(X, F)$ on any compact analytic domain $D$ of $X^\an$. To do so, let $\| \cdot \|$ be a norm defining the topology of $F(D)$ and $\epsilon > 0$. Then by density of $F(X)$ in $F(X^\an)$ there is $h \in F(X)$ such that $\| g - h\| < \epsilon$.
\end{proof}

The following simple fact about curves hides the fact that closed points in a Dedekind scheme are Cartier divisors:

\begin{lemma} \label{lemma:NoBlowUpInCurves} Let $U$ be an non-empty open subset in an affine Noetherian integral scheme $X$ of dimension $1$. Let $S \subset \Gamma(U, \cO_U)$ a subring containing $\Gamma(X, \cO_X)$. If $X$ is integrally closed in $U$, then the morphism $i \colon \Spec S \to X$ is an open immersion.
\end{lemma}

\begin{proof} The closed subset $X \smallsetminus U$ is made of finitely many closed points because $\dim X = 1$. It follows that the image of $i$ is an open subset $V \subset X$.
\begin{claim} For all $y \in Y = \Spec S$ we have $\cO_{Y,y}= \cO_{X,i(y)}$. \end{claim}

\begin{proof}[Proof of the Claim] Notice first that the induced morphism $j \colon U \to Y := \Spec S$ is an open immersion. Therefore it suffices to prove the claim only for $y \in Y \smallsetminus j(U)$. Now $X$ is integrally closed in $U$, thus the local ring $\cO_{X,x}$ is normal  hence a discrete valuation ring for all $x \in X \smallsetminus U$. Set $x= i(y)$ and let $\pi \in \cO_{X, x}$ be a uniformizer. Every $\cO_{X, x}$-submodule of $F = \Frac(\cO_{X, x})$ is of the form $\pi^n \cO_{X, x}$ for some $n \in \bbZ$. Since $\cO_{Y, y} \subset F$ is a ring, there are only two possibilities: either $\cO_{X, x} = \cO_{Y, y}$ or $\cO_{Y, y} = F$. The second one cannot happen: otherwise $y$ would be the generic point. 
\end{proof}

Now the morphism $i$ is injective on points. Indeed if $y, y ' \in Y$ have image $x$ under $i$, then we have $\cO_{Y, y'} = \cO_{X, x} = \cO_{Y, y} $ by the claim. Let $\frp, \frp' \subset S$ be the prime ideals corresponding to $y, y'$ respectively. Then $\frp = \frp' = \frm \cap S$ where  $\frm\subset \cO_{X,x}$ is the maximal ideal, thus $y = y'$. On the other hand, by the claim we have
\[ \Gamma(V, \cO_X) = \bigcap_{x \in V} \cO_{X, x} = \bigcap_{y \in Y} \cO_{Y, y} = S. \]
This concludes the proof.
\end{proof}

In the proofs that follow we will use repeatedly the following:

\begin{remark} \label{rmk:OpenEmbedding} Let $f \colon X \to Y$ and $g \colon Y \to Z$ be morphisms of schemes where $f$ is quasi-compact and has a scheme-theoretically dense image, and $g \circ f$ is an immersion. Then, $f$ is an immersion by \cite[\href{https://stacks.math.columbia.edu/tag/07RK}{Lemma 07RK}]{stacks-project} which is actually an open immersion because $f$ has scheme-theoretically dense image \cite[\href{https://stacks.math.columbia.edu/tag/01QV}{Lemma 01QV}]{stacks-project}. In practice we will use this when $X$ is a Noetherian scheme, hence the quasi-compactness of $f$ will be automatic. 
\end{remark}

The key to reach \cref{Thm:RestrictionToCurvesIsSurjective} is that on curves there are no proper dense subrings. More precisely:

\begin{lemma} \label{Lem:DenseSubalgebrasCurves} Let $X$ be an affine $A$-scheme of finite type of dimension $\le 1$ and $S \subset \cO(X)$ an $A$-subalgebra. Suppose that $S \subset \cO(X^\an)$ is dense and that there is a finitely generated $A$-subalgebra $T \subset S$ such that the induced morphism $X \to \Spec T$ is an open immersion. Then  $S = \cO(X)$.
\end{lemma}

\begin{proof} 
\emph{Step 1.} Suppose $X$ integral. If $X$ is a singleton, then $\cO(X)$ is a finite extension of $K$ and $S = \cO(X^\an)$ by density. Suppose $\dim X = 1$. By hypothesis there is a finitely generated $A$-subalgebra $T \subset S$ such that the morphism $j \colon X \to Y := \Spec T$ is a scheme-theoretically dense open immersion. The scheme $X$ is excellent of dimension $1$, thus so is $Y$. Let $\tilde{Y} = \Spec \tilde{T}$ be the normalization of~$Y$ in $X$ \cite[\href{https://stacks.math.columbia.edu/tag/0BAK}{Section 0BAK}]{stacks-project} and  $\tilde{S} = \im(S \otimes_{T} \tilde{T} \to \cO(X))$.  The morphism $\Spec \tilde{S} \to \tilde{Y}$ induced by the inclusion $\tilde{T} \subset \tilde{S}$ is an open immersion by \cref{lemma:NoBlowUpInCurves}. Now $\tilde{S} \subset \cO(X^\an)$ is dense because it contains $S$ and is finite over $S$ because $\tilde{T}$ is a finitely generated $T$-module. \Cref{Cor:SectionsAffineOpen} implies $\tilde{S} = \cO(X)$ thus $\cO(X)$ is a finitely generated $S$-module. Now the morphism $X \to \Spec S$ is at the same time an open immersion (the composite morphism $X \to \Spec S \to \Spec T$ is so thus we can apply \cref{rmk:OpenEmbedding}), is finite and has a connected target. Thus it is an isomorphism.

\medskip

\emph{Step 2.} Suppose $X$ reduced. It suffices to prove that each irreducible component of $X$ in closed in $\Spec S$. If we do so, then the image of $j \colon X \to \Spec S$ is open and closed. Since $X$ is scheme-theoretically dense in $\Spec S$, it must be the whole $\Spec S$ concluding the proof when $X$ is reduced. Now an irreducible component $X'$ of $X$ endowed with its reduced structure is affine because $X$ is affine and $X'$ is closed in $X$. Since $X'$ is affine, in order to prove the statement it suffices to show that the restriction homomorphism $S \to \cO(X')$ is surjective. Let $S'$ be the image of $S$ in $\cO(X')$. Since $X$ is affine, the $K$-analytic space $X^\an$ is Stein. Therefore the restriction homomorphism $\cO(X^\an) \to \cO(X'^\an)$ is surjective. By hypothesis $S \subset \cO(X^\an)$ is dense, thus the image $S'$ of $S$ in $\cO(X')$ is dense in $\cO(X'^\an)$ thanks to \cref{Prop:DensityOnClosedSubvarieties}. The image $T'$ of $T$ in $S'$ is such that $X' \to \Spec T'$ is an open embedding, because the composite morphism $X \to \Spec T' \to \Spec T$ is an immersion and we can apply \cref{rmk:OpenEmbedding}. Step 1 implies $S' = \cO(X')$ because $X'$ is irreducible, thus $X'$ is closed in $\Spec S$. 

\medskip

\emph{Step 3.} In the general case let $N \subset \cO(X)$ be the nilpotent radical. Let $r \in \bbN$ be index of nilpotency of $X$, that is, the unique integer such that $N^r \neq 0$ and $N^{r+1} = 0$. Let us prove $\cO(X) = S$ by induction on $r$. If $r = 0$, then $X$ is reduced and the statement follows from the preceding case. Suppose $r \ge 2$ and the statement true for $A$-schemes of finite type of dimension $\le 1$ whose index of nilpotency is $\le r-1$. Let $X'$ be the closed subscheme of $X$ defined by the ideal $N^{r}$. Let $S'$ be the image of $S$ in $\cO(X')$. The following diagram is exact and commutative
\begin{center}
\begin{tikzcd}
0 \ar[r] & S \cap N^{r} \ar[r] \ar[d] & S \ar[r] \ar[d] & S ' \ar[r] \ar[d] & 0 \\
0 \ar[r] & N^{r} \ar[r] & \cO(X) \ar[r] & \cO(X') \ar[r] & 0
\end{tikzcd}
\end{center}
Since $\cO(X^\an) \to \cO(X'^\an)$ is surjective, $S'$ is dense in $\cO(X'^\an)$. Moreover, as above, the image $T'$ of $T$ in $S'$ is such that $X' \to \Spec T'$ is an open embedding.  The index of nilpotency of $X'$ is $\le r - 1$, hence $S' = \cO(X')$ by inductive hypothesis. On the other hand, $N^r$ has a natural structure of $\cO(X)/N$-module, and it is finitely generated as such. In particular, since $S' = \cO(X')$, $N^r$ is a finitely generated $S'$-module. The lower line in the previous commutative and exact diagram, implies that $\cO(X)$ is a finite $S$-module. Therefore, the morphism $X \to \Spec S$ is both a scheme-theoretically dense open immersion and finite, thus an isomorphism.
\end{proof}

\begin{lemma} \label{Lemma:NonProperAffinoidCurvesAreAffine} Let $X$ be an $A$-scheme of finite type of dimension $\le 1$. Then, the following are equivalent:
\begin{enumerate}
\item  the scheme $X$ is affine;
\item for $s \in \Spec A$, the $\kappa(s)$-scheme of finite type $X_s$ does not contain positive-dimensional closed subschemes that are proper over $\kappa(s)$.
\end{enumerate}
\end{lemma}

\begin{proof} Let $\pi \colon X \to \Spec A$ be the structural morphism.

 (1) $\Rightarrow$ (2) If $X$ is affine, then all the fibers of $\pi$ are affine. The statement follows because proper affine schemes are finite.

(2) $\Rightarrow$ (1) One reduces to the situation where the affinoid $K$-algebra $A$ is integral, normal  and of dimension $\le 1$, the scheme $X$ is integral of dimension $1$ and the morphism $\pi$ is dominant. This is because of three facts. First, a point is affine, which covers the case of $0$-dimensional $X$. Second, a Noetherian scheme is affine if and only if each of its irreducible components (endowed with its reduced structure) is: this permits to reduce to the case where $A$ and $X$ are integral. Third, the normalization of $\Spec A$ is a finite morphism because $A$ is a Japanese ring. The second and third reduction steps use \cite[\href{https://stacks.math.columbia.edu/tag/01YQ}{Lemma 01YQ}]{stacks-project}. Under the above additional hypotheses, one distinguishes two cases. If $\dim A = 0$, then the affinoid $K$-algebra $A$ is a field and the statement is equivalent to the well-known fact that curves over a field with no projective component are affine. Suppose $\dim A = 1$. The ring $A$ is normal, Noetherian and of dimension $1$, thus a Dedekind ring \cite[\href{https://stacks.math.columbia.edu/tag/034X}{lemma 034X}]{stacks-project}. Being dominant, the morphism $\pi$ is flat. In particular, the fibers of $\pi$ are all $0$-dimensional, that is, the morphism $\pi$ is quasi-finite. By Zariski's main theorem \cite[\href{https://stacks.math.columbia.edu/tag/05K0}{lemma 05K0}]{stacks-project}, there is a finite morphism $\nu \colon \Spec B \to \Spec A$ and an open immersion $i \colon X \to \Spec B$ such that $\pi = \nu \circ i$. Up to replacing $\Spec B$ by the scheme-theoretic closure of the image of $i$, one may assume $B$ to be integral. Therefore, the affinoid $K$-algebra $B$ is of dimension $1$. In order to show that $X$ is affine, we may suppose that $B$ is normal because the normalization is a finite morphism ($B$ is a japanese ring); see \cite[\href{https://stacks.math.columbia.edu/tag/01YQ}{Lemma 01YQ}]{stacks-project}. One concludes because every open subset in an affine Dedekind scheme is affine.
\end{proof}

\begin{lemma} \label{Lemma:SteinsDontContainProper} Let $X$ be an $A$-scheme of finite type such that $X^\an$ is Stein. Then, for any $s \in \Spec A$, the $\kappa(s)$-scheme of finite type $X_s$  does not contain positive-dimensional closed subschemes that are proper over $\kappa(s)$.
\end{lemma}

\begin{proof} Let $S:= \Spec A$ so that $S^\an = \cM(A)$. Let $s \in S$. According to \cite[proposition 2.1.1]{BerkovichIHES}, there is a point $s^\an \in S^\an$ whose image in $S$ via the canonical morphism $S^\an \to S$ is $s$. In order to prove the statement, up to extending scalars to the complete residue field $K'$ at $s^\an$, one may assume that $s$ is $K$-rational. Indeed, consider the affinoid $K'$-algebra $A' := A \hotimes_{K} K'$, the scheme $S' = \Spec A'$ and the $S'$-scheme of finite type $X' = X \times_S S'$. Let $\pi' \colon X' \to S'$ be the morphism deduced from $\pi$ by base-change. The $K'$-point $s^\an$ of $S^\an$ defines a $K'$-point $s'^\an$ of $S'^\an$, which corresponds with a closed point $s'$ of $S'$. Let $F = S' \times_S \Spec K'$ be the fiber of the morphism $S' \to S$ induced by the inclusion $A \to A'$ at the $K'$-point of $S$ defined by $s^\an$. The point $s'$ defines a $K'$-point of the $K'$-scheme $F$ and induces a section of the projection $X'_F := X' \times_{S'} F \to X_s \times_{\kappa(s)} K'$. The situation is resumed in the following diagram:

\begin{center}
\begin{tikzpicture}[on top/.style={preaction={draw=white,-,line width=#1}}, on top/.default=4pt, scale=.9]
\def\a{1.5};
\def\b{1};

\node (A1) at (0,0) {$\Spec K'$};
\node (A2) at (2*\a,0) {$\Spec \kappa(s)$};
\node (A3) at (4*\a, 0) {$S$}; 

\node (B1) at (\a,\b) {$X_s \times_{\kappa(s)} K' $};
\node (B2) at (3*\a,\b) {$X_s$};
\node (B3) at (5*\a,\b) {$X$};

\node (C1) at (0,2*\b) {$F$};
\node (C3) at (4*\a,2*\b) {$S'$};

\node (D1) at (\a,3*\b) {$X'_F$};
\node (D3) at (5*\a,3*\b) {$X'$};

\node at (-\a/2.5,\b) {$_{s'}$};

\draw[->] (C1) -- (A1);

\draw[->] (A1) -- (A2);
\draw[->] (A2) -- (A3);

\draw[->] (B1) -- (A1);
\draw[->] (B2) -- (A2);
\draw[->] (B3) -- (A3);
\draw[->] (D3) -- (B3);
\draw[->] (D1) -- (D3);
\draw[->] (D1) -- (C1);
\draw[->] (D1) -- (B1);
\draw[->] (D3) -- (C3);
\draw[->] (B1) -- (B2);
\draw[->] (B2) -- (B3);

\draw[->] (A1) edge[bend left] node [left] {} (C1);
\draw[->] (B1) edge[bend left] node [left] {} (D1);

\draw[->, on top] (C1) -- (C3);
\draw[->, on top] (C3) -- (A3);

\end{tikzpicture}
\end{center}
Via the closed immersion $X_s \times_{\kappa(s)} K' \to X'_F$, the $K'$-scheme $X'_{s'}$ is identified with $X_s \times_{\kappa(s)} K'$. The $K'$-analytic space $X'^\an$ associated with $X'$ is the $K'$-analytic space obtained from $X^\an$ by extending scalars to $K'$. The $K'$-analytic space $X'^\an$ is Stein because the property of being Stein is compatible with scalar extensions. Since being proper is compatible with base-change, it suffices to show that the $K'$-scheme $X_s \times_{\kappa(s)} K'$ does not contain positive-dimensional closed subschemes that are proper over $K'$. By identifying  $X_s \times_{\kappa(s)} K'$ with $X'_{s'}$, one reduces to the case where $s$ is $K$-rational. Henceforth suppose that $s$ is $K$-rational. Let $Y$ be a closed subscheme of $X_s$ which is proper over $K$. Therefore the $K$-algebra $\Gamma(Y, \cO_Y)$ is finite and, by GAGA over a non-Archimedean field \cite[Corollary 3.4.10]{Berkovich90}, coincides with the $K$-algebra $\Gamma(Y^\an, \cO_{Y}^\an)$ of analytic functions on $Y^\an$. Now, since the point $s$ is $K$-rational, the subscheme $X_s$ is closed in $X$, thus $Y$ is a closed subscheme of $X$. It follows that the $K$-analytic space $Y^\an$ is Stein because it is a closed subspace of the Stein space $X^\an$. By compactness of $Y^\an$, an affinoid Stein exhaustion of $Y^\an$ is eventually stationary, showing that the $K$-analytic space $Y^\an$ is affinoid. Moreover, the $K$-algebra $\Gamma(Y^\an, \cO_Y^\an)$ being finite, the $K$-analytic space $Y^\an$ is finite over $K$. Therefore the $K$-scheme $Y$ is finite \cite[proposition 3.4.7 (4)]{Berkovich90}.\end{proof}

\begin{proof}[{Proof of \cref{Thm:RestrictionToCurvesIsSurjective}}] The $K$-analytic space $X^\an$ is Stein, thus holomorphically separable. It follows from \cref{Cor:CharacterizationQuasiAffine,,Thm:SteinImpliesQuasiAffine} that there is a finitely generated $A$-algebra $B \subset \cO(X)$ such that $X \to \Spec B$ is an open embedding. The $K$-analytic space $Y^\an$ is a closed subspace of the Stein analytic space $X^\an$ thus Stein. Thus the scheme $Y$ is affine by \cref{Lemma:NonProperAffinoidCurvesAreAffine,,Lemma:SteinsDontContainProper}. According to \cref{Prop:DensityOnClosedSubvarieties} the $A$-algebra $S = \im(\cO(X) \to \cO(Y))$ is dense in $\cO(Y^\an)$. Moreover, the image $T$ of $B$ in $S$ is such that $Y \to \Spec T$ is an open embedding because the composite morphism $Y \to \Spec T \to \Spec B$ is an immersion, so \cref{rmk:OpenEmbedding} gives the claim. Therefore we can apply \cref{Lem:DenseSubalgebrasCurves} and obtain $S = \cO(Y)$.
\end{proof}

\small

\bibliography{./../../biblio}

\bibliographystyle{amsalpha}

\end{document}